\documentclass[12pt,a4paper]{article}
\usepackage{graphicx,epsfig,color}
\usepackage{latexsym}
\usepackage{dsfont}
\usepackage{amsmath,amssymb,amstext}
\usepackage{bm,a4wide}
\usepackage[english]{babel}
\usepackage{psfrag}
\usepackage{subfigure}
\usepackage{placeins}
\usepackage{multirow}

\usepackage[
          colorlinks=true,linkcolor=blue
          ,pdfproducer={Latex with hyperref}
          ,pdfcreator={pdflatex}
          ]{hyperref}

\newtheorem{theorem}{Theorem}

\newtheorem{corollary}[theorem]{Corollary}

\newtheorem{lemma}[theorem]{Lemma}

\newtheorem{remark}[theorem]{Remark}

\newenvironment{proof}[1][Proof]{\noindent\textbf{#1.} }{\ \rule{0.5em}{0.5em}}

\def\NN{\mathbb{N}}

\def\paral{\|}
\def\be{\begin{equation}}
\def\ee{\end{equation}}

\let\ds\displaystyle
\let\eps\varepsilon

\begin{document}

\title{Numerical analysis of an asymptotic-preserving scheme for anisotropic
elliptic equations}
\author{Alexei Lozinski\footnotemark[4] \and Jacek Narski\footnotemark[2]
\and Claudia Negulescu\footnotemark[2] }
\maketitle

%\author{Alexei Lozinski\footnotemark[4] \and  \and Claudia Negulescu\footnotemark[2] }

\renewcommand{\thefootnote}{\fnsymbol{footnote}}

\footnotetext[2]{%
Institut de Math\'ematiques de Toulouse, UMR 5219 , Universit\'e de
Toulouse, CNRS, UPS/IMT, 118 route de Narbonne, F-31062 Toulouse, France} 
\footnotetext[4]{%
Universit\'e de Franche-Comt\'e, Laboratoire de Math\'ematiques, 16 route de
Gray, 25030 Besan\c{c}on, France} 
%({\tt claudia.negulescu@cmi.univ-mrs.fr})}

\renewcommand{\thefootnote}{\arabic{footnote}}

%\tableofcontents

\abstract{The main purpose of the present paper is to study from a
  numerical analysis point of view some robust methods designed to
  cope with stiff (highly anisotropic) elliptic problems. The
  so-called asymptotic-preserving schemes studied in this paper are
  very efficient in dealing with a wide range of $\eps$-values, where
  $0 < \eps \ll 1$ is the stiffness parameter, responsible for the
  high anisotropy of the problem. In particular, these schemes are
  even able to capture the macroscopic properties of the system, as
  $\eps$ tends towards zero, while the discretization parameters
  remain fixed. The objective of this work shall be to prove some
  $\eps$-independent convergence results for these numerical schemes
  and put hence some more rigor in the construction of such
  AP-methods.}
  
\medskip


%%%%%%%%%%%%%%%%%%%%%%%
\section{Introduction}\label{SEC1} 
%%%%%%%%%%%%%%%%%%%%%%%
In a series of previous works \cite{DDLNN,DDN,DLNN,LNN,NarskiOttaviani} some efficient numerical
schemes were introduced in the aim to solve at a moderate computational cost
some highly anisotropic elliptic and parabolic problems. The interest in solving such
problems comes for example from their regular occurrence in the modeling of
magnetically confined plasmas \cite{Chen,Hazel} and ionospheric plasmas  \cite{Nagy}, where the strong magnetic field creates
anisotropy. An accurate and not resource demanding description of tokamak plasma
dynamics is crucial for succeeding in the construction of a thermonuclear
fusion reactor, producing clean energy for the future. 

The problems cited above involve a small 
parameter $0 < \eps \ll 1$ measuring the anisotropy ratio in the diffusion matrix. This feature makes their numerical treatment rather involved, since the problems degenerate 
in the limit $\eps \rightarrow 0$ leading to a
break-down of traditional schemes for $\eps$ very small. This is caused both by the huge, $\eps$-dependent condition number of the discretized problem and by locking phenomena (the strong diffusion along a magnetic field line
makes the solution to be almost constant along these lines, which is incompatible with an approximation by piecewise polynomials unless the computational mesh is well aligned with the field). In the
previous works some efficient so-called \textit{asymptotic-preserving}
schemes were proposed and were shown to be able to cope with the
deficiencies of traditional schemes. Their basic idea is to mimic on the
discrete level the asymptotic behaviour of the continuous solution $u^\eps$  in the limit $\eps\to 0$, thus making the diagram in Fig.~\ref{diag1} commutative.
\begin{figure}[!ht]
\centering
\psfrag{T1}[][][1.]{$P^{\eps,h}$} \psfrag{T2}[][][1.]{$P^{\eps}$} %
\psfrag{T4}[][][1.]{$P^{0,h}$} \psfrag{T3}[][][1.]{$P^{0}$} %
\psfrag{L1}[][][1.]{$\eps\rightarrow 0$} 
\psfrag{H1}[][][1.]{$h\rightarrow 0$} 
\includegraphics[width=0.4\textwidth]{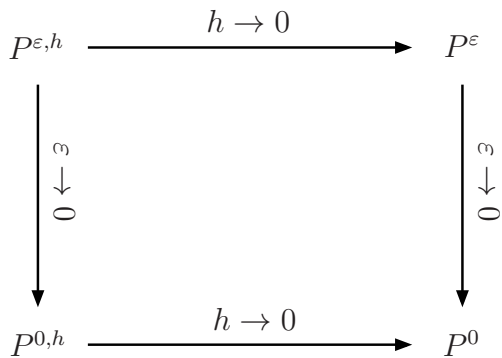}
\caption{Properties of AP-schemes}
\label{diag1}
\end{figure} 

In the present paper, we restrict out attention to the case of elliptic linear problems and are interested in two Asymptotic-Preserving schemes 
proposed in \cite{DLNN} and \cite{NarskiOttaviani}. The efficiency and advantages of the different schemes was put into evidence numerically.
However, the rigorous numerical analysis of these schemes is still lacking
and is the subject of the present paper. The trick that makes these schemes work is the introduction of an auxiliary variable $q^\eps$ which serves as a Lagrange multiplier in the limit $\eps\to 0$ 
corresponding to the constraint on $u^\eps$, which results from the degeneracy of the governing equations. We are thus in the realm of mixed problems, their penalized variants and the discretizations thereof, as in \cite{Brezzi,Girault}. 
One cannot adapt though directly the techniques from these books to the present case as the inf-sup conditions are not satisfied on the discrete level, when one discretizes with standard finite elements as in the above cited 
papers. In fact, the choice of appropriate functional spaces even on the continuous level is not straightforward. We are going here to propose an adequate functional setting with the inf-sup condition being introduced in 
a non standard way. We shall develop then a complete analysis of the finite element schemes with $\eps$-independent constants in  the error estimates, relying on some discrete inf-sup conditions in $h$-dependent norms.

This paper is organized as follows. In Section \ref{SEC2}, we
introduce the anisotropic elliptic problem, which is the starting
point of this work. A first Asymptotic-Preserving scheme for this
problem, slightly modifying that proposed in  \cite{DLNN} and well adapted for open field-line configurations, is then
presented and analyzed in detail. Section \ref{secStab} is concerned
with the introduction of a different Asymptotic-Preserving
reformulation of the same anisotropic elliptic problem, being able to
cope even with closed field-line configurations, which are often
encountered in tokamak plasma modeling. This leads to the scheme proposed in \cite{NarskiOttaviani}.  A detailed numerical analysis
of this scheme is then carried on. In Section \ref{sec:num_tests} we validate numerically the error estimates obtained. Finally, some technical lemmas are
postponed to the Appendices A and B.

%%%%%%%%%%%%%%%%%%%%%%%

\section{An AP-scheme for open field-line configurations}

\label{SEC2} %%%%%%%%%%%%%%%%%%%%%%%
Before presenting our model problem, let us first define some important
quantities. Let $b$ be a smooth field in a domain $\Omega \subset \mathbb{R}%
^d$, with $d=2,3$, and let us decompose the regular boundary $\Gamma
=\partial \Omega $ into three components following the sign of the
intersection with $b$: 
\begin{equation*}
\Gamma _{D}:= \{ x \in \Gamma \,\, / \,\, b(x) \cdot n(x) =0 \}\,, \quad
\Gamma_N:=\Gamma _{in} \cup \Gamma _{out}= \{ x \in \Gamma \,\, / \,\, b(x) \cdot n(x) \lessgtr 0  \}\, .
\end{equation*}
The vector $n$ is here the unit outward normal to $\Gamma $.

The direction of the anisotropy of our problem is defined by this
vector field $b \in (C^{\infty}(\Omega))^d$, which is supposed to
satisfy $|b(x)|=1$ for all $x \in \Omega$. Given this vector field
$b$, one can decompose now vectors $v \in \mathbb{R}^d$, gradients
$\nabla \phi$, with $\phi(x)$ a scalar function, and divergences
$\nabla \cdot v$, with $v(x)$ a vector field, into a part parallel to
the anisotropy direction and a part perpendicular to it. These parts
are defined as follows:
\begin{equation}
\begin{array}{llll}
\ds v_{\|}:= (v \cdot b)\, b \,, & \ds v_{\perp}:= (Id- b \otimes b) v\,, & 
\text{such that} & \ds v=v_{\|}+v_{\perp}\,, \\[3mm] 
\ds \nabla_{\|} \phi:= (b \cdot \nabla \phi)\, b \,, & \ds \nabla_{\perp}
\phi:= (Id- b \otimes b)\, \nabla \phi\,, & \text{such that} & \ds \nabla
\phi=\nabla_{\|}\phi+\nabla_{\perp}\phi\,, \\[3mm] 
\ds \nabla_{\|} \cdot v:= \nabla \cdot v_{\|} \,, & \ds \nabla_{\perp} \cdot
v:= \nabla \cdot v_{\perp}\,, & \text{such that} & \ds \nabla \cdot
v=\nabla_{\|}\cdot v+\nabla_{\perp}\cdot v\,.%
\end{array}%
\end{equation}
Given these notations we can now introduce the highly anisotropic elliptic
problem we are interested in, namely
\begin{gather}
\left\{ 
\begin{array}{ll}
-{\frac{1 }{\eps}} \nabla_\parallel \cdot \left(A_\parallel \nabla_\parallel
u^{\eps }\right) - \nabla_\perp \cdot \left(A_\perp \nabla_\perp u^{\eps %
}\right) = f & \text{ in } \Omega, \\[3mm] 
{\frac{1 }{\eps}} n_\parallel \cdot \left( A_\parallel \nabla_\parallel u^{%
\eps } \right) + n_\perp \cdot \left(A_\perp \nabla_\perp u^{\eps }\right) =
0 & \text{ on } \Gamma_N, \\[3mm] 
u^{\eps }= 0 & \text{ on } \Gamma _{D}\,.%
\end{array}
\right.  \label{P}
\end{gather}
The parameter $0<\eps \ll1$ is very small, inducing rather sever numerical
difficulties, when solving (\ref{P}) via standard methods. Indeed,
this elliptic system becomes degenerate in the limit $\eps \rightarrow 0$,
leading to the reduced problem 
\begin{gather}
(R)\,\,\, \left\{ 
\begin{array}{ll}
-\nabla_\parallel \cdot \left(A_\parallel \nabla_\parallel u \right) = 0 & 
\text{ in } \Omega, \\[3mm] 
n_\parallel \cdot \left( A_\parallel \nabla_\parallel u \right) = 0 & \text{
on } \Gamma_N, \\[3mm] 
u= 0 & \text{ on } \Gamma _{D}\,,%
\end{array}
\right.  \label{R}
\end{gather}
which has an infinite amount of solutions, all of them being constant along
the field lines. Numerically this degeneracy translates in a very
ill-conditioned linear system to be solved when $0 < \eps \ll 1$.\\
The aim of the present section will be the mathematical study of the elliptic problem (\ref{P}), in
particular the investigation of its asymptotic behaviour as $\eps$ tends towards zero, the introduction of an \textit{Asymptotic-Preserving} reformulation, better suited to pass to the limit $\eps \rightarrow 0$, and the detailed numerical analysis of the designed AP-scheme.
The reformulation of the singularly-perturbed problem (\ref{P}) is based on asymptotic arguments and is a sort of ``reorganization'' of the problem into a form, which allows for an automatic numerical transition from (\ref{P}) towards the limit-model (to be determined) as $\eps \rightarrow 0$, while keeping the discretization parameters fixed.

%%%%%%%%%%%%%%%%%%%%%%%
\subsection{Inflow Asymptotic-Preserving reformulation}
\label{SEC21} 
%%%%%%%%%%%%%%%%%%%%%%%

In order to avoid all the above mentioned difficulties corresponding to the non-uniqueness
of the reduced problem $(R)$, one has to pick up within all its solutions
the right limit solution, by fixing in an adequate manner  its value on the field lines. This was done in the previous works \cite{DDLNN,DDN,DLNN}, via the introduction of Lagrange multipliers, which are necessary
to recover the uniqueness in the limit $\eps \rightarrow 0$. The numerical resolution of the thus obtained Asymptotic-Preserving reformulations was shown to be stable and
accurate independently on the parameter $\eps$, which is a great advantage
as compared to standard discretizations for (\ref{P}). This essential
property of the designed AP-scheme was proved numerically,
its rigorous numerical analysis being the subject of the present paper.%
\newline

For the mathematical study, let us assume that the diffusion coefficients
and the source term satisfy the following hypothesis:\\

\noindent \textbf{Hypothesis A} \label{hypo} \textit{\ Let $f\in
H^{-1}(\Omega )$, $0<\eps<1$ be a fixed arbitrary parameter and $\overset{%
\circ }{\Gamma _{D}}\neq \varnothing $. The diffusion coefficients $%
A_{\|}\in W^{2,\infty }(\Omega )$ and $A_{\perp }\in \mathbb{M}_{d\times
d}(W^{2,\infty }(\Omega ))$ are supposed to verify the bounds 
\begin{gather}
0<A_{0}\leq A_{\|}(x)\leq A_{1}\,,\quad \text{for a.a.}\,\,\,x\in \Omega ,
\label{hyp1} \\[3mm]
A_{0}\|v\|^{2}\leq v^{t}A_{\perp }(x)v\leq A_{1}\|v\|^{2}\,,\quad \forall
v\in \mathbb{R}^{d}\,\,\,\text{with}\,\,\,v\cdot b(x)=0\,\,\,\text{for a.a.}%
\,\,\,x\in \Omega ,  \label{hyp2}
\end{gather}%
with some constants $0<A_{0}\leq A_{1}$. }\newline

Before we shall pass to a brief presentation of an AP-reformulation of (\ref%
{P}), we shall rewrite this problem in a slightly different form, masking the
perpendicular derivatives, which turn out to be cumbersome
for the numerical analysis. Indeed, the following reformulation, called in
the following $(P)^\eps$-problem 
\begin{equation}
(P)^\eps \,\,\, \left\{ 
\begin{array}{ll}
-{\frac{1-\eps }{\eps}}\nabla _{\parallel }\cdot \left( A_{\parallel }\nabla
_{\parallel }\,u^{\eps}\right) -\nabla \cdot \left( A\nabla u^{\eps}\right) =f
& \text{ in }\Omega , \\[3mm] 
{\frac{1-\eps }{\eps}}n_{\parallel }\cdot \left( A_{\parallel }\nabla
_{\parallel }\,u^{\eps}\right) +n\cdot \left( A\nabla u^{\eps}\right) =0 & 
\text{ on }\Gamma _N, \\[3mm] 
u^{\eps}=0 & \text{ on }\Gamma _{D}\,.%
\end{array}%
\right.  \label{PP}
\end{equation}%
is easily seen to be equivalent to problem (\ref{P}) by setting 
\begin{equation*}
A:=(b\otimes b)\, A_{\|}\, (b\otimes b)+(Id-b\otimes b)\,A_{\perp }\,(Id-b\otimes
b).
\end{equation*}%
Remark that by Hypothesis A, we have immediately $A_{0}\|v\|^{2}\leq
v^{t}A(x)v\leq A_{1}\|v\|^{2}\, $ for all $v\in \mathbb{R}^{d}\,\,\,$ and
for a.a. $x\in \Omega $, which is a sort of coercivity and boundedness
property for the diffusivity matrix $A$.\\

Let us now introduce the mathematical framework and define the Hilbert space 
$\mathcal{V}$ as follows 
\begin{equation}
\mathcal{V}:=\{v\in H^{1}(\Omega )\,,\text{ such that }v|_{\Gamma _{D}}=0\},
\label{V}
\end{equation}%
equipped with the scalar product 
\begin{equation}
(u,v)_{\mathcal{V}}=a(u,v):=\ds\int_{\Omega }A\nabla u\cdot \nabla v\,dx\,.
\label{bil}
\end{equation}
In the following, the bracket $(\cdot ,\cdot )$ will stand for the standard $%
L^{2}$-scalar product. We shall also frequently use the bilinear form $%
a_{\|}:\mathcal{V}\times \mathcal{V}\rightarrow \mathbb{R}$ and the
corresponding semi-norm $|\cdot |_{\|}:\mathcal{V}\rightarrow \mathbb{R}$
defined by 
\begin{equation}
a_{\|}(u,v):=\int_{\Omega }A_{\|}\nabla _{\|}u\cdot \nabla _{\|}v\,dx\,,\quad
|u|_{\|}:=\sqrt{a_{\|}(u,u)} \,.  \label{bilpar}
\end{equation}%
The weak formulation of problem (\ref{PP}) can be now written as: Find $u^{%
\eps}\in \mathcal{V}$ such that 
\begin{equation}
(P)^{\eps}\,\,\,\quad \frac{1-\eps }{\eps}a_{\|}(u^{\eps},v)+a(u^{\eps%
},v)=(f,v)\,,\quad \forall v\in \mathcal{V}\,.  \label{PPP}
\end{equation}%
Thanks to Hypothesis A and to Lax-Milgram theorem, problem (\ref{PPP})
admits a unique solution $u^{\eps}\in \mathcal{V}$ for all $\eps>0$.\\

The design of efficient schemes, which are uniformly stable
along the transition $\eps\rightarrow 0$, is based on the fundamental fact,
that the solutions $u^{\eps}\in \mathcal{V}$ of (\ref{PPP}) tend for $\eps%
\rightarrow 0$ towards some function $u^{0}$, constant along the field lines
of $b$, i.e. belonging to the following Hilbert-space 
\begin{equation}
\mathcal{G}=\{v\in \mathcal{V}\,,\text{ such that }\nabla _{\|}v=0\}\,, \quad (u,v)_{\mathcal{G}}:=(\nabla_\perp u,\nabla_\perp v)_{L^2(\Omega)}\,,
\label{G}
\end{equation}%
which consists of functions belonging to $\mathcal{V}$ with zero gradient along the field lines. Taking the test functions in (\ref{PPP}) from $\mathcal{G}$, and passing formally to the limit $\eps \rightarrow 0$, permits to identify the
problem satisfied by $u^{0}\in \mathcal{G}$, the so-called Limit model 
\begin{equation}
(L)\,\,\,\,\, a(u^{0},v)=(f,v)\,,\quad \forall v\in \mathcal{G}\,.
\label{LL}
\end{equation}%
Again, the Lax-Milgram theorem permits to show the existence and uniqueness
of a solution $u^{0}\in \mathcal{G}$ of this Limit problem (\ref{LL}). Remark that this Limit model is defined on a constrained space $\mathcal{G}$, and shall be equivalently reformulated in the sequel, on a constraint-less space.\\

The main idea behind the first AP-reformulation of problem \eqref{PP} is to rescale the parallel derivative of $u^{%
\eps}$ by introducing the auxiliary variable $q^{\eps}$ such that 
\begin{equation}
\nabla _{\Vert }q^{\eps}=\frac{1}{\eps}\nabla _{\Vert }u^{\eps}.
\label{qdef}
\end{equation}%
To ensure the uniqueness of $q^{\eps}$, we require in this section that $q^{%
\eps}=0$ on $\Gamma _{in}$. 
Remark that this is only possible if all the field lines are open and enter the domain (by $\Gamma_{in}$). For closed field lines, completely contained in $\Omega$, fixing $q^\eps$ on $\Gamma_{in}$ would be not enough for the uniqueness and other methods shall be developed in Section \ref{secStab}.\\
We thus introduce the Hilbert space 
\begin{equation}
\mathcal{L}_{in}=\{q\in L^{2}(\Omega )~/~\nabla _{\Vert }q\in L^{2}(\Omega )%
\text{ and }q|_{\Gamma _{in}}=0\}\,,  \label{GL}
\end{equation}%
equipped with the scalar product $a_{\Vert }(\cdot ,\cdot )$, inducing the
norm $|\cdot |_{\Vert }$. Note that this is indeed a norm on $\mathcal{L}%
_{in}$ since $|q|_{_{\Vert }}=0$ for $q\in \mathcal{L}_{in}$ means $\nabla
_{\Vert }q=0$ on $\Omega \,$,$\ $which in combination with the boundary
condition on $\Gamma _{in}$ implies $q=0.$

Substituting the definition (\ref{qdef}) of $q^{\eps}$ into (\ref{PPP}) yields the following problem, called in the following Inflow Asymptotic-Preserving reformulation of $(P)^{\eps}$: Find $(u^{\eps},q^{\eps})\in \mathcal{V}\times \mathcal{L}_{in}$ satisfying
\begin{equation}
(AP_{in})^{\eps}\,\,\,\left\{ 
\begin{array}{ll}
\ds a(u^{\eps},v)+(1-\eps )a_{\|}(q^{\eps},v)=(f,v), & \quad \forall v\in 
\mathcal{V} \\[3mm] 
\ds a_{\|}(u^{\eps},w)-\eps a_{\|}(q^{\eps},w)=0, & \quad \forall w\in 
\mathcal{L}_{in}\,.%
\end{array}%
\right.  \label{Pa}
\end{equation}%
The notation $(AP_{in})^{\eps}$ emphasizes the fact that we are introducing
an Asymptotic-Preserving reformulation of (\ref{PPP}), based on the Lagrange
multiplier $q^\eps$, which is uniquely determined through the inflow boundary condition on $%
\Gamma_{in}$. The reformulation \eqref{Pa} is completely equivalent to the starting model \eqref{PPP}. However, remark that putting formally $\eps=0$ in (\ref{Pa}) and
introducing for some technical reasons explained in the next subsection a
larger space $\tilde{\mathcal{L}}_{in}\supset{\mathcal{L}}_{in}$ leads to
the well-posed problem : Find $(u^{0},q^{0})\in \mathcal{V}\times \tilde{\mathcal{L}}%
_{in}$ such that 
\begin{equation}
(L_{in})\,\,\,\left\{ 
\begin{array}{ll}
\ds a(u^{0},v)+a_{\|}(q^{0},v)=(f,v), & \quad \forall v\in \mathcal{V} \\%
[3mm] 
\ds a_{\|}(u^{0},w)=0, & \quad \forall w\in \tilde{\mathcal{L}}_{in}\,,%
\end{array}%
\right.  \label{L}
\end{equation}%
which is an equivalent (saddle-point) reformulation of the Limit-problem (%
\ref{LL}). Indeed, instead of setting the problem on the constrained space $%
\mathcal{G}$, one introduces a Lagrange multiplier $q^{0}\in \tilde{\mathcal{%
L}}_{in}$, enabling us to solve the problem on the constraint-free space $%
\mathcal{V}\times \tilde{\mathcal{L}}_{in}$. 

%%%%%%%%%%%%%%%%%%%%%%%
\subsection{The Inf-Sup condition} \label{SEC22} 
%%%%%%%%%%%%%%%%%%%%%%%
Let us now focus on the mathematical study of the continuous problem $(AP_{in})^{\eps}$ and its asymptotic behaviour as $\eps$ tends towards zero. Due to the saddle-point structure of this problem, we shall make use of the traditional inf-sup theory \cite{Brezzi,ern}. The goal is to prove an $\eps$-independent inf-sup condition
corresponding to (\ref{Pa}), which ensures the
existence and uniqueness of a solution, as well as the convergence of the
AP-solution $(u^{\eps},q^{\eps})$ towards the L-solution $(u^{0},q^{0})$ as $%
\eps\rightarrow 0$. For this, we shall need a more adequate norm on the
space $\mathcal{L}_{in}$, in contrast to the one proposed in \cite{DLNN}
(see (\ref{GL})).\newline

Indeed, we would like that the form $a_{\|}(\cdot,\cdot)$ satisfies an
inf-sup estimate on the pair of spaces $\mathcal{V}$ plus the space of
functions $q$. This would be trivially the case, if we would search for $q$
in the space $\tilde{\mathcal{L}}_{in}$, defined as the closure of $\mathcal{%
L}_{in}$ in the following norm $|q|_*$ 
\begin{equation}  \label{S_norm}
|q|_*:=\sup_{v\in \mathcal{V}}\frac{a_{\|}(q,v)}{|v|_{\mathcal{V}}} \,.
\end{equation}%
Note that for all $q \in \mathcal{L}_{in}$, one has $|q|_{*}\leq |q|_{\|} $,
which means that the injection $\mathcal{L}_{in}\subset \tilde{\mathcal{L}}%
_{in}$ is continuous, however $\mathcal{L}_{in}\not=\tilde{\mathcal{L}}_{in}$
in general, as can be seen from the subsequent remarks.

\begin{remark}
\label{remal0} One can extend the continuous bilinear form to $a_{\|}: 
\tilde{\mathcal{L}}_{in} \times \mathcal{V} \rightarrow \mathbb{R}$ by
defining for each $q \in \tilde{\mathcal{L}}_{in} \backslash \mathcal{L}%
_{in} $ 
\begin{equation*}
a_{\|}(q,v):= \lim_{n \rightarrow \infty} a_{\|}(q_n,v)\,, \quad \forall v
\in \mathcal{V}\,,
\end{equation*}
for some $\{q_n\}_{n \in \mathbb{N}} \subset \mathcal{L}_{in}$ such that $%
q_n \rightarrow_{n \rightarrow \infty } q$ in $\tilde{\mathcal{L}}_{in}$.\\
Indeed, the sequence $\{q_n\}_{n \in \mathbb{N}}$ being a Cauchy-sequence in $\tilde{\mathcal{L}}_{in}$, one deduces immediately that $\{a_{\|}(q_n,v) \}_{n \in \mathbb{N}}$ is also a Cauchy-sequence for each fixed $v \in \mathcal{V}$, being hence convergent.
\end{remark}

\begin{remark}
\label{remal1} For any fixed $q \in \tilde{\mathcal{L}}_{in}$, the maximum
of $\frac{a_{\|}(q,v)}{|v|_{\mathcal{V}}}$ over $v\in \mathcal{V}$ is
attained with the function $v^{\ast }\in \mathcal{V}$, which is solution to the problem
\begin{equation}  \label{eq1}
(v^{\ast},w)_{\mathcal{V}}=a_{\|}(q,w)\quad \forall w\in \mathcal{V}\,.
\end{equation}%
Indeed, let us fix $q \in \tilde{\mathcal{L}}_{in}$ and $v^* \in \mathcal{V}$
be the corresponding solution to (\ref{eq1}). Then, any $v\in \mathcal{V}$
can be decomposed as $v=\alpha v^{\ast }+v^{\prime }$ with $\alpha:=\frac{%
(v,v^*)_{\mathcal{V}}}{|v^*|_{\mathcal{V}}^2}$ and $v^{\prime }\in \mathcal{V}$ verifying $%
(v^{\ast },v^{\prime })_{\mathcal{V}}=0$. We observe then that 
\begin{equation*}
\frac{a_{\|}(q,v)}{|v|_{\mathcal{V}}}=\frac{\alpha |v^{\ast }|_{\mathcal{V}%
}^{2}}{\sqrt{\alpha ^{2}|v^{\ast }|_{\mathcal{V}}^{2}+|v^{\prime }|_{%
\mathcal{V}}^{2}}}\leq |v^{\ast }|_{\mathcal{V}}=\frac{a_{\|}(q,v^{\ast })}{%
|v^{\ast }|_{\mathcal{V}}}\,.
\end{equation*}
\end{remark}

\begin{remark}
\label{remal2} In order to prove that $\mathcal{L}_{in}\not=\tilde{\mathcal{L%
}}_{in}$ it suffices to verify that the norms in $\mathcal{L}_{in}$ and $%
\tilde{\mathcal{L}}_{in}$ are not equivalent, \textit{i.e.} that there is no
constant $c>0$, such that $|q|_{\|} \le c |q|_*$ for all $q \in \mathcal{L}%
_{in}$. For this, it suffices to construct a sequence $\{ q_k \}_{k \in 
\mathbb{N}} \subset \mathcal{L}_{in}$, such that 
\begin{equation}  \label{infsup-bad}
\sup_{ v\in{\mathcal{V}}} \frac{a_{\|}(q_k,v)}{|q_k|_{\|}|v|_{\mathcal{V}}}
\le \frac ck \,\,\, \Rightarrow \,\,\,|q_k|_*\leq {\frac{ c }{k}}
|q_k|_{\|} \,\,\, \Rightarrow \,\,\,|q_k|_{\|}\geq {\frac{ k }{c}}
|q_k|_{*}\,, \quad \forall k \in \NN,
\end{equation}%
with $c>0$ a constant. One can easily do it in the following simple setting: let $\Omega =(0,\pi
)\times (0,\pi )$, $A_{\|}=1$, $A_{\perp}=Id$, $b=e_{2}$. Taking $q_k=\sin kx(\cos y-\cos
2y)$, which is a function in $\mathcal{L}_{in}$ for any integer $k>0$, we
can find the solution $v^{\ast}_k$ to problem (\ref{eq1}) corresponding to $%
q=q_k $ as 
\begin{equation*}
v^{\ast}_k=\sin kx\,(\frac{1}{k^{2}+1}\cos y-\frac{4}{k^{2}+4}\cos 2y).
\end{equation*}%
Now, in view of Remark \ref{remal1} 
\begin{equation*}
\sup_{v\in \mathcal{V}}\frac{a_{\|}(q_k,v)}{|q_k|_{\|}|v|_{\mathcal{V}}}=%
\frac{a_{\|}(q_k,v^{\ast}_k)}{|q_k|_{\|}|v^{\ast }|_{\mathcal{V}}}=\frac{%
|v^{\ast}_k|_{\mathcal{V}}}{|q_k|_{\|}}=\frac{1}{\sqrt{5}}\sqrt{\frac{1}{%
k^{2}+1}+\frac{16}{k^{2}+4}}.
\end{equation*}%
This gives an example of (\ref{infsup-bad}).
\end{remark}

Searching now for a solution $(u,q)$ belonging to $\mathcal{V}\times \tilde{\mathcal{L}}%
_{in}$ is the proper setting for our problem in the limit case $\eps =0 $.
Indeed, in this particular case, we have to cope with a standard saddle
point problem $(L_{in})$ and the inf-sup condition is satisfied in the space $\mathcal{V%
}\times \tilde{\mathcal{L}}_{in}$. However, this choice does not work any more for $%
\eps >0$, as  the term $a_{\|}(q,w)$ makes no more sense if we suppose
only $(q,w)\in \tilde{\mathcal{L}}_{in} \times \tilde{\mathcal{L}}_{in}$.
Hence, we propose to work for $\eps >0$ in the previous space $\mathcal{L}%
_{in}$ for the Lagrange multiplier $q$, associated however with the
following slightly different norm 
\begin{equation}  \label{norm1}
|q|_{\eps}:=(|q|^2_*+\eps |q|^2_{\|})^{\frac 12} \,, \quad \forall q \in 
\mathcal{L}_{in}\,,
\end{equation}%
which is equivalent to the old norm $| \cdot |_{\|}$ of $\mathcal{L}%
_{in} $ with $\eps $-dependent equivalence constants exploding as $\eps\to 0$~:  
\begin{equation*}
|q|_{\eps}\leq \sqrt{1+\eps } |q|_{\|}\text{ and }|q|_{\|}\leq \frac{1}{%
\sqrt{\eps }}|q|_{\eps} \,.
\end{equation*}%
The space $(\mathcal{L}_{in},|\cdot|_{\eps})$ is a Hilbert one
equipped with the scalar product 
\begin{equation*}
(q_{1},q_{2})_{\eps}:=(v_{1}^{\ast },v_{2}^{\ast })_{\mathcal{V}}+\eps %
a_{\|}(q_{1},q_{2})\,, \quad \forall q_1, \, q_2 \in \mathcal{L}_{in}\,,
\end{equation*}
where $v_{1}^{\ast },v_{2}^{\ast }$ are the unique solutions of problem (%
\ref{eq1}). In the limit $\eps \rightarrow 0$, this space transforms into
the Hilbert space $(\tilde{\mathcal{L}}_{in},|\cdot|_*)$ with the scalar product $(q_{1},q_{2})_*:=(v_{1}^{\ast
},v_{2}^{\ast })_{\mathcal{V}}$.\newline

We are finally able to introduce the right mathematical setting for a rigorous
study of the AP-problem (\ref{Pa}) and its convergence towards the Limit-problem (\ref{L}). The
Hilbert space adapted to our problem is 
\begin{equation*}
\mathcal{X}_{\eps }:=\left\{ 
\begin{array}{ccc}
\mathcal{V}\times \mathcal{L}_{in} & \text{for} & \eps >0 \\[2mm] 
\mathcal{V}\times \tilde{\mathcal{L}}_{in} & \text{for} & \eps =0\,,%
\end{array}%
\right. \quad \,\quad \|u,q\|_{\mathcal{X}_{\eps }}:=(|u|_{\mathcal{V}%
}^{2}+|q|_*^2+\eps|q|_{\|}^{2})^{1/2}\,,
\end{equation*}%
and the problem we are interested in, can now be simply written as: Find for each $\eps \in [0,1]$ the solution $(u^{\eps},q^{\eps}) \in \mathcal{X}_{\eps }$ to
\begin{equation}
(AP_{in})^{\eps}\,\,\,\left\{ 
\begin{array}{ll}
\ds a(u^{\eps},v)+(1-\eps )a_{\|}(q^{\eps},v)=(f,v), &  \\[3mm] 
\ds a_{\|}(u^{\eps},w)-\eps a_{\|}(q^{\eps},w)=0, & %
\end{array}\,, \quad \forall (v,w) \in  \mathcal{X}_{\eps }\,.
\right.  \label{Pa_bis}
\end{equation}%
For the further developments, we shall also introduce the coupled bilinear form $C_{\eps }:{\mathcal{X}_{%
\eps }}\times {\mathcal{X}_{\eps }}\rightarrow \mathbb{R}$ defined as
\begin{equation}
C_{\eps }((u,q),(v,w)):=a(u,v)+(1-\eps )a_{\|}(q,v)+a_{\|}(u,w)-\eps %
a_{\|}(q,w)\,.  \label{Ceps}
\end{equation}%
This bilinear form $C_{\eps }$ is uniformly continuous in $\eps
\in \lbrack 0,1]$, \textit{i.e.} 
\begin{equation}
C_{\eps }((u,q),(v,w))\leq 2\|u,q\|_{{\mathcal{X}_{\eps }}}\|v,w\|_{{\mathcal{X}%
_{\eps }}}\,,\quad \forall (u,q),(v,w)\in {\mathcal{X}_{\eps }}\,,
\label{cont_C}
\end{equation}%
as, using Cauchy-Schwarz inequality, one has
\begin{eqnarray*}
C_{\eps }((u,q),(v,w)) &\leq & |u|_{\mathcal{V}}|v|_{\mathcal{V}}+|q|_*|v|_{%
\mathcal{V}}+|u|_{\mathcal{V}}|w|_*+\eps |q|_{\|}|w|_{\|} \\
&\leq & \left( 2|u|_{\mathcal{V}}^{2}+|q|_*^{2}+\eps |q|_{\|}^{2}\right) ^{%
\frac{1}{2}} \left( 2|v|_{\mathcal{V}}^{2}+|w|_*^{2}+\eps %
|w|_{\|}^{2}\right) ^{\frac{1}{2}}\,.
\end{eqnarray*}%
The form $C_{\eps }$ enjoys furthermore the inf-sup property%
\begin{equation}
\inf_{(u,q)\in \mathcal{X}{_{\eps }}}\sup_{(v,w)\in \mathcal{X}{_{\eps }}}%
\frac{C_{\eps }((u,q),(v,w))}{\|u,q\|_{\mathcal{X}{_{\eps }}}\|v,w\|_{%
\mathcal{X}{_{\eps }}}}\geq \beta \,,  \label{infsup-C}
\end{equation}%
with a constant $\beta >0$ that does not depend on $\eps $. This is
established in the following lemma which is recast to a slightly more
general and abstract setting. Our particular result is recovered from this
lemma setting the bilinear forms $a(\cdot ,\cdot )$, $b(\cdot ,\cdot )$ and $%
c(\cdot ,\cdot )$ from the lemma to, respectively, $a(\cdot ,\cdot )$, $%
a_{\|}(\cdot ,\cdot )$ and $a_{\|}(\cdot ,\cdot ).$ Note that this result is very close to those from Section 4.3 of \cite{Girault} but we do not require here 
an inf-sup condition for the form $b$ in $V\times L$.

\begin{lemma}
\label{infsup} \textbf{(Inf-Sup condition)} Let $V$ and $L$ be Hilbert
spaces with their respective scalar products $a(\cdot ,\cdot )$ and $c(\cdot
,\cdot )$ inducing the norms $\|\cdot \|_{V}$ and $\|\cdot \|_{L}$. Let
moreover $\tilde{L}\supset L$ be another Hilbert space with the norm $%
\|\cdot \|_{\tilde{L}}$ such that $\|q\|_{\tilde{L}}\leq \|q\|_L$ for all $q\in L$. 
Let $b:\tilde{L}\times V\rightarrow \mathbb{R}$
be a bilinear form satisfying the continuity relation $\|b(q,v)\|\leq \|v\|_{V}\|q\|_{\tilde{L}}$
for all $v\in V$, $q\in L$ and 
\begin{equation}\label{infsup_b}
\inf_{q\in \tilde{L}}\sup_{v\in V}\frac{b(q,v)}{\|q\|_{\tilde{L}}\|v\|_{V}}%
=\alpha >0.
\end{equation}%

Set $L_{\eps }:=L$ for any $\eps >0$, $L_{0}:=\tilde{L}$ and let $X_{\eps }$
for any $\eps \geq 0$ denote the Hilbert space $V\times L_\eps$ equipped
with the norm $\|u,q\|_{X_{\eps }}:=(\|u\|_{V}^{2}+\|q\|_{\tilde{L}}^{2}+%
\eps \|q\|_{L}^{2})^{1/2}$. Introduce for any $\eps \in [0,1]$ the bilinear
form $C_{\eps }:X_{\eps}\times X_{\eps}\rightarrow \mathbb{R}$%
\begin{equation*}
C_{\eps }((u,q),(v,w))=a(u,v)+(1-\eps )b(q,v)+b(w,u)-\eps c(q,w)\,.
\end{equation*}%
Then  $C_{\eps }$ is continuous, with continuity constant $M=2$, and satisfies moreover the inf-sup condition
\begin{equation}
\inf_{(u,q)\in {{X}_{\eps }}}\sup_{(v,w)\in {{X}_{\eps }}}\frac{C_{\eps %
}((u,q),(v,w))}{\|u,q\|_{{{X}_{\eps }}}\|v,w\|_{{{X}_{\eps }}}}\geq \beta \,,
\label{infsup_C}
\end{equation}%
with a constant $\beta >0$ that depends only on $\alpha $.
\end{lemma}

\begin{proof}
To prove (\ref{infsup_C}), let us fix an arbitrary $(u,q)\in X{_{\eps }}$
and denote 
\begin{equation*}
Z:=\sup_{(v,w)\in {{X}_{\eps }}}\frac{C_{\eps }((u,q),(v,w))}{\|v,w\|_{{{X}_{%
\eps }}}}\,.
\end{equation*}%
We want to prove that $Z\geq \beta \|u,q\|_{{{X}_{\eps }}}$. First, we have 
\begin{equation*}
(1-\eps )\alpha \|q\|_{\tilde{{L}}}\leq \sup_{v\in {V}}\frac{(1-\eps )b(q,v)%
}{\|v\|_{{V}}}\leq \sup_{v\in {V}}\frac{C_{\eps }((u,q),(v,0))}{\|v\|_{{V}}}%
+\sup_{v\in {V}}\frac{a(u,v)}{\|v\|_{{V}}}\leq Z+\|u\|_{{V}}\,.
\end{equation*}%
Now, we take $v=u$, $w=-q$ and observe that%
\begin{eqnarray*}
C_{\eps }((u,q),(u,-q)) &=&a(u,u)-\eps b(q,u)+\eps c(q,q) \\
&\geq &(1-\frac{\eps }{2})\|u\|_{V}^{2}+\frac{\eps }{2}\|q\|_{L}^{2}\,,
\end{eqnarray*}%
implying altogether 
\begin{eqnarray*}
\frac{1}{2}\|u\|_{V}^{2}+\frac{\eps }{2}\|q\|_{L}^{2}+\frac{\alpha ^{2}(1-%
\eps )^{2}}{8}\|q\|_{\tilde{{L}}}^{2} &\leq &C_{\eps }((u,q),(u,-q))+\frac{1%
}{8}(Z+\|u\|_{V})^{2} \\
&\leq &Z\|u,-q\|_{X{_{\eps }}}+\frac{1}{4}Z^{2}+\frac{1}{4}\|u\|_{V}^{2}\,.
\end{eqnarray*}%
Thus,%
\begin{equation*}
\frac{1}{4}\|u\|_{V}^{2}+\frac{\eps }{2}\|q\|_{L}^{2}+\frac{\alpha ^{2}(1-%
\eps )^{2}}{8}\|q\|_{\tilde{{L}}}^{2}\leq Z\|u,q\|_{X{_{\eps }}}+\frac{1}{2}%
Z^{2}\,\leq \frac{1}{2\gamma }\|u,q\|_{X{_{\eps }}}^{2}+\frac{1+\gamma }{2}%
Z^{2}\,,
\end{equation*}%
for any $\gamma >0$ by Young inequality. Besides, we have for any $\eps
\in \lbrack 0,1]$ 
\begin{equation*}
\frac{\eps }{2}\|q\|_{L}^{2}+\frac{\alpha ^{2}(1-\eps )^{2}}{8}\|q\|_{\tilde{%
{L}}}^{2}\geq c_{0}(\|q\|_{\tilde{{L}}}^{2}+\eps \|q\|_{L}^{2})\,,
\end{equation*}%
with a constant $c_{0}>0$ depending only on $\alpha $. Indeed, for $\eps \in
\lbrack 0,1/2]$ we can observe that $(1-\eps )^{2}\geq \frac{1}{4}$ and conclude.
For $\eps \in \lbrack 1/2,1]$, we can neglect the term with $\|q\|_{\tilde{{L%
}}}^{2}$ on the left-hand side and use $\|q\|_{\tilde{L}}\leq \|q\|_L$.

\noindent Thus, assuming without loss of generality that $c_{0}\leq \frac{1}{4}$ we have
\begin{equation*}
c_{0}\|u,q\|_{X{_{\eps }}}^{2}\,\leq \frac{1}{2\gamma }\|u,q\|_{X{_{\eps }}%
}^{2}+\frac{1+\gamma }{2}Z^{2}\,.
\end{equation*}%
Taking finally a sufficiently big $\gamma $ gives immediately $\|u,q\|_{{{X}_{\eps }}}\leq
(1/\beta )Z$ with a constant $\beta >0$, independent of $\eps $.
\end{proof}

\medskip

The following theorem is the main theorem on the continuous level, which
shows that the AP-reformulation (\ref{Pa}) of problem (\ref{PPP}) is
well-posed and better adapted to capture the macro-scale behaviour of $u^%
\eps
$ in the limit $\eps \rightarrow 0$. This AP-model provides thus a link
between the micro-scale ($\eps \sim 1$) and the macro-scale ($\eps \sim 0$)
behaviour of the system.

\begin{theorem}
\label{MainThCont} \textbf{(Existence/Uniqueness/$\eps$-Convergence)} Let
hypothesis A be satisfied. The AP-problem (\ref{Pa_bis}) is well-posed for each $%
\eps\in[0,1]$, i.e. for any $f\in\mathcal{V}^{\prime }$ and any $\eps\in[0,1]
$ there exists a unique solution $(u^\eps,q^\eps) \in {\mathcal{X}_\eps}$,
which satisfies 
\begin{equation*}
\|u^\eps,q^\eps\|_{\mathcal{X}_\eps} \le \frac{1}{\beta} \|f\|_{\mathcal{V}%
^{\prime }}\,,
\end{equation*}
with $\beta >0$ the constant given by the inf-sup condition (\ref{infsup_C}%
). Moreover, we have the $\eps$-convergences 
\begin{equation*}
\|u^{\eps }-u^{0},q^{\eps }-q^{0}\|_{{\mathcal{X}_\eps}}\rightarrow 0\,,
\quad \text{for} \quad \eps \rightarrow 0\,.
\end{equation*}
If we suppose more regular data, as $f \in L^2(\Omega)$, then one has even $q^0 \in \mathcal{L}_{in}$ and the
estimates 
\begin{equation} \label{Ces}
|u^\eps- u^0|_{\mathcal{V}} \le C \sqrt{\eps}\,, \quad |q^\eps- q^0|_* \le
C \sqrt{\eps}\,,
\end{equation}
with $C>0$ some $\eps$-independent constant. 
\end{theorem}

\begin{proof}
The existence and uniqueness of a solution $(u^{\eps},q^{\eps})\in {\mathcal{%
X}_{\eps}}$ for each $\eps\geq 0$, is a simple consequence of the Banach-Ne\v{c}as-Babu\v{s}ka (hereafter BNB)
theorem \cite{ern}. To prove the convergence $(u^{\eps},q^{\eps})\rightarrow
(u^{0},q^{0})$ we assume first that $f\in L^{2}(\Omega )$. We have proved in 
\cite{DLNN} that $q^{0}\in \mathcal{L}_{in}$ in this case. Subtracting now \eqref{L} from \eqref{Pa} yields
\begin{equation*}
C_{\eps}((u^{\eps}-u^{0},q^{\eps}-q^{0}),(v,w))=\eps a_{\|}(q^{0},v+w)\,,%
\quad \forall (v,w)\in \mathcal{V}\times {\mathcal{L}}_{in}\,.
\end{equation*}%
Thus, for any $\eps>0$, by the inf-sup property, there exist $(v,w)\in \mathcal{X}_{\eps}=\mathcal{V}\times {\mathcal{L}}_{in}$
such that%
\begin{equation*}
\beta' \|u^{\eps}-u^{0},q^{\eps}-q^{0}\|_{{\mathcal{X}_{\eps}}}\,\|v,w\|_{{%
\mathcal{X}_{\eps}}}\leq \eps a_{\|}(q^{0},v+w) \leq\eps|q^{0}|_{\|}|v+w|_{%
\|}\leq \sqrt{2\eps}|q^{0}|_{\|}\|v,w\|_{{\mathcal{X}_{\eps}}}\,,
\end{equation*}%
with some $0<\beta'<\beta$, for ex. $\beta':=\beta/2$, which implies $\|u^{\eps}-u^{0},q^{\eps}-q^{0}\|_{{\mathcal{X}_{\eps}}}\leq 
\frac{\sqrt{2\eps}}{\beta'}\,|q^{0}|_{\|}$,  leading to the convergence estimates \eqref{Ces}.
\newline
We are now going to generalize this result to any $f\in \mathcal{V}^{\prime
} $ by a density argument. Let us denote simply by $U^{\eps}(f)$ the
solution $(u^{\eps},q^{\eps})\in \mathcal{X}_{\eps}$ of (\ref{Pa_bis})
associated to $f\in \mathcal{V}^{\prime }$. Now fix some $f\in \mathcal{V}%
^{\prime }$. Since $L^{2}(\Omega )$ is dense in $\mathcal{V}^{\prime }$, for
any $\delta >0$ there exists $f_{\delta }\in L^{2}(\Omega )$ such that $%
f=f_{\delta }+R_{\delta }$ with $\|R_{\delta }\|_{\mathcal{V}^{\prime }}<%
\frac{\delta \beta' }{4}$. Hence, there exists $\eps_{0}>0$ such that for all 
$\eps<\eps_{0}$%
\begin{equation*}
\|U^{\eps}(f)-U^{0}(f)\|_{\mathcal{X}_{\eps}}\leq \,\|U^{\eps}(f_{\delta
})-U^{0}(f_{\delta })\|_{\mathcal{X}_{\eps}}+\|U^{\eps}(R_{\delta
})-U^{0}(R_{\delta })\|_{\mathcal{X}_{\eps}}<\frac{\delta }{2}+\frac{2}{%
\beta' }\|R_{\delta }\|_{\mathcal{V}^{\prime }}<\delta \,.
\end{equation*}%
Here we used the fact that $f_{\delta }\in L^{2}(\Omega )$ which implies $\|U^{\eps}(f_{\delta })-U^{0}(f_{\delta })\|_{\mathcal{X}%
_{\eps}}\leq C{\frac{\sqrt{\eps}}{\beta' }}\rightarrow 0$ as $\eps\rightarrow
0$.
\end{proof}

%%%%%%%%%%%%%%%%%%%%%%%

\subsection{The numerical analysis of the inflow AP-scheme}

%%%%%%%%%%%%%%%%%%%%%%%
Having reformulated on the continuous level the singularly-perturbed  problem $(P)^\eps$ into a system  $(AP_{in})^\eps$ which is better suited to capture the macroscopic limit as $\eps \rightarrow 0$, we shall now discretize via a standard approach this new system and analyse the obtained AP-scheme in detail. In particular error estimates are deduced and the convergence of
the scheme independently on the anisotropy parameter $\eps$ is shown.

Let us introduce a mesh ${\mathcal{T}}_{h}$ on $\Omega $ consisting of
triangles (resp. rectangles) of maximal size $h$, let ${V}_{h}\subset 
\mathcal{V}$ be the finite dimensional space of $\mathbb{P}_{k}$ (resp. $%
\mathbb{Q}_{k}$) finite elements on ${\mathcal{T}}_{h}$, and let us define ${L}_{h}:={V}_{h}\cap \mathcal{L}_{in}={V}_{h}\cap \tilde{\mathcal{L}}_{in}$ as well as $X_{h}:={V}_{h}\times {L}_{h}$. Note
that we require ${L}_{h}\subset V_{h}$, which signifies that we enforce the
boundary conditions on $\Gamma _{D}$ for functions in ${L}_{h}$, cf. \cite%
{DLNN}. We are thus looking for a discrete solution $(u_{h}^{\eps},q_{h}^{%
\eps })\in {V}_{h}\times {L}_{h}$ of 
\begin{equation}
(AP_{in})^\eps_{h}\,\,\,\left\{ 
\begin{array}{l}
a(u_{h}^{\eps },v_{h})+(1-\eps )a_{\|}(q_{h}^{\eps },v_{h})=(f,v_{h})\,,%
\quad \forall v_{h}\in {V}_{h} \\[3mm] 
a_{\|}(u_{h}^{\eps },w_{h})-\eps a_{\|}(q_{h}^{\eps },w_{h})=0,\quad \forall
w_{h}\in {L}_{h}\,.%
\end{array}%
\right.  \label{Ph}
\end{equation}

\noindent The analysis of this scheme would be straightforward if the discrete inf-sup
condition 
\begin{equation}  \label{infsupfake}
\inf_{q_{h}\in {L}_{h}}\sup_{v_{h}\in {V}_{h}}\frac{a_{\|}(q_{h},v_{h})}{%
|q_{h}|_*|v_{h}|_{\mathcal{V}}}\geq \alpha\,,
\end{equation}%
were satisfied with an $\eps$- as well as mesh-independent constant $\alpha >0$. However, this
constant is unfortunately mesh dependent, as shown in Appendix C. In order
to circumvent this difficulty we introduce the following mesh-dependent norm
on $L_{h}$%
\begin{equation*}
|q_{h}|_{*h}:=\sup_{v_{h}\in {V}_{h}}\frac{a_{\|}(q_{h},v_{h})}{|v_{h}|_{%
\mathcal{V}}}\,.
\end{equation*}%
Note that this is indeed a norm, since $|q_{h}|_{*h}=0$ implies $%
a_{\|}(q_{h},q_{h})=0$ due to the inclusion ${L}_{h}\subset V_{h}$ and thus $%
\nabla _{\|}q_{h}=0$, which in combination with the boundary conditions on $%
\Gamma _{in}$ yields $q_{h}=0.$\\

\noindent We now equip the space $X_{h}$ with the norm 
$$
\|u_{h},q_{h}\|_{X_{\eps,h}}=(|u_{h}|_{\mathcal{V}}^{2}+|q_{h}|_{*h}^{2}+\eps |q_{h}|_{\|}^{2})^{1/2}\,.
$$
By Lemma \ref{infsup}, the bilinear form $C_{\eps}$
is continuous on $X_{h}$ with this norm, and enjoys the inf-sup property%
\begin{equation}
\inf_{(u_{h},q_{h})\in X_{h}}\sup_{(v_{h},w_{h})\in X_{h}}\frac{C_{\eps%
}((u_{h},q_{h}),(v_{h},w_{h}))}{\|u_{h},q_{h}\|_{X_{\eps,h}}\|v_{h},w_{h}%
\|_{X_{\eps,h}}}\geq \beta \,,  \label{infsuph}
\end{equation}%
with a constant $\beta >0$ that does not depend neither on the mesh nor on $%
\eps $. This implies the discrete version of theorem \ref{MainThCont}.

\begin{theorem}
\label{APproph} \textbf{(Discrete Existence/Uniqueness/$\eps$-Convergence)}
The discrete AP-problem (\ref{Ph}) admits for each fixed $h >0$ and $\eps \ge
0 $ a unique solution $(u_{h}^{\eps},q_{h}^{\eps})\in {V}_{h}\times {L}_{h}$%
, satisfying 
\begin{equation*}
\|u_{h}^{\eps},q_{h}^{\eps}\|_{X_{\eps,h}}\leq \frac{1}{\beta }\|f\|_{%
\mathcal{V}^{\prime }}\,,
\end{equation*}%
and one has the $\eps$-convergence 
$$
\|u_{h}^{\eps}-u_{h}^{0},q_{h}^{\eps}-q_{h}^{0}\|_{X_{\eps,h}}
\rightarrow 0 \quad \textrm{for} \quad \eps\rightarrow 0.
$$

Moreover, the condition number of the matrix corresponding to problem (\ref{Ph}) is bounded by a constant independent of $\eps$ (assuming that the same bases
of $V_{h}$ and $L_h$ are chosen for all values of $\eps$).
\end{theorem}

\begin{proof}
The existence and uniqueness of a solution $(u_{h}^{\eps},q_{h}^{\eps})\in
 X_{h}$ for each $\eps\geq 0$, is a simple consequence of the BNB theorem \cite{ern}. The convergence $(u_{h}^{\eps},q_{h}^{\eps})\rightarrow _{\eps\rightarrow 0}(u_{h}^{0},q_{h}^{0})$ can be established by the same arguments as in
the proof of Theorem \ref{MainThCont}.

We turn now to the study of the condition number. Let $\{\phi^u
_{1},\ldots ,\phi^u _{N^u}\}$ (resp. $\{\phi^q
_{1},\ldots ,\phi^q _{N^q}\}$) be a basis of $V_{h}$ (resp. $L_h$). 
We shall identify every function $u_{h}\in V_{h}$ (resp. $q_{h}\in L_{h}$)
with a vector $\vec{u}\in \mathbb{R}^{N_u}$ (resp. $\vec{q}\in \mathbb{R}^{N_q}$) consisting of the expansion
coefficients of $u_{h}$ (resp. $q_h$) in these bases.  
Denoting the Euclidean norm of a vector by $||\cdot||_{2}$ and using the equivalence of norms on a finite dimensional space, we observe that for all $u_{h}\in V_{h}$ and $q_{h}\in L_{h}$ we have
\begin{equation*}
\mu _{u}||\vec{u}||_{2}\leq |u_{h}|_{\mathcal{V}}\leq \nu _{u}||\vec{u%
}||_{2}\,,\quad \mu _{q}||\vec{q}||_{2}\leq |q_{h}|_{||}\leq \nu
_{q}||\vec{q}||_{2},\quad \mu _{\ast }||\vec{q}||_{2}\leq |q_{h}|_{\ast
h}\leq \nu _{\ast }||\vec{q}||_{2}
\end{equation*}%
with some positive constants $\mu$'s and $\nu$'s. 
We shall moreover identify any $\Phi _{h}=(u_{h},q_{h})\in X_{h}$
with a vector $\vec{\Phi}\in \mathbb{R}^{N}$, $N=N_u+N_q$, such that  $\vec{\Phi}=(\vec{u}^T,\vec{q}^T)^T$.
We observe for any such $\Phi _{h}$ that 
$||\vec{\Phi}||_2^{2} =||\vec{u}||_2^{2}+||\vec{q}||_2^{2}  $, which in combination with the estimates above gives
$$
\min \{ \mu_u^2, \mu_*^2+\eps\, \mu_q^2 \} ||\vec{\Phi}||_2^{2} \le ||\Phi _{h}||_{X_{\varepsilon ,h}}^2 \le 
\max \{ \nu_u^2, \nu_*^2+\eps\, \nu_q^2 \} ||\vec{\Phi}||_2^{2}\,.
$$

Let now $A$  denote the $N_u\times N_u$ matrix with entries $a_{ij}=a(\phi^u_{i},\phi^u_{j})$, 
$B$  the $N_u\times N_q$ matrix with entries $b_{ij}=a_{||}(\phi^u_{i},\phi^q_{j})$, 
and $C$  the $N_q\times N_q$ matrix with entries $c_{ij}=a_{||}(\phi^q_{i},\phi^q_{j})$. 
The matrix corresponding to problem (\ref{Ph}) can be then written in the following block form 
$$
\mathbb{A}^{\varepsilon }=\left(\begin{array}{cc}
  A & ( 1 - \varepsilon) B\\
  B^T & \varepsilon C
\end{array}\right)\,.
$$
Its 2-norm denoted by $||\cdot ||_{2}$ is bounded for all $\varepsilon \in [0,1]$ by
\begin{eqnarray*}
  ||  \mathbb{A}^{\varepsilon} ||_2 & = &  \sup_{\vec{\Phi}, \vec{\Psi} \in
  \mathbb{R}^N \setminus \{0\}}  \frac{\vec{\Psi} \cdot
  \mathbb{A}^{\varepsilon}  \vec{\Phi}}{|| {  \vec{\Phi}
  ||_2^{} ||  \vec{\Psi} ||_2}} =  \sup_{\Phi_h, \Psi_h \in X_h \setminus
  \{0\}}  \frac{C_{\eps} (\Phi_h, \Psi_h)}{||  \vec{\Phi} ||_2^{} ||  \vec{\Psi}
  ||_2^{}}\\
  & \leq &  M \sup_{\Phi_h, \Psi_h \in X_h \setminus \{0\}}  \frac{|| \Phi_h
  ||_{X_{\varepsilon, h}}^{} || \Psi_h ||_{X_{\varepsilon, h}}^{}}{||
  \vec{\Phi} ||_2 ||\vec{\Psi} ||_2} \leq M \max (\nu_u^2, \nu_{\ast}^2 +
  \nu_q^2) 
\end{eqnarray*}
where $M$ is the ($\varepsilon $-independent) continuity constant of $%
C_{\varepsilon }$. Similarly, using the inf-sup property of this bilinear
form we know that for all $\Phi _{h}\in X_{h}$ there exists $\Psi _{h}\in
X_{h}$ such that%
\begin{equation*}
\begin{array}{lll}
\displaystyle \beta' \min\{\mu_u^{2},\mu_*^2\}||\vec{\Phi}||_{2}||\vec{\Psi}||_{2}
&\leq& \beta ||\Phi
_{h}||_{X_{\varepsilon ,h}}||\Psi _{h}||_{X_{\varepsilon ,h}}\leq C_{\eps%
}(\Phi _{h},\Psi _{h})\\[3mm]
 &=& \vec{\Psi}\cdot \mathbb{A}^{\varepsilon }\vec{\Phi} \leq \displaystyle ||\vec{\Psi}||_{2}||\mathbb{A}^{\varepsilon }\vec{\Phi}||_{2}\,,
\end{array}
\end{equation*}
with an $\eps$-independent constant $\beta>\beta'>0$. This simplifies to 
$$
\beta' \min\{\mu_u^{2},\mu_*^2\}||\vec{\Phi}||_{2}\leq
||\mathbb{A}^{\varepsilon }\vec{\Phi}||_{2}\,,
$$ 
or equivalently 
$$\beta' \min\{\mu_u^{2},\mu_*^2\}
||(\mathbb{A}^{\varepsilon })^{-1}\vec{\Phi}||_{2}\leq ||\vec{\Phi}||_{2}\,, \quad \forall\,\vec{\Phi}\in \mathbb{R}^{N}\,.
$$
Thus, the condition number can be estimated as
\begin{equation}\label{ESTii}
cond_{2}(\mathbb{A}^{\varepsilon })=||\mathbb{A}^{\varepsilon }||_{2}||(\mathbb{A}^{\varepsilon
})^{-1}||_{2}\leq \frac{M\max (\nu_u^2, \nu_{\ast}^2 +  \nu_q^2)}{\beta' \min\{\mu_u^{2},\mu_*^2\}}\,,
\end{equation}
which is an $\eps$-independent bound.
\end{proof}

\begin{remark}\label{CondEst}
Let us try to be more quantitative in our estimate of $cond_{2}(\mathbb{A}^{\varepsilon
})$. In what follows, the
symbols $\lesssim $ and $\sim$ will hide the constants of order 1, independent of the mesh. 
Consider the standard finite element setting: the bases of $V_{h}$ and $%
L_{h}$ are formed by the hat finite element functions on a quasi-uniform
mesh. We know in this case that $||u_{h}||_{L^{2}}^{2}\sim h^{2}||\vec{u}%
||_{2}^{2}$ and $|u_{h}|_{\mathcal{V}}\le C_Ih^{-1} ||u_{h}||_{L^{2}}$ by
the inverse inequality with a constant $C_I>0$ that depends only on the mesh regularity \cite{ern}. We also recall the Poincar\'e inequality $||u_{h}||_{L^{2}}\le C_P |u_{h}|_{\mathcal{V}}$. The same holds for $q_{h}$ and leads to
\begin{equation*}
\mu _{u}\sim \mu _{q}\sim h\text{ and }\nu _{u}\sim \nu _{q}\sim 1\,.
\end{equation*}%
We also have $|q_{h}|_{\ast h }\le |q_{h}|_{||}$, hence $\nu _{\ast
}\le \nu _{q}$. Moreover, for any $q_h\in L_h$ we prove, using the
inverse and Poincar\'e inequalities, that
\begin{eqnarray*}
|q_{h}|_{\ast h} &\geq &\frac{a_{\Vert }(q_{h},q_{h})}{|q_{h}|_{%
\mathcal{V}}}=\frac{|q_{h}|_{||}^{2}}{\left( |q_{h}|_{\perp
}^{2}+|q_{h}|_{||}^{2}\right) ^{1/2}}\geq \frac{|q_{h}|_{||}^{2}}{\left(
C_{I}^{2}h^{-2}||q_{h}||_{L^{2}}^{2}+|q_{h}|_{||}^{2}\right) ^{1/2}} \\
&=&\frac{C_{P}^{2}C_{I}^{2}h^{-2}|q_{h}|_{||}^{2}+|q_{h}|_{||}^{2}}{%
(C_{P}^{2}C_{I}^{2}h^{-2}+1)\left(
C_{I}^{2}h^{-2}||q_{h}||_{L^{2}}^{2}+|q_{h}|_{||}^{2}\right) ^{1/2}}\geq 
\frac{\left( C_{I}^{2}h^{-2}||q_{h}||_{L^{2}}^{2}+|q_{h}|_{||}^{2}\right)
^{1/2}}{C_{P}^{2}C_{I}^{2}h^{-2}+1} \\
&\geq& \frac{\left( C_{I}^{2}h^{-2}+C_{P}^{2}\right) ^{1/2}}{%
C_{P}^{2}C_{I}^{2}h^{-2}+1}||q_{h}||_{L^{2}}\sim h||q_{h}||_{L^{2}} \sim
h^{2}||\vec{q}||_{2}\,.
\end{eqnarray*}%
This implies $\mu _{\ast }\gtrsim h^{2}$, so that (\ref{ESTii}) becomes finally%
\begin{equation*}
cond_{2}(\mathbb{A}^{\varepsilon })\lesssim \frac{1}{h^{4}}\,.
\end{equation*}
\end{remark}

\begin{theorem}\label{thm:Err_est}
\textbf{($h$-Convergence)} Let $k\geq 1$ and $V_{h}\subset {\mathcal{V}}$ be
the $\mathbb{P}_{k}$ (or $\mathbb{Q}_{k}$) finite element space on a regular
mesh ${\mathcal{T}}_{h}$. Suppose moreover that problem (\ref{Pa_bis}) has a
solution $(u^{\eps},q^{\eps})\in {\mathcal{X}}_{\eps}$, having the
regularity $u^{\eps}\in H^{k+1}(\Omega )$, $q^{\eps}\in H^{k+1}(\Omega )$.
Then one has the estimate
\begin{equation}
|u^{\eps}-u_{h}^{\eps}|_{\mathcal{V}}\leq c\,h^{k}(|u^{\eps}|_{H^{k+1}}+|q^{\eps%
}|_{H^{k+1}})\,,  \label{Err_est}
\end{equation}%
with a constant $c>0$ that depends neither on the mesh, nor on $\eps$.
\end{theorem}

\begin{proof}
Let $\hat{u}_{h}^{\eps}\in V_{h}$ and \ $\hat{q}_{h}^{\eps}\in L_{h}$ be the
standard nodal interpolant of $u^{\eps}$ and $q^{\eps}$ satisfying \cite{ern}%
\begin{equation*}
|u^{\eps}-\hat{u}_{h}^{\eps}|_{H^1}\leq c\,h^{k}|u^{\eps}|_{H^{k+1}}\text{
and \ }|q^{\eps}-\hat{q}_{h}^{\eps}|_{H^{1}}\leq c\,h^{k}|q^{\eps%
}|_{H^{k+1}}\,.
\end{equation*}%
We can now derive the error estimates in the $H^{1}$-norm for $u^{\eps}$ in
the way similar to Cea's lemma: by the inf-sup property, there exists $%
(v_{h},w_{h})\in X_{h}$ with $\|v_{h},w_{h}\|_{X_{\eps,h}}=1$ such that (with some $0<\beta'<\beta$)
\begin{eqnarray} \label{noqh}
|u_{h}^{\eps}-\hat{u}_{h}^{\eps}|_{H^{1}} &\leq &\|u_{h}^{\eps}-\hat{u}_{h}^{%
\eps},q_{h}^{\eps}-\hat{q}_{h}^{\eps}\|_{X_{\eps,h}}\leq \frac{1}{\beta' }C_{%
\eps}((u_{h}^{\eps}-\hat{u}_{h}^{\eps},q_{h}^{\eps}-\hat{q}_{h}^{\eps%
}),(v_{h},w_{h})) \\
&=&\frac{1}{\beta' }C_{\eps}((u^{\eps}-\hat{u}_{h}^{\eps},q^{\eps}-\hat{q}%
_{h}^{\eps}),(v_{h},w_{h})) \notag\\
&\leq& c\,(|u^{\eps}-\hat{u}_{h}^{\eps}|_{\mathcal{V%
}}^{2}+|q^{\eps}-\hat{q}_{h}^{\eps}|_{\ast h}^{2}+\eps|q^{\eps}-\hat{q}_{h}^{%
\eps}|_{\Vert }^{2})^{1/2} 
\notag\\[2mm]
&\leq &c\,(|u^{\eps}-\hat{u}_{h}^{\eps}|_{H^{1}}+|q^{\eps}-\hat{q}_{h}^{\eps%
}|_{H^{1}})\,,
\notag
\end{eqnarray}%
since%
\begin{equation*}
|q^{\eps}-\hat{q}_{h}^{\eps}|_{\ast h}=\sup_{v_{h}\in {V}_{h}}\frac{a_{\Vert
}(q^{\eps}-\hat{q}_{h}^{\eps},v_{h})}{|v_{h}|_{\mathcal{V}}}\leq
|q^{\eps}-\hat{q}_{h}^{\eps}|_{\Vert }\leq |q^{\eps}-\hat{q}_{h}^{\eps%
}|_{H^{1}}.
\end{equation*}%
We can now employ the interpolation error estimates to conclude.
\end{proof}

\begin{remark} \label{Rem9}
The error estimate (\ref{Err_est}) would be of course useless if the norms $%
|u^{\eps}|_{H^{k+1}}$, $|q^{\eps}|_{H^{k+1}}$ were $\eps-$dependent and
exploding in the limit $\eps\rightarrow 0$. Fortunately, it is not the case.
We expect indeed that $|u^{\eps}|_{H^{k+1}}$ is bounded uniformly in $\eps$
by the norm of $f$ in $H^{k-1}(\Omega )$ and $|q^{\eps}|_{H^{k+1}}$ is
bounded uniformly in $\eps$ by the norm of $f$ in $H^{k+1}(\Omega )$. This
can be easily proved in the case of a simple aligned geometry, see Appendix
A. We conjecture that this remains true also in a general setting. 
\end{remark}

\begin{remark}
If we do not omit the norm of $q^\eps_h-\hat{q}^\eps_h$ in the left-hand side of the first inequality in (\ref{noqh}), we also get an error estimate for $q^\eps_h$
$$
|q^{\eps}-q_{h}^{\eps}|_{||}\leq c\frac{h^{k}}{\sqrt\eps}(|u^{\eps}|_{H^{k+1}}+|q^{\eps}|_{H^{k+1}})\,,
$$
which degenerates as $\eps$ goes to 0. We are not sure, if this estimate is sharp, but we recall that $q^\eps$ is an auxiliary variable, without any intrinsic meaning.
\end{remark}

\color{black} %%%%%%%%%%%%%%%%%%%%%%%%%%%%%%%

\section{Second AP-reformulation for general field lines}

\label{secStab} %%%%%%%%%%%%%%%%%%%%%%%%%%%%%%%
The fundamental idea of the AP-reformulation introduced in Section \ref{SEC2}
is the introduction of a Lagrange multiplier $q^{\eps}\in {\mathcal{L}}_{in}$
in order to handle well with the constraint $\nabla _{\Vert }u^{0}=0$ in the
limit $\eps\rightarrow 0$. This Lagrange multiplier was uniquely determined
up to a constant on the field lines, which was fixed by imposing $q_{|\Gamma
_{in}}^{\eps}=0$. The disadvantage of this scheme is that it requires to
identify the inflow part of the boundary, which can be cumbersome in
practice or even not possible if some of the field lines are closed and lie
completely inside the domain $\Omega $. It is thus tempting to abandon the
zero inflow boundary condition and to search for the auxiliary $q^{\eps}$-variable
in the Hilbert space 
\begin{equation}
\mathcal{L}=\{\xi\in L^{2}(\Omega )~/~\nabla _{{||}}\xi\in L^{2}(\Omega )\}\,,\quad (u,v)_{\mathcal{L}}:=(u,w)+(\nabla_{||}u,\nabla_{||}w )\,.
\label{Lspace}
\end{equation}%
The problem with this idea is that we loose now uniqueness of the solution
if we attempt to implement the AP-reformulation \eqref{Pa} just changing $\mathcal{L}%
_{in}$ $\ $to $\mathcal{L}$. To circumvent this difficulty, it was proposed
in \cite{NarskiOttaviani} to introduce a stabilization term into the AP
reformulation so that it becomes: Find $(u^{\eps,\sigma },\xi^{\eps,\sigma
})\in \mathcal{V}\times \mathcal{L}$ such that%
\begin{equation}
(AP_{\mathcal{S}})^{\eps,\sigma }\,\,\,\left\{ 
\begin{array}{ll}
\ds a(u^{\eps,\sigma },v)+(1-\eps)a_{\Vert }(\xi ^{\eps,\sigma },v)=(f,v), & 
\quad \forall v\in \mathcal{V} \\[3mm] 
\ds a_{\Vert }(u^{\eps,\sigma },w)-\eps a_{\Vert }(\xi ^{\eps,\sigma
},w)-\sigma (\xi ^{\eps,\sigma },w)=0, & \quad \forall w\in \mathcal{L}\,,%
\end{array}%
\right.  \label{APstab}
\end{equation}%
where $\sigma >0$ is a small stabilization parameter, chosen consistently
with the overall discretization error. It is this term which permits to have the uniqueness, as will be shown in Lemma \ref{LemContSig}.

In the limit $\eps\rightarrow0$ this system yields: Given $\sigma>0$,
find $(u^{0,\sigma},\xi^{0,\sigma})\in\mathcal{V}\times\tilde{\mathcal{L}}^{2}$,
solution to 
\begin{equation}
(L_{\mathcal{S}})^{\sigma}\,\,\,\left\{ \begin{array}{ll}
\ds a(u^{0,\sigma},v)+a_{\Vert}(\xi^{0,\sigma},v)=(f,v), & \quad\forall v\in\mathcal{V}\\[3mm]
\ds a_{\Vert}(u^{0,\sigma},w)-\sigma(\xi^{0,\sigma},w)=0, & \quad\forall w\in\tilde{\mathcal{L}}^2 \,,
\end{array}\right.\label{Lstab}
\end{equation}
where $\mathcal{\tilde{L}}^2$ is, loosely speaking, the
closure of $\mathcal{L}$ in the $|\cdot|_{\ast}$ semi-norm (\ref{S_norm}) intersected with $L^{2}(\Omega)$,
i.e. 
\[
 \tilde{\mathcal{L}}^2=\left\{ \xi\in L^2(\Omega)\,/\,\sup_{v\in\mathcal{V}}\frac{a_{\|}(\xi,v)}{|v|_{\mathcal{V}}}<\infty\right\} \frac{_{}}{}_{}
\]
This space is a Hilbert-space associated with the scalar product
$$
 (u,w)_{\tilde{\mathcal{L}}^2}:=(u,w)+(u^*,w^*)_{L^2}\,, \quad \forall (u,w) \in \tilde{\mathcal{L}}^2\,, 
$$
where $u^*$ resp. $w^*$ are the unique solutions of \eqref{eq1} corresponding to $u$ resp. $w$. 
We need this special space, first of all, to be able to treat the limit-problem $(L_{\mathcal S})^\sigma$ with the inf-sup theory, similar to the inflow-case, and also in order to be able to define the stabilization term $\sigma(\xi^{0,\sigma},w)$.

\begin{remark}
Remark also that we have $\tilde{\mathcal{L}}^2\neq\mathcal{L}$.
Let us prove it in the following simple setting: let $\Omega=(0,\pi)\times(0,\pi)$,
$A_{\|}=1$, $A_{\perp}=Id$, $b=e_{2}$. For any $q=\sum_{k,l=1}^{\infty}q_{kl}\sin kx\cos ly$,
the calculation as in Remark 3 gives 
\[
|q|_{\ast}^{2}=\sum_{k,l=1}^{\infty}\frac{l^{4}}{k^{2}+l^{2}}|q_{kl}|^{2}\,,
\]
so that taking $q_{kl}$ such that $q_{kl}=\frac{1}{l}$ if $k=l^{2}$ and $q_{kl}=0$ for any $k\neq l^{2}$ we have
\[
|q|_{\ast}^{2}=\sum_{l=1}^{\infty}\frac{l^{2}}{l^{4}+l^{2}}<\infty\,,
\]
so that $q\in\tilde{\mathcal{L}}$. Moreover, clearly $q\in L^{2}(\Omega)$.
However, 
\[
|q_{}|_{\|}^{2}=\sum_{k,l=1}^{\infty}l^{2}|q_{kl}|^{2}=\infty\,.
\]
\end{remark}
 %%%%%%%%%%%%%%%%%%%%%%%%%%%%%%%%%%

%%%%%%%%%%%%%%%%%%%%%%%%%%%%%%%%%%

\subsection{Mathematical analysis on the continuous level}

%%%%%%%%%%%%%%%%%%%%%%%%%%%%%%%%%%%
To analyze the well-posedness of problem (\ref{APstab}) ans its asymptotic limit behaviour for $\eps \rightarrow 0$, we shall rewrite it in an equivalent manner, better suited for mathematical studies. For this, we observe
first that the second equation in (\ref{APstab}) gives for all $\varepsilon
,\sigma >0$ 
\begin{equation*}
(\xi ^{\eps,\sigma },w)=0,\quad \forall w\in \mathcal{G_L}\,, \quad \textrm{with} \quad \mathcal{G_L}:=\{v\in \mathcal{L}\,|\,\nabla _{\|}v=0\}\,,
\end{equation*}%
which means that $\xi ^{\eps,\sigma }$ belongs to the following space%
\begin{equation}
\mathcal{A}=\{\xi \in \mathcal{L}~/~(\xi ,w)=0\quad \forall w\in \mathcal{G_L}%
\}\,,  \label{Aspace}
\end{equation}%
which consists thus of functions from $\mathcal{L}$ with zero (weighted)
average along the field lines. 
Remark that one has the decomposition $\mathcal{L}=\mathcal{G_L} \oplus ^\perp \mathcal{A}$. Problem (\ref{APstab}) can be
hence rewritten as: Find $(u^{\eps,\sigma },\xi^{\eps,\sigma })\in \mathcal{V}%
\times \mathcal{A}$ such that%
\begin{equation}
(AP_{\mathcal{S}}^{\prime })^{\eps,\sigma }\,\,\,\left\{ 
\begin{array}{ll}
\ds a(u^{\eps,\sigma },v)+(1-\eps)a_{\Vert }(\xi ^{\eps,\sigma },v)=(f,v), & 
\quad \forall v\in \mathcal{V} \\[3mm] 
\ds a_{\Vert }(u^{\eps,\sigma },w)-\eps a_{\Vert }(\xi ^{\eps,\sigma
},w)-\sigma (\xi ^{\eps,\sigma },w)=0, & \quad \forall w\in \mathcal{A}\,.%
\end{array}%
\right.  \label{APstab1}
\end{equation}%
We emphasize that this reformulation is completely equivalent to (\ref{APstab}) for all $\eps>0$ and $\sigma >0$ and is done solely for the purposes of mathematical analysis. The formulation  used for the
numerical discretization will be (\ref{APstab}).

Note that $\mathcal{A}$ becomes a Hilbert space when equipped with the scalar product $a_{||} (\cdot,\cdot)$ and corresponding norm $|\cdot
|_{\Vert }$. Indeed, if $\xi \in \mathcal{A}$ and $|\xi |_{\Vert }=0$ then $%
\xi \in \mathcal{G_L}$ which implies $\xi =0$ since $\mathcal{A}$ is
orthogonal to $\mathcal{G_L}$. For the same reasons, the semi-norm $|\cdot|_{\ast}$ (\ref{S_norm}) is actually a norm when applied to $\mathcal{A}$. We can thus introduce the closure $\tilde{\mathcal{A}}$ of $\mathcal{A}$ with respect to $|\cdot|_*$, needed as usual, for the $\eps \rightarrow 0$ limit model. The $(L_{\mathcal{S}})^{\sigma }$-problem will be shown to be equivalent to: Find $(u^{0,\sigma},\xi^{0,\sigma})\in\mathcal{V}\times(\tilde{\mathcal{A}}\cap L^2(\Omega))$,
solution to 
\begin{equation}
(L'_{\mathcal{S}})^{\sigma}\,\,\,\left\{ \begin{array}{ll}
\ds a(u^{0,\sigma},v)+a_{\Vert}(\xi^{0,\sigma},v)=(f,v), & \quad\forall v\in\mathcal{V}\\[3mm]
\ds a_{\Vert}(u^{0,\sigma},w)-\sigma(\xi^{0,\sigma},w)=0, & \quad\forall w\in\tilde{\mathcal{A}}\cap L^2(\Omega) \,.
\end{array}\right.\label{Lstab1}
\end{equation}

As mentioned earlier, formulations \eqref{APstab1} and \eqref{Lstab1} are better adapted for the mathematical study, then the completely equivalent ones \eqref{APstab} and \eqref{Lstab}. In particular in the limit $\sigma \rightarrow 0$, they permit to get the following problems: Find $(u^{%
\eps},\xi ^{\eps})\in {\mathcal{V}}\times {\mathcal{A}}$ solution of 
\begin{equation}
(AP_{\mathcal{A}})^{\eps}\,\,\,\left\{ 
\begin{array}{ll}
\ds a(u^{\eps},v)+(1-\eps)a_{\Vert }(\xi ^{\eps},v)=(f,v), & \quad \forall
v\in \mathcal{V} \\[3mm] 
\ds a_{\Vert }(u^{\eps},w)-\eps a_{\Vert }(\xi ^{\eps},w)=0, & \quad \forall
w\in \mathcal{A}\,.%
\end{array}%
\right.  \label{APA}
\end{equation}%
which is equivalent to the original problem (\ref{PPP}) and hence also to the inflow AP-reformulation \eqref{Pa}. In the present case, we fix the Lagrangian variable $\xi ^{\eps}$ by requiring zero average along the field lines, {\it i.e.} $\xi ^{\eps}\in {\mathcal{A}}$, in the former case we fixed the corresponding Lagrangian variable $q^\eps$ by setting $q^\eps$ zero on the inflow boundary $\Gamma_{in}$, {\it i.e.} $q ^{\eps}\in {\mathcal{L}}_{in}$. Note that we do not want here to discretize the space $\mathcal{A}$ directly.  This space arises only in the limit $\sigma\to 0$, which is never performed in practice when one implements the scheme of this paper.  On the contrary, the scheme from  \cite{DDLNN} relies on a direct discretization of $\mathcal{A}$ which results in a rather complicated  numerical method. Remark also that we abandoned in (\ref{APA}) the requirement that the $\xi$-variable has to belong to $L^2(\Omega)$, as there is no more need, for $\sigma=0$.\\

Letting now formally $\eps \rightarrow 0$ in  \eqref{APA}, we  obtain the
problem: Find $(u^{0},\xi ^{0})\in {\mathcal{V}}\times \tilde{\mathcal{A}}$ such that
\begin{equation}
(L_{\mathcal{A}})\,\,\,\left\{ 
\begin{array}{ll}
\ds a(u^{0},v)+a_{\Vert }(\xi ^{0},v)=(f,v), & \quad \forall v\in \mathcal{V}
\\[3mm] 
\ds a_{\Vert }(u^{0},w)=0, & \quad \forall w\in \tilde{\mathcal{A}}\,,%
\end{array}%
\right.  \label{LA}
\end{equation}%
which is an equivalent (saddle-point) reformulation of the original limit problem (%
\ref{LL}).\\

For the reader convenience, we draw in Figure \ref{ima1} a scheme, with all the problems we introduced so far, and their relations. In the following Lemmata and Theorems, we shall prove some of these relations and convergences, adapting the results from the previous section \ref{SEC22} to the present case containing two parameters, $\eps$ and $\sigma$. 
\begin{figure}[h] 
\begin{center} 
\input{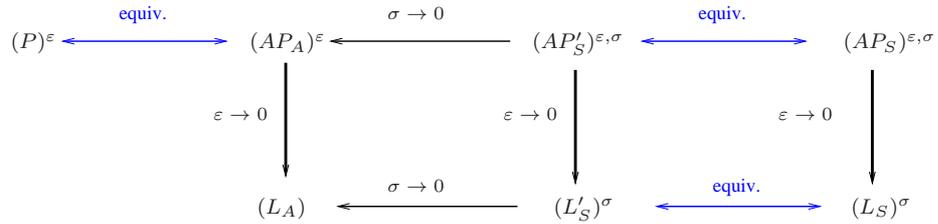}
\end{center} 
\caption{\label{ima1} Stabilized reformulations of the original problem $(P)^\eps$.} 
\end{figure}

\begin{lemma}
\label{infsupsig} \textbf{(Inf-Sup condition)} Let $V$, $L$, $\tilde{L}$, $\hat L$ be Hilbert spaces 
such that $L\subset\tilde{L}$ and $L\subset\hat L$ with continuous inclusions and  $\|\xi\|_{\tilde{L}}\leq \|\xi\|_L$ for all $\xi\in L$.
Let $a(\cdot ,\cdot )$, $c(\cdot ,\cdot )$, $d(\cdot ,\cdot )$ denote the scalar products on respectively $V$, $\tilde{L}$, $\hat L$
and $b(\cdot ,\cdot ):\tilde{L}\times V\rightarrow \mathbb{R}$
be a bilinear form   satisfying 
$\|b(\xi,v)\|\leq \|v\|_{V}\|\xi\|_{\tilde{L}}$
for all $v\in V$, $\xi\in L$ as well as the inf-sup condition
\begin{equation}\label{infsup_b_bis}
\inf_{\xi\in \tilde{L}}\sup_{v\in V}\frac{b(\xi,v)}{\|\xi\|_{\tilde{L}}\|v\|_{V}}%
=\alpha >0.
\end{equation}%
Define furthermore the Hilbert space 
$X_{\eps,\sigma }$ for $\eps\geq 0$, $\sigma \geq 0$ by 
\begin{equation*}
X_{\eps,\sigma }:=\left\{ 
\begin{array}{l}
\ds V\times L\text{, if }\eps>0,\sigma \geq 0 \\[2mm] 
\ds V\times (\tilde{L}\cap \hat{L})\text{, if }\eps=0,\sigma >0 \\[2mm] 
\ds V\times \tilde{L}\text{, if }\eps=0,\sigma =0%
\end{array}%
\right.\,,
\end{equation*}%
and equip it with the norm $\|u,\xi \|_{X_{\eps,\sigma }}:=(\|u\|_{V}^{2}+\|\xi \|_{%
\tilde{L}}^{2}+\eps\|\xi \|_{L}^{2}+\sigma \|\xi \|_{\hat{L}}^{2})^{1/2}$. 

For any $\eps \ge 0$ and $\sigma \geq 0$ let $C_{\eps,\sigma }:X_{\eps,\sigma}\times X_{\eps,\sigma}\rightarrow \mathbb{R}$
be the bilinear form  defined by 
\begin{equation*}
C_{\eps,\sigma}((u,\xi ),(v,w))=a(u,v)+(1-\eps)b(\xi ,v)+b(w,u)-\eps c(\xi ,w)-\sigma
d(\xi ,w)\,.
\end{equation*}%
Then $C_{\eps,\sigma}$ is continuous and satisfies the inf-sup condition 
\begin{equation}
\inf_{(u,\xi )\in {{X}_{\eps{,\sigma }}}}\sup_{(v,w)\in {{X}_{\eps,\sigma }}}%
\frac{C_{\eps,\sigma }((u,\xi ),(v,w))}{\|u,\xi \|_{{{X}_{\eps,\sigma }}}\|v,w\|_{{%
{X}_{\eps,\sigma }}}}\geq \beta \,,  \label{infsup-abs}
\end{equation}%
with a constant $\beta>0 $ that depends only on $\alpha $.
\end{lemma}

\begin{proof} The proof of this lemma follows the same lines as that of
Lemma \ref{infsup} and we give here only a short version of it.
For any $(u,\xi )\in X{_{\eps,\sigma}}$, denoting 
\[
Z:=\sup_{(v,w)\in{{X}_{\eps,\sigma}}}\frac{C_{\eps,\sigma}((u,\xi ),(v,w))}{\|v,w\|_{{{X}_{\eps,\sigma}}}}\,,
\]
we can prove  that $
(1-\eps)\alpha\|\xi\|_{\tilde{{L}}}\leq Z+\|u\|_{{V}}$.
Now, taking $v=u$, $w=-\xi $ we observe that
\begin{eqnarray*}
C_{\eps,\sigma}((u,\xi ),(u,-\xi )) & = & a(u,u)-\eps b(\xi,u)+\eps c(\xi,\xi)+\sigma d(\xi ,\xi )\\
 & \geq & (1-\frac{\eps}{2})\|u\|_{V}^{2}+\frac{\eps}{2}\|\xi\|_{L}^{2}+\sigma\|\xi\|_{\hat{L}}^{2}\,,
\end{eqnarray*}
implying altogether 
\begin{equation*}
\frac{1}{2}\|u\|_{V}^{2}+\frac{\eps}{2}\|\xi\|_{L}^{2}+\frac{\alpha^{2}(1-\eps)^{2}}{8}\|\xi\|_{\tilde{{L}}}^{2}+\sigma\|\xi\|_{\hat{L}}^{2} 
\le
 Z\|u,-\xi\|_{X{_{\eps,\sigma}}}+\frac{1}{4}Z^{2}+\frac{1}{4}\|u\|_{V}^{2}\,.
\end{equation*}
Following again the inequalities from the proof of Lemma \ref{infsup}, we see that there exists 
a constant $c_{0}\in(0,\frac{1}{4}]$ depending only on $\alpha$ such that for any $\gamma>0$ and
$\eps\in[0,1]$ 
\[
c_{0}\|u,\xi\|_{X{_{\eps,\sigma}}}^{2}\,\leq\frac{1}{2\gamma}\|u,\xi\|_{X{_{\eps,\sigma}}}^{2}+\frac{1+\gamma}{2}Z^{2}\,.
\]
Taking finally a sufficiently big $\gamma$ yields $\|u,\xi\|_{{{X}_{\eps}}}\leq(1/\beta)Z$
with a constant $\beta>0$ depending only on $\alpha$. 
\end{proof}

\begin{lemma} \label{LemContSig} \textbf{(Existence/Uniqueness for $\eps \ge 0$ and $\sigma>0$)} Let hypothesis A be satisfied. The stabilized AP-problem
$(AP_{\mathcal{S}})^{\eps,\sigma}$ (resp. $(L_{\mathcal{S}})^{\sigma}$)
is well-posed for each $\eps\in(0,1]$ and $\sigma>0$ (resp. $\eps=0, \sigma>0$),
i.e. for any $f\in\mathcal{V}^{\prime}$ there exists a unique solution
$(u^{\eps,\sigma},\xi ^{\eps,\sigma})\in\mathcal{V}\times \mathcal{L}$
(resp.
$(u^{0, \sigma}, \xi^{0, \sigma}) \in \mathcal{V} \times \widetilde{\mathcal{L}}^2$), which satisfies
\begin{equation}
  \label{est} \|u^{\eps, \sigma}, \xi^{\eps, \sigma} \|_{\mathcal{X}_{\eps, \sigma}} \le C
  \|f\|_{\mathcal{V}'} \hspace{0.25em},
\end{equation}
with $\|u^{}, \xi^{} \|_{\mathcal{X}_{\eps, \sigma}}$:=$(|u|_{\mathcal V}^2 + |\xi|^2_{\ast}
+  \eps|\xi|_{\| }^2 + \sigma \|\xi\|_{L^2}^2)^{1 / 2}$ and some $C > 0$ independent on $\eps$ and $\sigma$. 
Moreover we have the $\eps$-convergence
$$
||u^{\eps, \sigma}-u^{0, \sigma},\xi^{\eps, \sigma}- \xi^{0, \sigma} ||_{\mathcal{X}_{\eps, \sigma}} \rightarrow 0\quad \textrm{for} \quad \eps \rightarrow 0\,.
$$
\end{lemma}
\begin{proof}
The existence and uniqueness of the solution to the reformulated problems
$(AP'_{\mathcal{S}})^{\eps, \sigma}$ and $(L'_{\mathcal{S}})^{\sigma}$ follows
directly from Lemma \ref{infsupsig}  by setting $V = \mathcal{V}$, $L =
\mathcal{A}$, $\tilde{L} = \mathcal{\tilde{A}}$, $\hat{L} = L^2 (\Omega)$.
Now, the equivalence of $(AP_{\mathcal{S}})^{\eps, \sigma}$ and
$(AP'_{\mathcal{S}})^{\eps, \sigma}$ is easily seen from the decomposition
$\mathcal{L} = \mathcal{G_L} \oplus^{\perp} \mathcal{A}$. Similarly, the
equivalence of $(L_{\mathcal{S}})^{\sigma}$ and $(L'_{\mathcal{S}})^{\sigma}$ can be derived from the decomposition $\tilde{\mathcal{L}} = \mathcal{G_L}
\oplus^{\perp} \widetilde{\mathcal{A}}$.
\end{proof}

\begin{theorem}
\label{qepssig} \textbf{(Existence/Uniqueness for $\eps \ge 0$ and $\sigma=0$)} 
Let
hypothesis A be satisfied. The $(AP_{\mathcal{A}})^{\eps}$-problem (\ref{APA}) (resp. $(L_{\mathcal{A}})$-problem (\ref{LA})) is well-posed for each $\eps\in(0,1]$ (resp. $\eps=0$), i.e. for any $f\in\mathcal{V}^{\prime }$ and any $\eps\in[0,1]
$ there exists a unique solution $(u^\eps,\xi^\eps) \in {\mathcal{X}_{\eps,0}}$,
which satisfies 
\begin{equation*} 
\|u^\eps,\xi^\eps\|_{\mathcal{X}_{\eps,0}} \le C \|f\|_{\mathcal{V}%
^{\prime }}\,,
\end{equation*}
with some $C >0$ independent on $\eps$. Furthermore, one has the $\eps$-convergence
\begin{equation*}
\|u^{\eps }-u^{0},\xi^{\eps }-\xi^{0}\|_{{\mathcal{X}_{\eps,0}}}\rightarrow 0\,,
\quad \text{for} \quad \eps \rightarrow 0\,.
\end{equation*}
If we suppose more regular data, as $f \in L^2(\Omega)$, then one has even $\xi^0 \in \mathcal{A}$ and the
estimates 
\begin{equation*}
|u^\eps- u^0|_{\mathcal{V}} \le C \sqrt{\eps}\,, \quad |\xi^\eps- \xi^0|_* \le
C \sqrt{\eps}\,,
\end{equation*}
with $C>0$ some $\eps$-independent constant.
\end{theorem} 
\begin{proof}
The existence and uniqueness of a solution $(u^{\eps},\xi ^{\eps%
})\in {\mathcal{V}}\times {\mathcal{A}}$ to $(AP_{\mathcal{A}})^{\eps}$
resp. $(u^{0},\xi ^{0})\in {\mathcal{V}}\times \tilde{\mathcal{A}}$ to $(L_{%
\mathcal{A}})$ is easily established using Lemma \ref{infsup}. The statements about the convergence as $\eps\to 0$ follow in the same way as in the proof of Theorem \ref{MainThCont}.
\end{proof}

\begin{theorem}
\label{convepssig} \textbf{($\sigma$-Convergence)}
Let hypothesis A be satisfied and moreover, $(u^{\eps,\sigma },\xi^{%
\eps,\sigma })\in {\mathcal{V}} \times {\mathcal{A}}$ be solution to $(AP_{%
\mathcal{S}})^{\eps ,\sigma }$ and $(u^{\eps},\xi^{\eps})\in {\mathcal{V}}
\times {\mathcal{A}}$ solution of $(AP_{\mathcal{A}})^{\eps } $, with $\eps>0$. Suppose
that $\xi^{\eps}\in H^{1}(\Omega )$. Then 
\begin{equation}
  \|u^{\eps,\sigma }-u^{\eps },\xi^{\eps,\sigma }-\xi^{\eps }\|_{\mathcal{X}_{%
      \eps,\sigma }}\leq c\sigma |\xi^{\eps }|_{H^{1}}\,,
  \label{eq:sigma_err}
\end{equation}%
with a constant $c>0$ independent of $\sigma $ and $\eps $.
\end{theorem}

\begin{proof}
We turn now to the convergence as $\sigma\to 0$. Using the combined bilinear form $C_{\eps,\sigma }$ and recalling the problems $(AP_{\mathcal S}')^{\eps,\sigma}$ resp. $(AP_{\mathcal A})^{\eps}$, we can write 
\begin{equation*}
C_{\eps,\sigma }((u^{\eps,\sigma },\xi ^{\eps,\sigma }),(v,w))=(f,v)\,, \quad \forall\,  (v,w)\in \mathcal{V}\times \mathcal{A}\,,
\end{equation*}%
\begin{equation*}
C_{\eps,0}((u^{\eps},\xi ^{\eps}),(v,w))=(f,v)\,, \quad \forall\, (v,w)\in \mathcal{V}\times \mathcal{A}\,,
\end{equation*}%
where%
\begin{equation*}
C_{\eps,\sigma }((u,\xi ),(v,w)):=a(u,v)+(1-\eps)a_{\Vert }(\xi ,v)+a_{\Vert
}(u,w)-\eps a_{\Vert }(\xi ,w)-\sigma (\xi ,w)\,.
\end{equation*}%
Note that $C_{\eps,0}$ coincides with $C_{\eps}$ as defined by (\ref{Ceps}).
Taking the difference gives%
\begin{equation*}
C_{\eps,\sigma }((u^{\eps,\sigma }-u^{\eps},\xi ^{\eps,\sigma }-\xi ^{\eps%
}),(v,w))=\sigma (\xi ^{\eps},w)\leq \sigma |\xi ^{\eps}|_{{\mathcal{%
V}}}|w|_{{\mathcal{V}}^{\prime }}\leq c\sigma |\xi ^{\eps}|_{{\mathcal{V}}%
}|w|_{*}\,.
\end{equation*}%
We have used here the bound $|w|_{{\mathcal{V}}^{\prime }}\leq c|w|_{*}$
valid for $w\in \mathcal{A}$ as proved below (Corollary \ref{H1Lt}).

Now, remind that the form $C_{\eps ,\sigma }$ enjoys the inf-sup property%
\begin{equation*}
\inf_{(u,\xi)\in \mathcal{X}{_{\eps,\sigma }}}\sup_{(v,w)\in \mathcal{X}{_{%
\eps,\sigma }}}\frac{C_{\eps,\sigma }((u,\xi),(v,w))}{\|u,\xi\|_{\mathcal{X}{_{%
\eps,\sigma }}}\|v,w\|_{\mathcal{X}{_{\eps,\sigma }}}}\geq \beta \,,
\end{equation*}%
where $\|u,\xi\|_{\mathcal{X}{_{\eps,\sigma }}}=(|u|_{\mathcal{V}%
}^{2}+|\xi|_{*}^{2}+\eps |\xi|_{\|}^{2}+\sigma \|\xi\|_{L^{2}}^{2})^{1/2}$, so
that $|w|_{*}\leq \|v,w\|_{\mathcal{X}{_{\eps,\sigma }}}$. We can thus
conclude that there exists $(v,w)\in \mathcal{X}_{\eps,\sigma }$ such that $%
\|v,w\|_{\mathcal{X}{_{\eps,\sigma }}}=1$ and%
\begin{equation*}
\beta' \|u^{\eps ,\sigma }-u^{\eps },\xi^{\eps ,\sigma }-\xi^{\eps }\|_{%
\mathcal{X}{_{\eps,\sigma }}}\leq C_{\eps ,\sigma }((u^{\eps ,\sigma }-u^{%
\eps },\xi^{\eps ,\sigma }-\xi^{\eps }),(v,w))\leq c\, \sigma |\xi^{\eps }|_{%
\mathcal{V}}\|v,w\|_{\mathcal{X}{_{\eps,\sigma }}}=c\, \sigma |\xi^{\eps %
}|_{H^{1}}\,,
\end{equation*}
with some $0<\beta'<\beta$, for example $\beta'=\beta/2$. This concludes the proof.
\end{proof}

\begin{remark}
Without the additional hypothesis $\xi ^{\eps}\in H^{1}(\Omega )$, we can
easily prove a sub-optimal estimate%
\begin{equation*}
\|u^{\eps,\sigma }-u^{\eps},\xi ^{\eps,\sigma }-\xi ^{\eps}\|_{\mathcal{X}_{%
\eps,\sigma }}\leq c\sqrt{\sigma }\Vert \xi ^{\eps}\Vert _{L^{2}}\,.
\end{equation*}%
Indeed,%
\begin{equation*}
C_{\eps,\sigma }((u^{\eps,\sigma }-u^{\eps},\xi ^{\eps,\sigma }-\xi ^{\eps%
}),(v,w))=\sigma (\xi ^{\eps},w)\leq \sigma \|\xi ^{\eps%
}\|_{L^{2}(\Omega )}\|w\|_{L^{2}(\Omega )}\leq c\sqrt{\sigma }\|\xi ^{\eps%
}\|_{L^{2}(\Omega )}\|v,w\|_{\mathcal{X}{_{\eps,\sigma }}}\,.
\end{equation*}
\end{remark}

\begin{remark} The conclusions of Theorem \ref{qepssig} remain true (after an obvious rephrasing) in the limit case $\eps=0$ 
since the proof relies on the estimates in the norm of $\mathcal{X}_{\eps,\sigma}$ which remains a valid norm in the limit $\eps\to 0$.
\end{remark}

It remains to prove the bound $|w|_{{\mathcal V}^{\prime }}\leq c\,|w|_{*}$ valid for $%
w\in \mathcal{A}$. This will be done using the following result:

\begin{lemma} \label{lleemm}
Let $u\in \mathcal{V}$ and consider $v\in \mathcal{A}$ being the unique
solution to%
\begin{equation}
a_{\|}(v,w)=(u,w),\quad \forall w\in \mathcal{A.}  \label{aparv}
\end{equation}%
Then $v\in H^{1}(\Omega )$ and there exists a constant $c>0$ such that $%
\|v\|_{H^{1}}\leq c\, |u|_{\mathcal{V}}.$
\end{lemma}

\begin{proof}
To simplify the notations, let us restrict ourselves to the 2D case in this
proof (the extension to $d > 2$ is rather straightforward). There is an evident bound
$|| \nabla_{\|} v ||_{L^2} \leq c  || u ||_{L^2}$ which implies
$|| v ||_{L^2} \leq C ||u||_{L^2}$ by a Poincar{\'e} type inequality
{\cite{DDLNN}}. To continue, let us change the coordinates on $\Omega$ and suppose that there exist new
coordinates $(\xi_1, \xi_2)$ so that $\Omega$ becomes the unit square
$\Omega_{\xi} = (0, 1)^2$ and $\nabla_{||}$ becomes $\alpha (\xi_1, \xi_2)
\frac{\partial}{\partial \xi_2}$ with some positive function $\alpha$.
Problem (\ref{aparv}) is written in these new coordinates as
\begin{equation*}
\int_{\Omega _{\xi }}N\, \frac{\partial v}{\partial \xi _{2}}\frac{\partial w}{%
\partial \xi _{2}}\, d\xi_1d\xi_2=\int_{\Omega _{\xi }}J\, uw\, d\xi_1d\xi_2\,,
\end{equation*}%
where $J=J(\xi _{1},\xi _{2})$ is the Jacobian and $N=N(\xi _{1},\xi
_{2})=A_{||}\, J\, \alpha ^{2}$, which are positive functions given by the geometry. 

Let us now replace here $w$ by $\frac{\partial \omega}{\partial \xi_1}$
with arbitrary and sufficiently smooth function $\omega$ such that
$\omega  = 0$ at $\xi_1 = 0$ and at $\xi_1 =
1$. Integration by parts with respect to $\xi_1$ yields then
\[ \int_{\Omega_{\xi}} \frac{\partial N}{\partial \xi_1}  \frac{\partial
   v}{\partial \xi_2}  \frac{\partial \omega }{\partial \xi_2} \, d\xi_1d\xi_2+
   \int_{\Omega_{\xi}} N \frac{\partial^2 v}{\partial \xi_1 \partial \xi_2} 
   \frac{\partial \omega }{\partial \xi_2} \, d\xi_1d\xi_2= \int_{\Omega_{\xi}}
   \frac{\partial (Ju)}{\partial \xi_1} \,\omega \, d\xi_1d\xi_2. \]
   
Noting that $\omega $ is not differentiated in the last formula wrt $\xi_1$ we can
use density arguments and extend this relation to a broader class of test
functions $\omega $, not necessarily vanishing at $\xi_1 = 0, 1$. In particular, we
can now set $\omega  = \frac{\partial v}{\partial \xi_1}$ and get
\begin{equation*}
\int_{\Omega _{\xi }}N\, \left( \frac{\partial ^{2}v}{\partial \xi _{1}\partial
\xi _{2}}\right) ^{2}\, d\xi_1d\xi_2=\int_{\Omega _{\xi }}\frac{\partial (Ju)}{\partial \xi
_{1}}\frac{\partial v}{\partial \xi _{1}}\, d\xi_1d\xi_2-\int_{\Omega _{\xi }}\frac{%
\partial N}{\partial \xi _{1}}\frac{\partial v}{\partial \xi _{2}}\frac{%
\partial ^{2}v}{\partial \xi _{1}\partial \xi _{2}}\, d\xi_1d\xi_2\,.
\end{equation*}
This, reminding $\left\| \frac{\partial v}{\partial \xi_2} \right\|_{L^2} \leq
c || \nabla_{\|} v ||_{L^2} \leq c  || u ||_{L^2}$, entails by
Young inequality
\begin{equation}\label{vmixed}
  \left\| \frac{\partial^2 v}{\partial \xi_1 \partial \xi_2} \right\|^2_{L^2}
   \leq c \gamma  || u ||^2_{H^1} + \frac{c}{\gamma}
     \left\| \frac{\partial v}{\partial \xi_1} \right\|^2_{L^2}
   + c || u ||^2_{L^2}  \,,
\end{equation}
with a fixed constant $c > 0$ and arbitrary $\gamma> 0$.

Applying a Poincar{\'e} type inequality to $J \frac{\partial v}{\partial
\xi_1}$, we can write $\forall \xi_1 \in (0, 1)$
\begin{equation}\label{poinx2}
  \int_0^1 \left( \frac{\partial v}{\partial \xi_1} \right)^2 d \xi_2 \leq C
  \int_0^1 \left( \frac{\partial^2 v}{\partial \xi_1 \partial \xi_2} \right)^2
  d \xi_2 + C \left( \int_0^1 J\, \frac{\partial v}{\partial \xi_1}\, d \xi_2
  \right)^2\,.
\end{equation}
Remind that $v \in \mathcal{A}$, which means
\[ \int_0^1 J (\xi_1, \xi_2) \,v (\xi_1, \xi_2)\, d \xi_2 = 0 \hspace{1em} \forall
   \xi_1 \in (0, 1)\,, \]
or, after differentiation wrt $\xi_1$,
\[ \int_0^1 J \,\frac{\partial v}{\partial \xi_1} \,d \xi_2 + \int_0^1
   \frac{\partial J}{\partial \xi_1}\, v\, d \xi_2 = 0 \hspace{1em} \forall \xi_1
   \in (0, 1)\,. \]
Relation (\ref{poinx2}) can be now rewritten as
\[ \int_0^1 \left( \frac{\partial v}{\partial \xi_1} \right)^2 d \xi_2 \leq C
   \int_0^1 \left( \frac{\partial^2 v}{\partial \xi_1 \partial \xi_2}
   \right)^2 d \xi_2 + C \left( \int_0^1 \frac{\partial J}{\partial \xi_1} \,v\,d
   \xi_2 \right)^2 \,\]
which, after integrating over $\xi_1 \in (0, 1)$, with the aid of (\ref{vmixed}) and 
the the bound $|| v ||_{L^2} \leq C ||u||_{L^2}$, gives
\[ \left\| \frac{\partial v}{\partial \xi_1} \right\|^2_{L^2} \leq C \left\|
   \frac{\partial^2 v}{\partial \xi_1 \partial \xi_2} \right\|^2_{L^2} + C
   \|v\|^2_{L^2} \leq Cc \gamma  || u ||^2_{H^1} +
   \frac{Cc}{\gamma}   \left\| \frac{\partial v}{\partial
   \xi_1} \right\|^2_{L^2} + \tilde{C}  || u ||^2_{H^1} . \]
This implies, taking $\gamma$ sufficiently big. \[ \left\| \frac{\partial
   v}{\partial \xi_1} \right\|^{}_{L^2} \leq C  || u ||_{H^1}\,,
\] which gives the desired result since, as already noted, $\left\| \frac{\partial v}{\partial \xi_2} \right\|_{L^2} 
 \leq c  || u ||_{L^2}$.
\end{proof}

\begin{corollary}
\label{H1Lt} Let $\xi\in \mathcal{A}$. Then one has $|\xi|_{{\mathcal{V}}%
^{\prime }}\leq c\, |\xi|_{*}$ with some constant $c>0$.
\end{corollary}

\begin{proof}
One can immediately see that $|\xi|_{{\mathcal{V}}^{\prime }}=|u|_{{\mathcal{%
V}}}$ where $u\in \mathcal{V}$ solves 
\begin{equation} \label{iiiii}
(\nabla u,\nabla w)=(\xi,w),\quad \forall w\in \mathcal{V.}
\end{equation}%
This means in particular that $\xi=-\Delta u$. Let now $v\in \mathcal{A}$ be
the solution to (\ref{aparv}), corresponding to $u$ solution to \eqref{iiiii}. Lemma \ref{lleemm} implies thus that $v \in H^1(\Omega)$ and one has
\begin{equation*}
|\xi|_{{\mathcal{V}}^{\prime }}=|u|_{{\mathcal{V}}}=\frac{(-\Delta u,u)}{%
|u|_{{\mathcal{V}}}}=\frac{(\xi,u)}{|u|_{H^{1}}}=\frac{a_{\|}(v,\xi)}{%
|u|_{H^{1}}}\leq c\, \frac{a_{\|}(\xi,v)}{\|v\|_{H^{1}}}\leq c\, |\xi|_{*}%
\text{.}
\end{equation*}
\end{proof}

\subsection{Numerical analysis for the stabilized AP-scheme}

Let us introduce a mesh ${\mathcal{T}}_{h}$ on $\Omega $ consisting of
triangles (resp. rectangles) of maximal size $h$ and let ${V}_{h}\subset 
\mathcal{V}$ be the space of $\mathbb{P}_{k}$ (resp. $\mathbb{Q}_{k}$)
finite elements on ${\mathcal{T}}_{h}$. We want now to discretize the
stabilized problem (\ref{APstab}) and remark that we can use $V_{h}$ for
both variables $u$ and $\xi$. We are thus looking for a discrete solution $%
(u_{h}^{\eps ,\sigma },\xi_{h}^{\eps
,\sigma })\in {V}_{h}\times {V}_{h}$ of 
\begin{equation}
(AP_{\mathcal{S}})_{h}^{\eps ,\sigma }\,\,\,\left\{ 
\begin{array}{l}
a(u_{h}^{\eps ,\sigma },v_{h})+(1-\eps )a_{\|}(\xi_{h}^{\eps ,\sigma
},v_{h})=(f,v_{h})\,,\quad \forall v_{h}\in {V}_{h} \\[3mm] 
a_{\|}(u_{h}^{\eps ,\sigma },w_{h})-\eps a_{\|}(\xi_{h}^{\eps ,\sigma
},w_{h})-\sigma (\xi_{h}^{\eps ,\sigma },w_{h})=0,\quad \forall w_{h}\in {V}%
_{h}\,.%
\end{array}%
\right.  \label{Phsig0}
\end{equation}

Let us decompose now ${V}_{h}=G_{h}\oplus A_{h}$ with $G_{h}=V_{h}\cap 
\mathcal{G}=V_{h}\cap 
\mathcal{G}_{\mathcal L}$ and $A_{h}$ being the $L^{2}-$orthogonal complement of $G_{h}$.
Taking test functions from $G_{h}$ in the second equation of (\ref{Phsig0}),
we see that $\xi_{h}^{\eps ,\sigma }\in A_{h}$ so that this problem can be
in fact equivalently rewritten as: Find $(u_{h}^{\eps ,\sigma },\xi_{h}^{%
\eps ,\sigma })\in {V}_{h}\times {A}_{h}$ such that%
\begin{equation}
(AP_{\mathcal{S}}')_{h}^{\eps ,\sigma }\,\,\,\left\{ 
\begin{array}{l}
a(u_{h}^{\eps ,\sigma },v_{h})+(1-\eps )a_{\|}(\xi_{h}^{\eps ,\sigma
},v_{h})=(f,v_{h})\,,\quad \forall v_{h}\in {V}_{h} \\[3mm] 
a_{\|}(u_{h}^{\eps ,\sigma },w_{h})-\eps a_{\|}(\xi_{h}^{\eps ,\sigma
},w_{h})-\sigma (\xi_{h}^{\eps ,\sigma },w_{h})=0,\quad \forall w_{h}\in {A}%
_{h}\,.%
\end{array}%
\right.  \label{Phsig}
\end{equation}%
The advantage of the last reformulation is purely analytical, as we can now
reintroduce the mesh dependent norm on $A_{h}$%
\begin{equation*}
|\xi_{h}|_{*h}=\sup_{v_{h}\in {V}_{h}}\frac{a_{\|}(\xi_{h},v_{h})}{%
|v_{h}|_{\mathcal{V}}}\,.
\end{equation*}%
We now equip the space $X_{h}:={V}_{h}\times {A}%
_{h}$ with the norm $\|u_{h},\xi_{h}\|_{X_{\eps ,\sigma
,h}}:=(|u_{h}|_{V}^{2}+|\xi_{h}|_{*h}^{2}+\eps
|\xi_{h}|_{\|}^{2}+\sigma |\xi_{h}|_{L^{2}}^{2})^{1/2}$. By Lemma \ref%
{infsupsig}, the bilinear form $C_{\eps ,\sigma }$ is continuous on $(X_{h}, ||\cdot,\cdot||_{X_{\eps ,\sigma
,h}})$ and enjoys the inf-sup property%
\begin{equation}
\inf_{(u_{h},\xi_{h})\in X_{h}}\sup_{(v_{h},w_{h})\in X_{h}}\frac{C_{\eps %
,\sigma }((u_{h},\xi_{h}),(v_{h},w_{h}))}{|u_{h},\xi_{h}|_{X_{\eps ,\sigma
,h}}|v_{h},w_{h}|_{X_{\eps ,\sigma ,h}}}\geq \beta \,,  \label{infsuphsig}
\end{equation}%
with a constant $\beta >0$ that does not depend neither on the mesh nor on $%
\eps $ and $\sigma$. This implies

\begin{theorem}
\textbf{(Discrete Existence/Uniqueness/$\sigma,\eps$-Convergences)} The
discrete AP-problem (\ref{Phsig0}) admits a unique solution $(u_{h}^{\eps%
,\sigma },\xi_{h}^{\eps,\sigma })\in {V}_{h}\times {V}_{h}$, satisfying 
\begin{equation*}
\|u_{h}^{\eps,\sigma },\xi_{h}^{\eps,\sigma }\|_{X_{\eps,\sigma,h}}\leq \frac{1%
}{\beta} \|f\|_{\mathcal{V}^{\prime }}\,.
\end{equation*}%
Moreover, for any $\varepsilon \geq 0$ fixed, one has the convergences 
\begin{equation*}
u_{h}^{\eps,\sigma }\rightarrow _{\sigma \rightarrow 0}u_{h}^{\eps%
,0}\,;\quad \quad \xi_{h}^{\eps,\sigma }\rightarrow _{\sigma \rightarrow
0}\xi_{h}^{\eps,0} \quad \,,
\end{equation*}%
where $(u_{h}^{\eps,0},\xi_{h}^{\eps,0})\in V_{h}\times A_{h}$ is the unique
solution to (\ref{Phsig}) with $\sigma =0.$ We also have 
\begin{equation*}
u_{h}^{\eps,0}\rightarrow _{\eps\rightarrow 0}u_{h}^{0,0}\,;\quad \quad
\xi_{h}^{\eps,0}\rightarrow _{\eps\rightarrow 0}\xi_{h}^{0,0}\quad \,,
\end{equation*}%
where $(u_{h}^{0,0},\xi_{h}^{0,0})\in V_{h}\times A_{h}$ is the unique
solution to (\ref{Phsig}) with $\varepsilon =\sigma =0.$

The condition number of the matrix corresponding to problem (\ref{Phsig0}) is bounded
by a constant that depends on $\sigma$ but not on $\eps$ (assuming that the same bases
of $V_{h}$ and $L_h$ are chosen for all values of $\eps,\sigma$).
\end{theorem}

\begin{proof}
The existence and uniqueness of a solution $(u_{h}^{\eps,\sigma },\xi_{h}^{%
\eps,\sigma })\in V_{h}\times V_{h}$ for each $\eps\geq 0$, $\sigma >0$ is a
simple consequence of the BNB theorem \cite{ern}. As mentioned already this
solution lies in fact in $V_{h}\times A_{h}$ and it is thus also the
solution to (\ref{Phsig}). By the same arguments, the latter problem admits
a unique solution $(u_{h}^{\eps,\sigma },\xi_{h}^{\eps,\sigma })\in
V_{h}\times A_{h}$ also in the case $\sigma =0$. To prove the convergence $%
(u_{h}^{\eps,\sigma },\xi_{h}^{\eps,\sigma })\rightarrow (u_{h}^{\eps%
,0},\xi_{h}^{\eps,0})$ for $\eps\geq 0$ fixed, we observe%
\begin{align*}
&C_{\eps,\sigma }((u_{h}^{\eps,\sigma }-u_{h}^{\eps,0},\xi_{h}^{\eps,\sigma
}-\xi_{h}^{\eps,0}),(v_{h},w_{h})) =\sigma (\xi_{h}^{\eps,0},w_{h}) \leq
\sigma \|\xi_{h}^{\eps,0}\|_{L^{2}}\|w_{h}\|_{L^{2}} \\
&\qquad \leq \sqrt{\sigma }\,\|\xi_{h}^{\eps,0}\|_{L^{2}}\|v_{h},w_{h}\|_{X_{\eps%
,\sigma ,h}}, 
 \hspace{4cm} \forall (v_{h},w_{h})\in {V}_{h}\times {A}_{h}\,,
\end{align*}%
and we conclude using the discrete inf-sup property for $C_{\eps,\sigma }$.
The proof of the other convergence $\varepsilon \rightarrow 0$ while $\sigma
=0$ is done exactly in the same way as in Theorem \ref{APproph}.

We turn now to the study of the condition number. We recall the notations from the proof of Theorem \ref{APproph} with the only change that there is no longer the space $L_h$, which has been replaced by $V_h$. 
In particular, the constants $\mu_q$, $\nu_q$, $\mu_*$, $\nu_*$ are now evaluated on $V_h$ instead of $L_h$ and one can have $\mu_q=\mu_*=0$.
Denoting by $\hat{\mu}$ and $\hat{\nu}$ the minimal and maximal
eigenvalues of the mass matrix $(\phi^u_i, \phi^u_j)_{L^2 (\Omega)}$ we
conclude for any $\eps, \sigma \ge 0$ and any $\Phi_h=(u_h,\xi_h) \in V_h \times V_h$
\[ \min (\mu_u^2, \sigma \hat{\mu}^2) ||  \vec{\Phi} ||^2_2 \leq
   || \Phi_h ||^2_{X_{\eps, \sigma, h}} \leq \max \{\nu_u^2, \nu_{\ast}^2 +
   \eps \hspace{0.17em} \nu_q^2 + \sigma \hat{\nu}^2 \} ||  \vec{\Phi} ||_2
   \,.
\]
Introducing the matrix of problem (\ref{Phsig0}), denoted by $\mathbb{A}^{\eps, \sigma}$, and repeating the calculations of Theorem \ref{APproph} we obtain
\[ || \mathbb{A}^{\eps, \sigma} ||_2 \leq M \max \{\nu_u^2, \nu_{\ast}^2 + \eps
   \nu_q^2 + \sigma \hat{\nu}^2 \}\,, \qquad  \beta' \min
   (\mu_u^2, \sigma \hat{\mu}^2) ||  \vec{\Phi} ||_2 \leq ||
   \mathbb{A}^{\eps, \sigma}  \vec{\Phi} ||_2\,, \]
for any $\vec{\Phi} \in \mathbb{R}^N$, so that one has finally
\[ cond_2 (\mathbb{A}^{\eps, \sigma}) = || \mathbb{A}^{\eps, \sigma} ||_2 || (\mathbb{A}^{\eps,
   \sigma})^{- 1} ||_2 \leq \frac{M \max \{\nu_u^2, \nu_{\ast}^2 + \eps
   \nu_q^2 + \sigma \hat{\nu}^2 \}}{\beta \min (\mu_u^2,
   \sigma \hat{\mu}^2)}\,, \]
which is an $\eps$-independent bound.
\end{proof}

\begin{remark}\label{rmk:CondEstSigma} 
As already observed in Remark \ref{CondEst} we have 
  \begin{equation*}
    \hat\mu\sim \hat\nu\sim h,\quad
    \mu _{u}\sim h\text{ and }\nu _{u}\sim \nu _{q} \sim \nu _{*}\sim 1\,.
  \end{equation*}%
  Hence, assuming $\eps,\sigma\in [0,1]$, one obtains
  \begin{equation*}
    cond_{2}(\mathbb{A}^{\varepsilon,\sigma})\lesssim \frac{1}{\sigma h^{2}}\,.
  \end{equation*}
\end{remark}

\begin{theorem} \textbf{($h$-Convergence)}
\label{thm:erestsig}
Let $k\geq 1$ and $V_{h}$ be the $P_{k}$ or $Q_{k}$ finite element space on
a regular mesh ${\mathcal{T}}_{h}$. Suppose moreover that problem (\ref%
{APstab}) has the solution $u^{\eps ,\sigma }\in H^{k+1}(\Omega )$, $\xi^{%
\eps ,\sigma }\in H^{k+1}(\Omega )$ and problem (\ref{APA}) has a solution $%
u^{\eps }\in H^{1}(\Omega )$, $\xi^{\eps }\in H^{1}(\Omega )$. Then 
\begin{equation}
|u^{\eps}-u_{h}^{\eps,\sigma }|_{H^{1}}\leq Ch^{k}(|u^{\eps ,\sigma
}|_{H^{k+1}}+|\xi^{\eps ,\sigma }|_{H^{k+1}})+C\sigma |\xi^{\eps }|_{H^{1}}\,,
\label{erestsig}
\end{equation}%
with a constant $C>0$ that depends neither on the mesh, nor on $\eps$ or $%
\sigma $.
\end{theorem}

\begin{proof}
In the same way as in the inflow case we prove that 
\begin{equation*}
|u^{\eps,\sigma }-u_{h}^{\eps,\sigma }|_{H^{1}}\leq Ch^{k}(|u^{\eps %
}|_{H^{k+1}}+|\xi^{\eps }|_{H^{k+1}}).
\end{equation*}%
It remains to invoke Lemma \ref{qepssig} and the triangle inequality to
conclude.
\end{proof}

\begin{remark} \label{Rem22}
The error estimate (\ref{erestsig}) would be of course useless if the norms $%
|u^{\eps ,\sigma }|_{H^{k+1}}$, $|\xi^{\eps ,\sigma }|_{H^{k+1}}$ were
dependent on $\eps $ and $\sigma $. Fortunately, it is not the case. We
expect indeed that $|u^{\eps ,\sigma }|_{H^{k+1}}$ is bounded uniformly in $%
\eps $ by the norm of $f$ in $H^{k-1}(\Omega )$ and $|\xi^{\eps ,\sigma
}|_{H^{k+1}}$ is bounded uniformly in $\eps $ by the norm of $f$ in $%
H^{k+1}(\Omega )$. This can be easily proved in the case of simple aligned
geometry, see Appendix A. 
\end{remark}
\begin{remark}
One can also easily obtain
\begin{equation*}
|\xi^{\eps,\sigma }-\xi_{h}^{\eps,\sigma }|_{H^{1}}\leq \frac{C}{\sqrt\eps}\left[
(|u^{\eps ,\sigma}|_{H^{k+1}}+|\xi^{\eps ,\sigma }|_{H^{k+1}})+\sigma |\xi^{\eps }|_{H^{1}}
\right]
\end{equation*}
which degenerates as in the inflow case, as $\eps$ goes to 0. Again, we are not sure, if this estimate is sharp, but we recall that $\xi^{\eps,\sigma }$ is an auxiliary variable, without any intrinsic meaning.
\end{remark}
%%%%%%%%%%%%%%%%%%%%%%%%%%%%%%%%%%%%%%%%%%%%%%%%

%%%%%%%%%%%%%%%%%%%%%%%%%%
\section{Numerical tests}\label{sec:num_tests}
%%%%%%%%%%%%%%%%%%%%%%%%%%%%

Let us now study numerically both AP-reformulations, the inflow as well as the stabilized one. We consider in the following a
square computational domain $\Omega =[0,1]\times [0,1]$ and the non-uniform and not coordinate-aligned $b$ field:
\begin{gather}
  b = \frac{B}{|B|}\, , \quad
  B =
  \left(
    \begin{array}{c}
      \alpha  (2y-1) \cos (\pi x) + \pi \\
      \pi \alpha  (y^2-y) \sin (\pi x)
    \end{array}
  \right)\,,
  \label{eq:J99a} 
\end{gather}
as well as a sample function $u^0$ which is constant in the direction of the $b$-field:
\begin{gather}
  u^{0} = \sin \left(\pi y +  \alpha (y^2-y)\cos (\pi x) \right)
  \label{eq:J79a}.
\end{gather}
Here $\alpha \ge 0$ is a parameter to be fixed in the following different test cases and describes the variations of $b$. We choose $u^0$ to be the $\eps \rightarrow 0$ limit solution of the anisotropic problem $(P)^\eps$, hence solution of \eqref{LL},  and construct an exact solution of \eqref{PPP} by adding a perturbation proportional to $\varepsilon $, {\it i.e.}
\begin{gather}
  u^{\varepsilon } = \sin \left(\pi y +\alpha (y^2-y)\cos (\pi
  x) \right) + \varepsilon \cos \left( 2\pi x\right) \sin \left(
  \pi y+  \alpha (y^2-y)\cos (\pi x) 
  \right)\,.
  \label{eq:Jc0a}
\end{gather}
Note that the
auxiliary variable $q^\varepsilon $, solution of \eqref{Pa}, is in this case equal to
\begin{gather}
  q^\varepsilon = \cos \left( 2\pi x\right) \sin \left(
  \pi y+  \alpha (y^2-y)\cos (\pi x)
  \right)
  -
  \sin \left(
  \pi y+  \alpha (y^2-y)\cos (\pi x)
  \right)
  \label{eq:Jffb}.
\end{gather}

Finally, we compute the right hand side accordingly and have thus constructed an exact solution of problem \eqref{PPP}. All simulations
(unless stated otherwise) were performed using a $\mathbb Q _2$ finite
element method.\\

Aim  of this section is to study and validate from a numerical point of view the error estimates established in the last two sections. In particular, we investigate firstly the error introduced by the stabilization procedure
in the $(AP_{\mathcal S})^{\varepsilon ,\sigma }$ formulation, meaning the $\sigma$-convergence estimate of (\ref{eq:sigma_err}) in Theorem
\ref{qepssig} is verified numerically. Then the $h$-convergence of
both methods is studied and the estimates (\ref{Err_est}) and
(\ref{erestsig}) are confirmed in both anisotropic ($\eps \ll 1$) and isotropic $(\eps\sim 1)$ regimes. Next, we show that both methods are Asymptotic-Preserving in the parameter $\eps$. 
The conditioning of the corresponding linear systems appear effectively to scale
in agreement with Remarks \ref{CondEst} and \ref{rmk:CondEstSigma}.
Finally, the case of a less regular force term $f$, belonging merely to $L^2(\Omega )$ (and not
to $H^1(\Omega )$) is studied --- the convergence of the schemes is
tested beyond the validity of Theorems \ref{thm:Err_est} and
\ref{thm:erestsig}.

\subsection{Stabilization error ($\eps \ge 0$, $h>0$ fixed, $\sigma \rightarrow 0$)}

Let us start by studying the error introduced by a stabilization term proportional to $\sigma $ in the $(AP_S)^{\varepsilon ,\sigma }$ reformulation, in particular we shall estimate numerically for fixed $\eps \ge 0$ and $h>0$ the $L^2$- resp. $H^1$-errors between the exact solution $u^\eps$ constructed in \eqref{eq:Jc0a} and the numerical stabilized solution $u^{\eps,\sigma}_{h}$, solution of \eqref{APstab} or \eqref{Phsig0}, {\it i.e.} $||u^\eps-u_{h}^{\eps,\sigma}||$. The
mesh size $h$ is fixed to $0.01$.  Numerical simulations are performed
for the stabilization constant $\sigma $ varying from $1$ to $10^{-15}$, considering three different regimes : no anisotropy ($\varepsilon =1$), strong
anisotropy with direction aligned with the coordinate system
($\varepsilon = 10^{-10}$, $\alpha =0$) and strong anisotropy with
variable direction ($\varepsilon =10^{-10}$, $\alpha =2$). The $L^2$- and $H^1$-errors are presented as a function of $\sigma $ in Figure \ref{fig:sigma_error}.\\
In the first regime, with no anisotropy
present in the system, the stabilization constant does not influence
the precision at all. Indeed, it is exactly what is expected as the terms
involving $\xi^{\varepsilon ,\sigma }$ do not appear for $\eps =1$ in the first
equation of $(AP_{\mathcal S})^{\varepsilon ,\sigma }$. In the second
regime, with strong and aligned anisotropy, the precision of the
scheme is influenced by the stabilization procedure only for $\sigma$-values greater than $10^{-5}$ in the $L^2$-norm and greater than
$10^{-3}$ in the $H^1$-norm. 
The error dependence in $\sigma$ is here linear, according to the Theorem \ref{qepssig}, and is explained simply by the fact that the stabilization influences the results for large $\sigma >0$.

For $\sigma $-values  smaller than
these critical values the accuracy of the scheme in both
norms remains unchanged and is given only by the mesh size. This
holds true even if the value of the stabilization constant is close
to the machine precision ($10^{-15}$) and can be explained by the fact that we are in an aligned test-case. Indeed, normally for $\sigma \rightarrow 0$ the error should increase, due to the non-uniqueness of the $\xi$-solution. Here we are however plotting the error corresponding to the $u^\eps$-function, which is uniquely determined. The non-uniqueness of $\xi^{\eps,\sigma}$ steps in only in the not-aligned case, which is our third regime of strong anisotropy with variable
direction. In this case, the curves show an expected $\sigma$-behaviour, the optimal value of $\sigma $ being between $10^{-8}$ and
$10^{-5}$ for the $L^2$-error and between $10^{-10}$ and $10^{-3}$ for
the $H^1$-norm. To explain this, observe that in the limit $\sigma \rightarrow 0$ the auxiliary function $\xi^{\eps,\sigma}$ is uniquely determined up to a constant on the field lines. This constant will normally not interfere in the computation of $u^{\eps,\sigma}$, as only the parallel derivatives of $\xi^{\eps,\sigma}$ are present in the $u^{\eps,\sigma}$-equation. However, if the mesh is not aligned with the field lines, this parallel derivative mixes the directions, introducing errors which lead to the observed behaviour of the error as $\sigma \rightarrow 0$.\\

\def\xxxa{0.49\textwidth}
\begin{figure}
  \centering
  \subfigure[$L^2$ error]
  {\includegraphics[angle=-90,width=\xxxa]{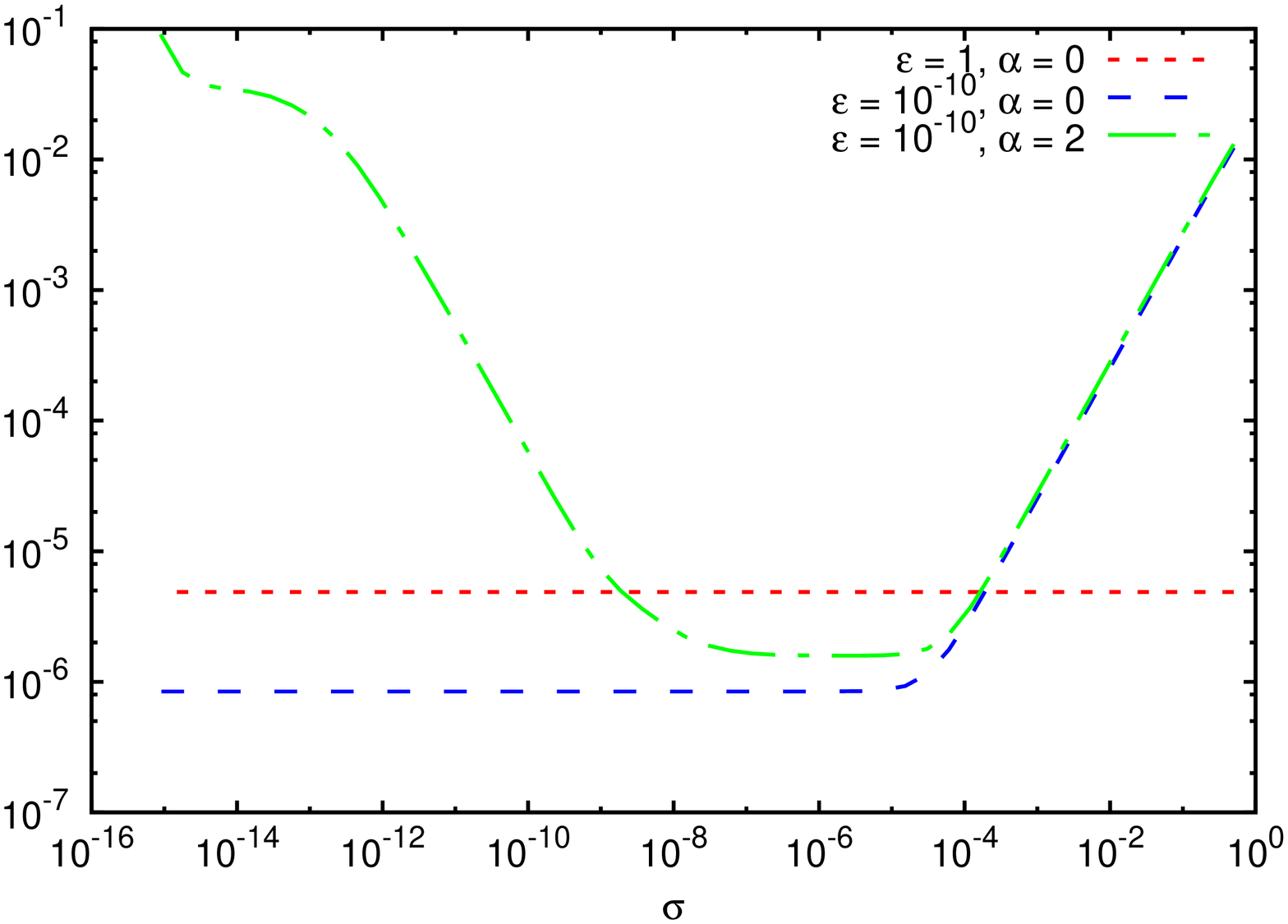}}
  \subfigure[$H^1$ error]
  {\includegraphics[angle=-90,width=\xxxa]{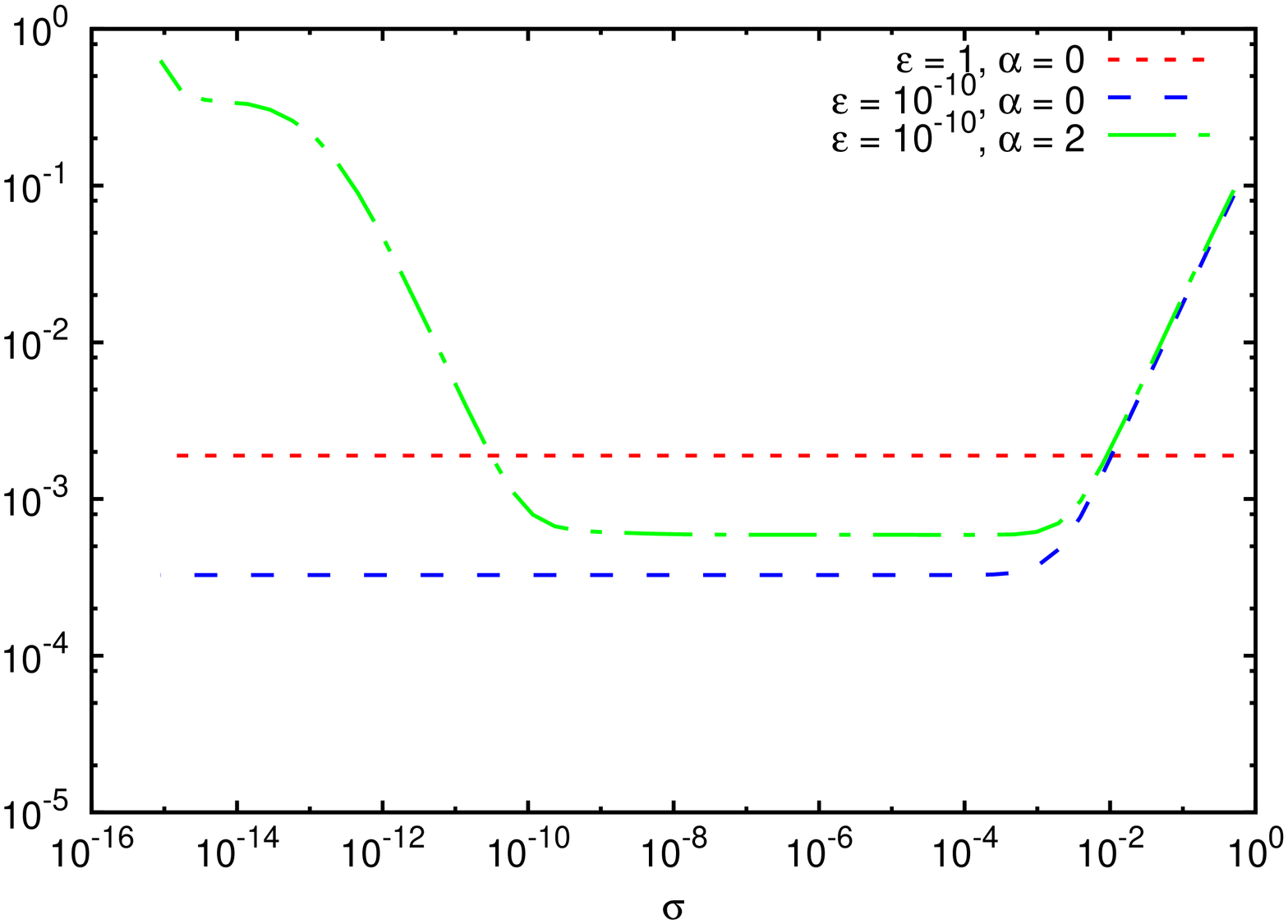}}
  
  \caption{ Absolute error $||u^\eps-u_{h}^{\eps,\sigma}||_{L^2}$ (on the left) and $||u^\eps-u_{h}^{\eps,\sigma}||_{H^1}$ (on the right) with respect to the exact solution $u^\eps$, as a function of $\sigma $, for
    $h=0.01$ and three regimes : no anisotropy, strong and aligned
    anisotropy as well as strong anisotropy with variable direction.}
  \label{fig:sigma_error}
\end{figure}

Having tested the $\sigma$-dependence of the error $||u^\eps-u_{h}^{\eps,\sigma}||$ for fixed $h>0$, we are now interested in how these curves remodel for different $h$-meshes. The $\sigma$-convergence is hence compared for different mesh sizes in
the most difficult setting, that is to say when a strong anisotropy
with variable direction is present in the system ($\varepsilon =10^{-10}$,
$\alpha =2$). Numerical simulations were performed for
the mesh size ranging from $0.1$ to $0.003125$. Cumulative results are
presented on the Figure \ref{fig:sigma_error2}. The plateau for which
the accuracy of the scheme does not depend on the stabilization
parameter is clearly dependent on the mesh size. As a consequence, the value of $\sigma $ should be clearly made mesh
dependent. We observe that in the case of $\mathbb Q_2$ finite
elements the upper and lower bounds for the optimal value scale like
$h^3$ and $h^4$ for the $L^2$-error, while for the $H^1$-error the
respective scaling is approximately $h^2$ and $h^6$. It is therefore
reasonable to put $\sigma = h^3$ (or $\sigma =h^2$ if one is
interested in the $H^1$-precision only). Note that this scaling
depends on the finite element method used. In general, if a
$\mathbb P_k$ (or $\mathbb Q_k$) method is used, the optimal choice of
$\sigma$ is $h^{k+1}$, which ensures optimal $h$-convergence of the
method in the $L^2$-norm.

\def\xxxa{0.49\textwidth}
\begin{figure}
  \centering
  \subfigure[$L^2$ error]
  {\includegraphics[angle=-90,width=\xxxa]{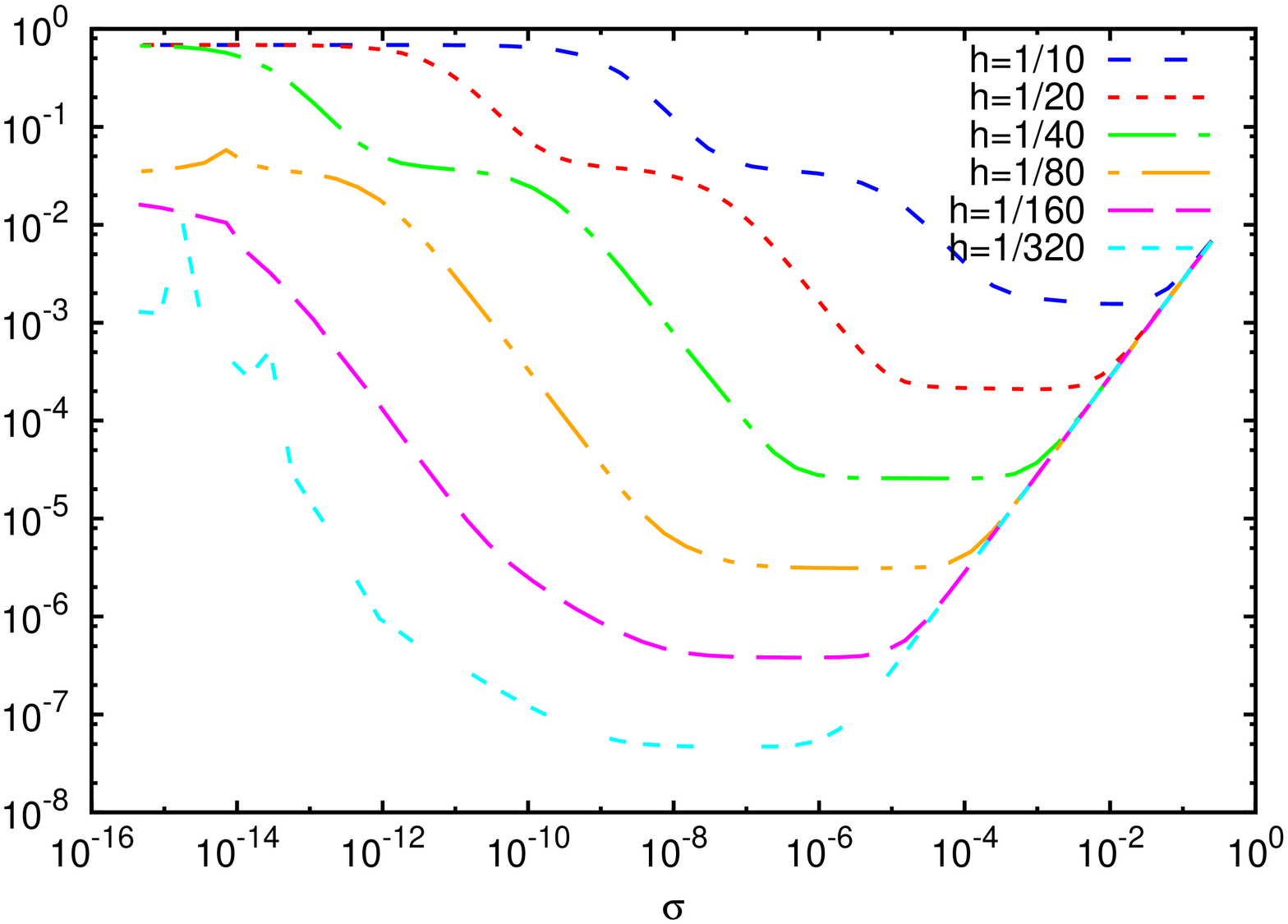}}
  \subfigure[$H^1$ error]
  {\includegraphics[angle=-90,width=\xxxa]{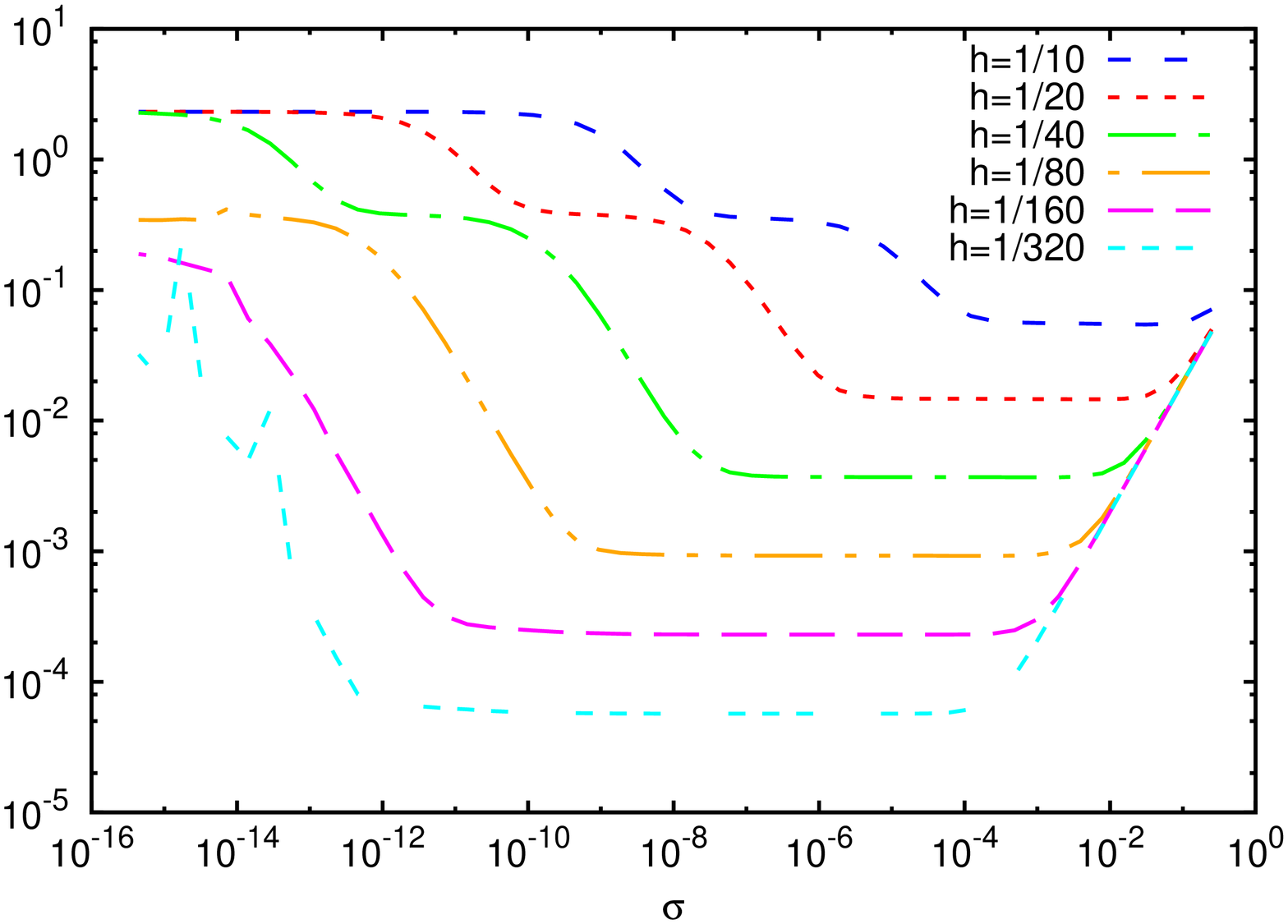}}
  
  \caption{Absolute error $||u^\eps-u_{h}^{\eps,\sigma}||_{L^2}$ (on the left) and $||u^\eps-u_{h}^{\eps,\sigma}||_{H^1}$  (on the right) as a function of $\sigma $, for
    different values of $h$ and for $\varepsilon =10^{-10}$ and
    $\alpha =2$.}
  \label{fig:sigma_error2}
\end{figure}

\subsection{$h$-convergence ($\eps \ge 0$ fixed, $\sigma=h^3$, $h \rightarrow 0$)}

Let us now turn our attention to the $h$-convergence of both
Asymptotic Preserving reformulations \eqref{Pa_bis} and \eqref{APstab}. Since the AP-scheme with inflow
boundary conditions was studied in a previous work \cite{DLNN} we are
mainly interested in the behaviour of the scheme with
stabilization. As in the previous subsection, numerical tests are
performed in three regimes : an isotropic one ($\varepsilon =1$ and
$\alpha =0$) and two anisotropic regimes ($\varepsilon =10^{-10}$ and
$\alpha =0$ or $\alpha =2$ ). The stabilization coefficient is set to
$\sigma = h^3$ as a consequence of the last subsection. The convergence rate in the $L^2$- and $H^1$-norms is
presented on Figure \ref{fig:h_conv_stab}. As expected the optimal
convergence rate (of a ${\mathbb Q}_2$-FEM) is found in both norms. Next we compare the results
with the convergence of the AP-scheme with inflow boundary
conditions in Tables \ref{tab:h_conv_L2} and
\ref{tab:h_conv_H1}. Note that in the case of no anisotropy or
anisotropy aligned with the coordinate system ($\alpha =0$), both
schemes give quasi exactly the same precision for both $L^2$ and $H^1$-norms.  In the last
regime the stabilized scheme is slightly less accurate compared to the
$(AP_{in})^{\eps }$ scheme. A small loss of the
convergence rate of the stabilized scheme is observed for the smallest
mesh size in both norms.
\begin{table}
  \centering
  \begin{tabular}{|c||c|c||c|c||c|c|}
    \hline
    \multirow{2}{*}{$h$} &
    \multicolumn{2}{|c||}{\rule{0pt}{2.5ex}$\varepsilon=1$, $\alpha=0$} &
    \multicolumn{2}{|c||}{$\varepsilon=10^{-10}$, $\alpha=0$} &
    \multicolumn{2}{|c|}{$\varepsilon=10^{-10}$, $\alpha=2$} \\
    \cline{2-7} 
    \rule{0pt}{2.5ex}
    &$(AP_{in})^{\eps }$&$(AP_{\mathcal{S}})^{\eps,\sigma }$
    &$(AP_{in})^{\eps }$&$(AP_{\mathcal{S}})^{\eps,\sigma }$
    &$(AP_{in})^{\eps }$&$(AP_{\mathcal{S}})^{\eps,\sigma }$\\
    \hline
    \rule{0pt}{2.5ex}$0.1$ &       $5.39\times 10^{-3}$ & $5.39\times 10^{-3}$ & $1.19\times 10^{-3}$ & $1.19\times 10^{-3}$ & $2.81\times 10^{-3}$ & $2.18\times 10^{-3}$ \\\hline
    \rule{0pt}{2.5ex}$0.05$ &      $6.97\times 10^{-4}$ & $6.97\times 10^{-4}$ & $1.49\times 10^{-4}$ & $1.49\times 10^{-4}$ & $3.16\times 10^{-4}$ & $2.87\times 10^{-4}$ \\\hline
    \rule{0pt}{2.5ex}$0.025$ &     $8.79\times 10^{-5}$ & $8.79\times 10^{-5}$ & $1.86\times 10^{-5}$ & $1.86\times 10^{-5}$ & $3.77\times 10^{-5}$ & $3.53\times 10^{-5}$ \\\hline
    \rule{0pt}{2.5ex}$0.0125$ &    $1.10\times 10^{-5}$ & $1.10\times 10^{-5}$ & $2.33\times 10^{-6}$ & $2.33\times 10^{-6}$ & $4.57\times 10^{-6}$ & $4.31\times 10^{-6}$ \\\hline
    \rule{0pt}{2.5ex}$0.00625$ &   $1.38\times 10^{-6}$ & $1.38\times 10^{-6}$ & $2.91\times 10^{-7}$ & $2.91\times 10^{-7}$ & $5.60\times 10^{-7}$ & $5.29\times 10^{-7}$ \\\hline
    \rule{0pt}{2.5ex}$0.003125$ &  $1.72\times 10^{-7}$ & $1.72\times 10^{-7}$ & $3.64\times 10^{-8}$ & $3.64\times 10^{-8}$ & $6.89\times 10^{-8}$ & $6.52\times 10^{-8}$ \\\hline
    \rule{0pt}{2.5ex}$0.0015625$ & $2.15\times 10^{-8}$ & $2.15\times 10^{-8}$ & $5.51\times 10^{-9}$ & $4.78\times 10^{-9}$ & $1.07\times 10^{-9}$ & $8.05\times 10^{-9}$ \\\hline
           
  \end{tabular}
  \caption{Comparison of the $L^2$ relative precision $||u^{\eps}-u^{\eps,\sigma}_h||_{L^2}/||u^{\eps,\sigma}_h||_{L^2}$ of both
    reformulations in both isotropic and anisotropic regimes for
    different mesh sizes and stabilization constant set to $\sigma=h^3$.}
  \label{tab:h_conv_L2}
\end{table}

\begin{table}
  \centering
  \begin{tabular}{|c||c|c||c|c||c|c|}
    \hline
    \multirow{2}{*}{$h$} &
    \multicolumn{2}{|c||}{\rule{0pt}{2.5ex}$\varepsilon=1$, $\alpha=0$} &
    \multicolumn{2}{|c||}{$\varepsilon=10^{-10}$, $\alpha=0$} &
    \multicolumn{2}{|c|}{$\varepsilon=10^{-10}$, $\alpha=2$} \\
    \cline{2-7} 
    \rule{0pt}{2.5ex}
    &$(AP_{in})^{\eps }$&$(AP_{\mathcal{S}})^{\eps,\sigma }$
    &$(AP_{in})^{\eps }$&$(AP_{\mathcal{S}})^{\eps,\sigma }$
    &$(AP_{in})^{\eps }$&$(AP_{\mathcal{S}})^{\eps,\sigma }$\\
    \hline
    \rule{0pt}{2.5ex}$0.1$ &       $4.48\times 10^{-2}$ & $4.48\times 10^{-2}$ & $1.46\times 10^{-2}$ & $1.46\times 10^{-2}$ & $2.44\times 10^{-2}$ & $2.33\times 10^{-2}$ \\\hline
    \rule{0pt}{2.5ex}$0.05$ &      $1.13\times 10^{-2}$ & $1.13\times 10^{-2}$ & $3.67\times 10^{-3}$ & $3.67\times 10^{-3}$ & $6.34\times 10^{-3}$ & $6.12\times 10^{-3}$ \\\hline
    \rule{0pt}{2.5ex}$0.025$ &     $2.84\times 10^{-3}$ & $2.84\times 10^{-3}$ & $9.19\times 10^{-4}$ & $9.19\times 10^{-4}$ & $1.60\times 10^{-3}$ & $1.54\times 10^{-3}$ \\\hline
    \rule{0pt}{2.5ex}$0.0125$ &    $7.11\times 10^{-4}$ & $7.11\times 10^{-4}$ & $2.30\times 10^{-4}$ & $2.30\times 10^{-4}$ & $3.99\times 10^{-4}$ & $3.83\times 10^{-4}$ \\\hline
    \rule{0pt}{2.5ex}$0.00625$ &   $1.78\times 10^{-4}$ & $1.78\times 10^{-4}$ & $5.75\times 10^{-5}$ & $5.75\times 10^{-5}$ & $9.93\times 10^{-5}$ & $9.53\times 10^{-5}$ \\\hline
    \rule{0pt}{2.5ex}$0.003125$ &  $4.45\times 10^{-5}$ & $4.45\times 10^{-5}$ & $1.44\times 10^{-5}$ & $1.44\times 10^{-5}$ & $2.46\times 10^{-5}$ & $2.37\times 10^{-5}$ \\\hline
    \rule{0pt}{2.5ex}$0.0015625$ & $1.11\times 10^{-5}$ & $1.11\times 10^{-5}$ & $3.76\times 10^{-6}$ & $3.76\times 10^{-6}$ & $6.08\times 10^{-6}$ & $5.87\times 10^{-6}$ \\\hline
           
  \end{tabular}
  \caption{Comparison of the $H^1$ relative precision $||u^{\eps}-u^{\eps,\sigma}_h||_{H^1}/||u^{\eps,\sigma}_h||_{H^1}$ of both
    reformulations in both isotropic and anisotropic regimes for
    different mesh sizes and stabilization constant set to $\sigma=h^3$.} 
  \label{tab:h_conv_H1}
\end{table}

\def\xxxa{0.49\textwidth}
\begin{figure}
  \centering
  \subfigure[$L^2$ error]
  {\includegraphics[angle=-90,width=\xxxa]{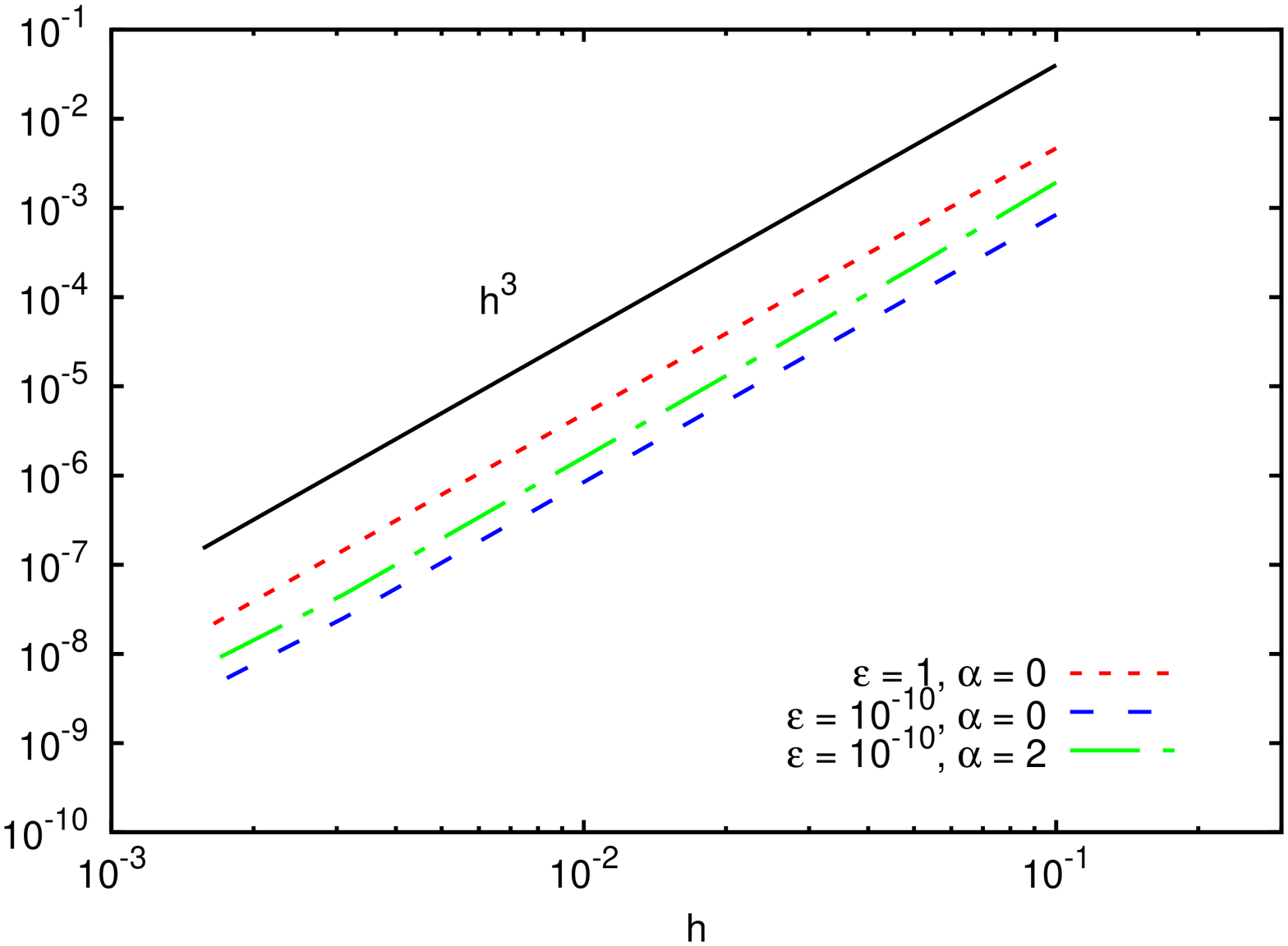}}
  \subfigure[$H^1$ error]
  {\includegraphics[angle=-90,width=\xxxa]{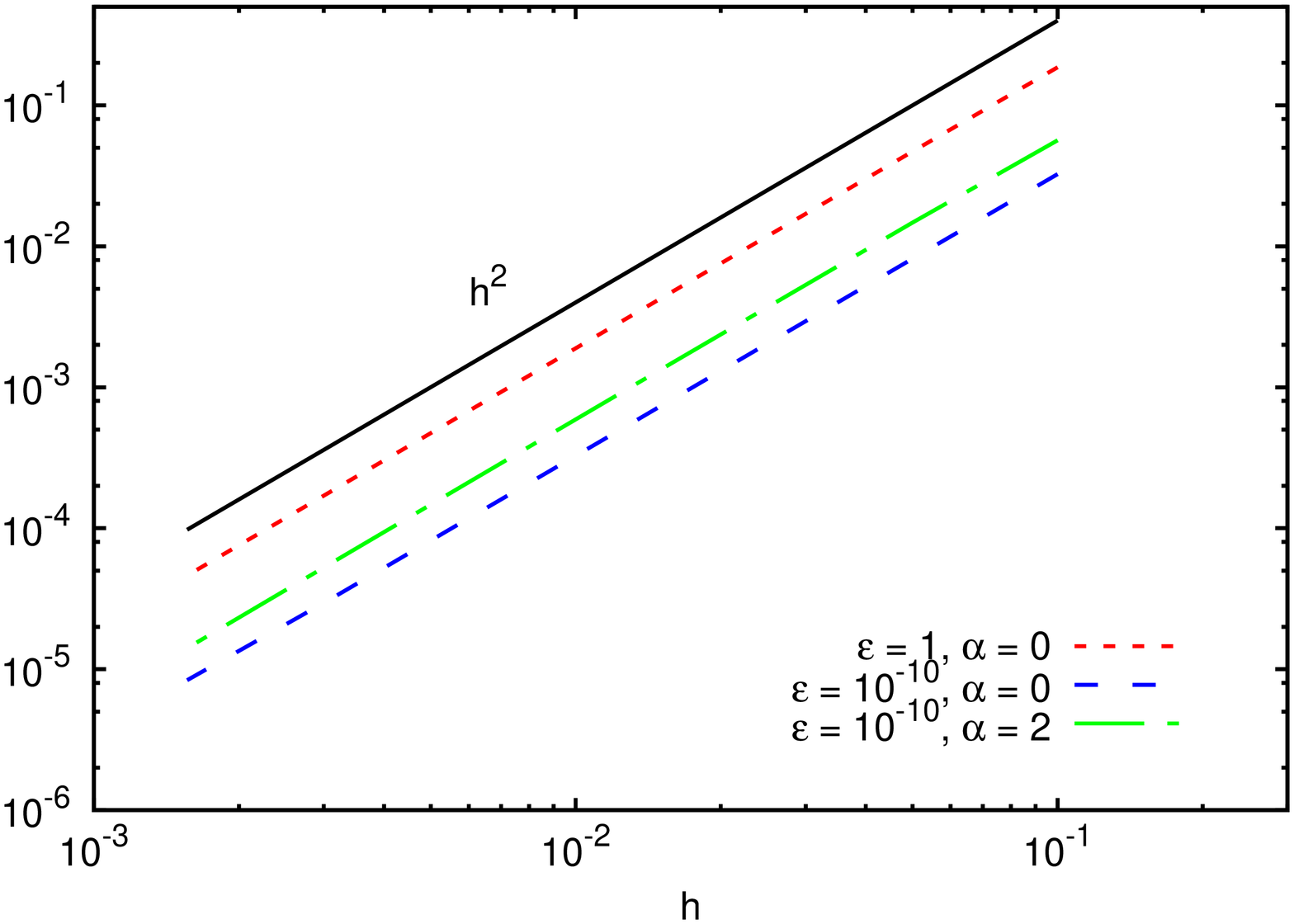}}
  
  \caption{Absolute error $||u^\eps-u_{h}^{\eps,\sigma}||_{L^2}$  (on the left) and $||u^\eps-u_{h}^{\eps,\sigma}||_{H^1}$  (on the right) as a function of $h$ and fixed $\sigma=h^3$. One
    isotropic regime ($\varepsilon =1$, $\alpha =0$) and two
    anisotropic ones: $\varepsilon =10^{-10}$ and $\alpha =0$ or
    $\alpha =2$ are investigated. The optimal convergence rate is found.}
  \label{fig:h_conv_stab}
\end{figure}

\subsection{AP-property ($h>0$ fixed, $\sigma=h^3$, $\eps \rightarrow 0$)}

Next, we test if both schemes are indeed Asymptotic Preserving as $\eps \rightarrow 0$. The
mesh size is fixed to $h=0.01$, $\sigma$ is set to $\sigma=h^3$ and numerical simulations are performed
for a variable anisotropy direction ($\alpha =2$) with an  anisotropy
strength $\eps$ varying from $10^{-20}$ to $10$. Both schemes exhibit the
desired property, as shown in Figure
\ref{fig:eps_conv}. In particular, the absolute error for both reformulations
and in both norms is independent of $\varepsilon $ (for
$\varepsilon < 0.1$). The error curves are practically indistinguishable. For large $\eps$-values, the errors are increasing due to the fact that the here presented schemes are designed to cope with $\eps \ll 1$ singularities.

\def\xxxa{0.49\textwidth}
\begin{figure}
  \centering
  \subfigure[$L^2$ error]
  {\includegraphics[angle=-90,width=\xxxa]{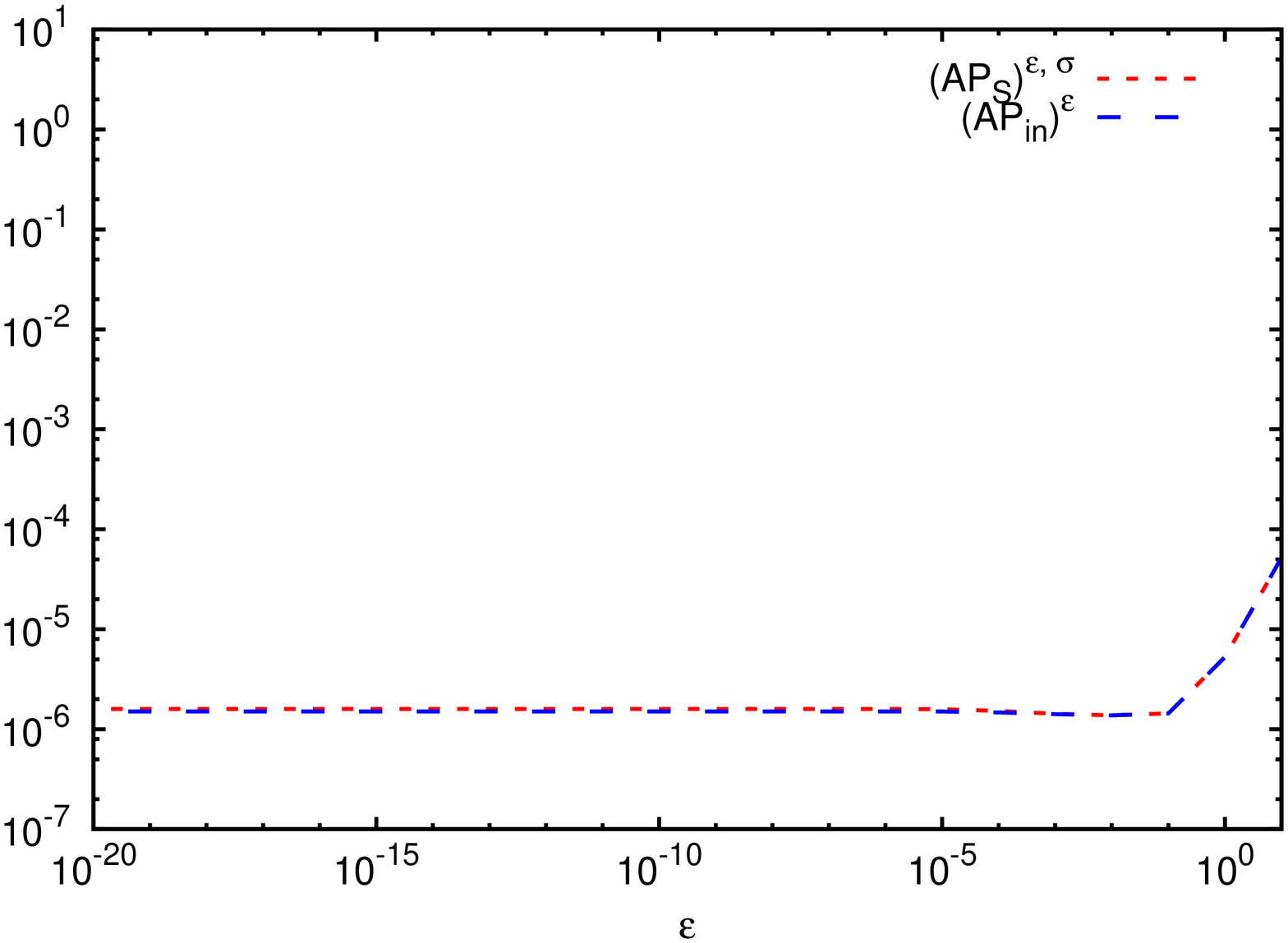}}
  \subfigure[$H^1$ error]
  {\includegraphics[angle=-90,width=\xxxa]{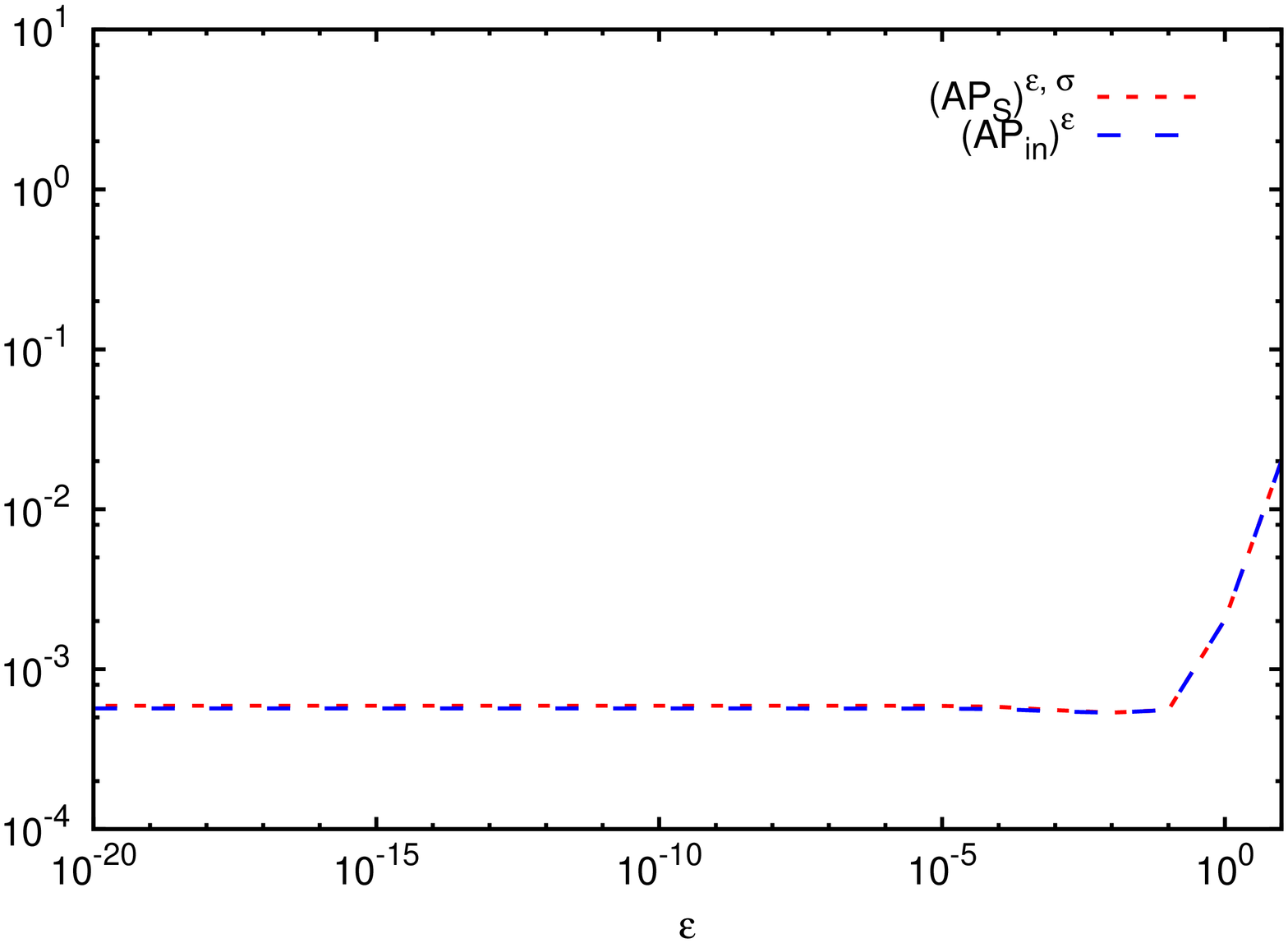}}
  
  \caption{Absolute error $||u_{ex}^\eps-u_{num}^{\eps,\sigma}||_{L^2}$ (on the left) and $||u_{ex}^\eps-u_{num}^{\eps,\sigma}||_{H^1}$  (on the right)  as a function
    of $\varepsilon $ for an anisotropy not aligned with the coordinate
    system ($\alpha =0$) and the mesh size $h=0.01$, $\sigma=h^3$. The error curves
    are superposed, both schemes show similar accuracy independently
    of $\varepsilon $.}
  \label{fig:eps_conv}
\end{figure}

\subsection{Matrix conditioning}

Finally, let us now turn our attention to the conditioning of the matrices
associated with the numerical resolution of both schemes $(AP_{in})^{\eps}$ resp. $(AP_{\mathcal{S}})^{\eps,\sigma }$. The strong anisotropy case with variable direction ($\alpha=2$) is considered for different mesh sizes $h>0$. The stabilization constant $\sigma $
is set to $h^3$ in the $(AP_S)^{\varepsilon ,\sigma }$ reformulation and the anisotropy strength $\varepsilon $ is set to $10^{-10}$. Sparse matrices were assembled in every
case and the condition number was estimated using the matlab function
\verb|condest()| returning the estimate of $cond_1$. The results are
displayed on Figure \ref{fig:cond_h}. As expected, the conditioning
scales as $1/h^4$ for the inflow reformulation and as
$1/\sigma h^2 = 1/h^5$ for the stabilized method. The first method
results in better conditioned matrices in this setting. However, if
one is interested mainly in the $H^1$ precision the stabilization
constant $\sigma $ could be set to $h^2$ resulting in a conditioning
proportional to $1/h^4$ for the stabilized method, discretized with the
$\mathbb Q_2$ finite elements.

\def\xxxa{0.49\textwidth}
\begin{figure}
  \centering
  {\includegraphics[angle=-90,width=\xxxa]{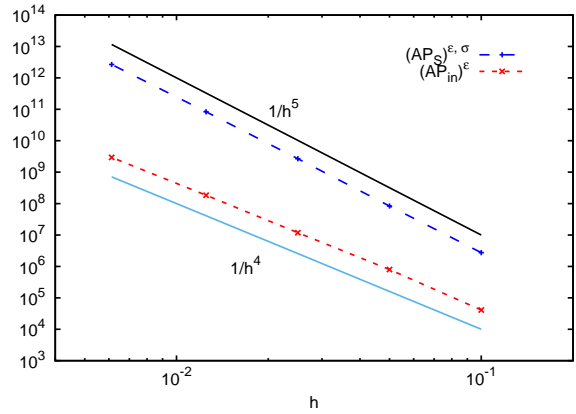}}
  
  \caption{Conditioning ($cond_1$) of the matrices associated with both AP
    schemes as a function of the mesh size for strong and nonaligned
    anisotropy ($\varepsilon =10^{-10}$, $\alpha =2$). The predicted
    scaling is found.}
  \label{fig:cond_h}
\end{figure}

\subsection{The case of $f\not\in H^1(\Omega )$}\label{sec:test_f_notH1}

Aim of this subsection is to investigate the error estimates in a case where the right hand side $f$ is less regular than supposed in the theoretical part of the last two sections.
All simulations in this section are preformed with a $\mathbb Q _1$
finite element method and the stabilization parameter in the
$(AP_{\mathcal S})^{\varepsilon ,\sigma }$-formulation is set to
$h^2$.  
In this case we have the $h$-estimates \eqref{Err_est} resp. \eqref{erestsig} with $k=1$ and we recall Remark \ref{Rem9} resp. \ref{Rem22}.
Let us now choose $u^0$ to be defined by
\begin{align}
  \begin{split}\label{eq:Jgfb}
    u^0 =&
    \left(
      \left(  y+  \alpha (y^2-y)\cos (\pi x)/\pi \right) ^2
      \ln \left(  y+  \alpha (y^2-y)\cos (\pi x)/\pi  \right) - 1.5
    \right)
    \\
    &+7.5
    \left( y+  \alpha (y^2-y)\cos (\pi x)/\pi \right)\,,
  \end{split}
\end{align}
so that the right hand side for the limit problem is a function that
belongs to $L^2$ and not to $H^1$. If $\alpha =0$ (the field is
aligned), then the right hand side of the limit problem equals to
$\ln y$.

We remind that in view of our
theoretical result, the $H^1$-norms of $q^\varepsilon $ and
$\xi^{\varepsilon, \sigma} $ are not guaranteed to be bounded if the
force term is not $H^1(\Omega )$. Nothing can be said on the convergence of the
numerical methods in this test case since the right hand side
(\ref{eq:Jgfb}) is not in $H^2(\Omega )$. We consider two anisotropic regimes
($\varepsilon =10^{-10}$): with anisotropy direction aligned with the
coordinate system ($\alpha =0$) resp. with variable direction ($\alpha =2$).

Numerical simulations show that the $H^1$-norms of $q_h^\varepsilon $ and
$\xi_h^{\varepsilon, \sigma} $ grow as $h$ approaches $0$ for the variable
anisotropy direction.  This seems to confirm our expectations.  On the
other hand, the $H^1$-norm remains constant when the anisotropy is
aligned with the coordinate system. The $L^2$-norm seems to be bounded
regardless of the method for both studied regimes. The results are
displayed on Figure \ref{fig:notL2_qnorm}. To our surprise, optimal
convergence rate of $u_h^\varepsilon $ and
$u_h^{\sigma ,\varepsilon} $ is conserved in the tested $h$ range ---
see Figure \ref{fig:notL2_conv}.

\def\xxxa{0.49\textwidth}
\begin{figure}
  \centering
  \subfigure[$L^2$ norm]
  {\includegraphics[angle=-90,width=\xxxa]{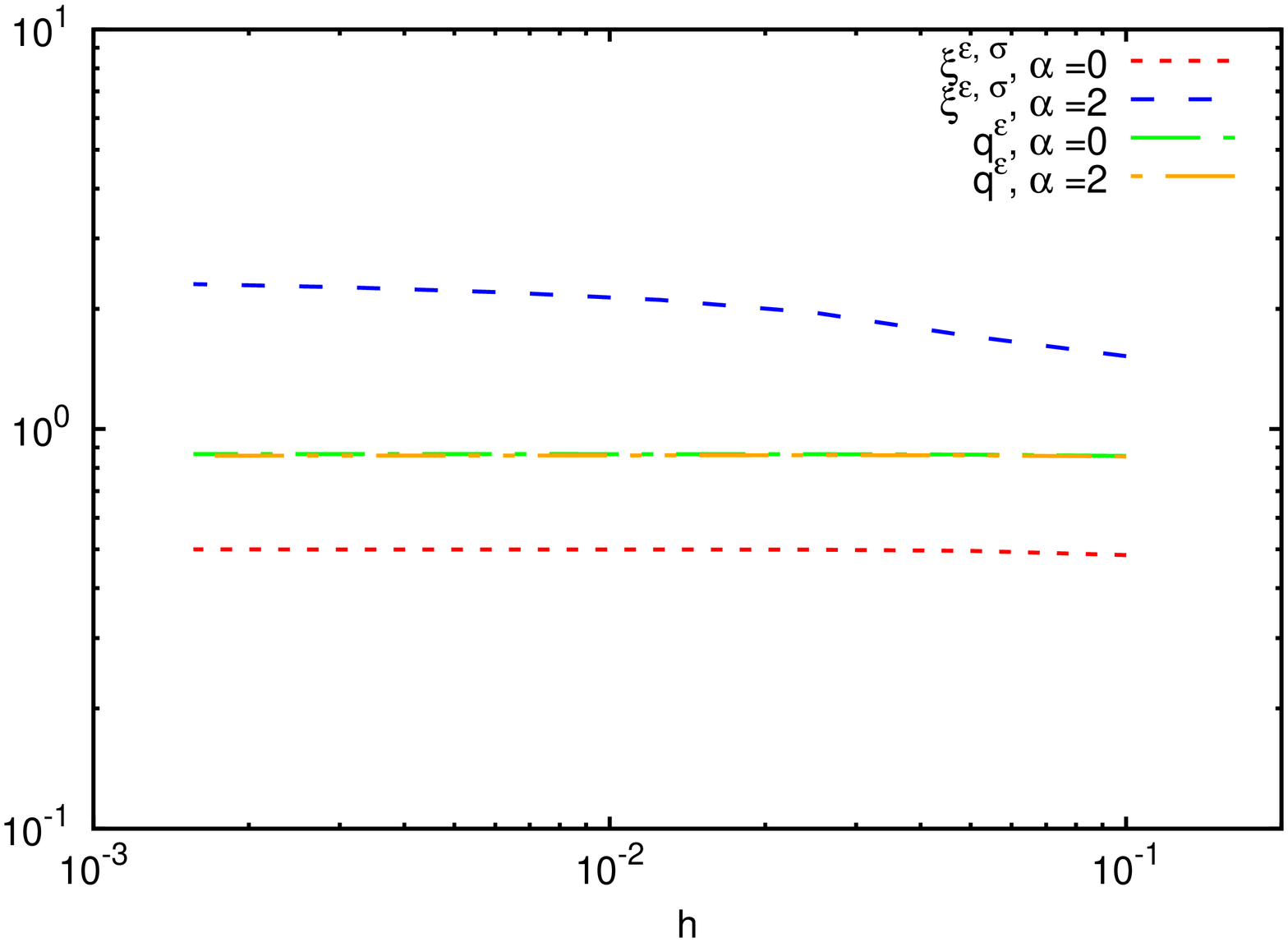}}
  \subfigure[$H^1$ norm]
  {\includegraphics[angle=-90,width=\xxxa]{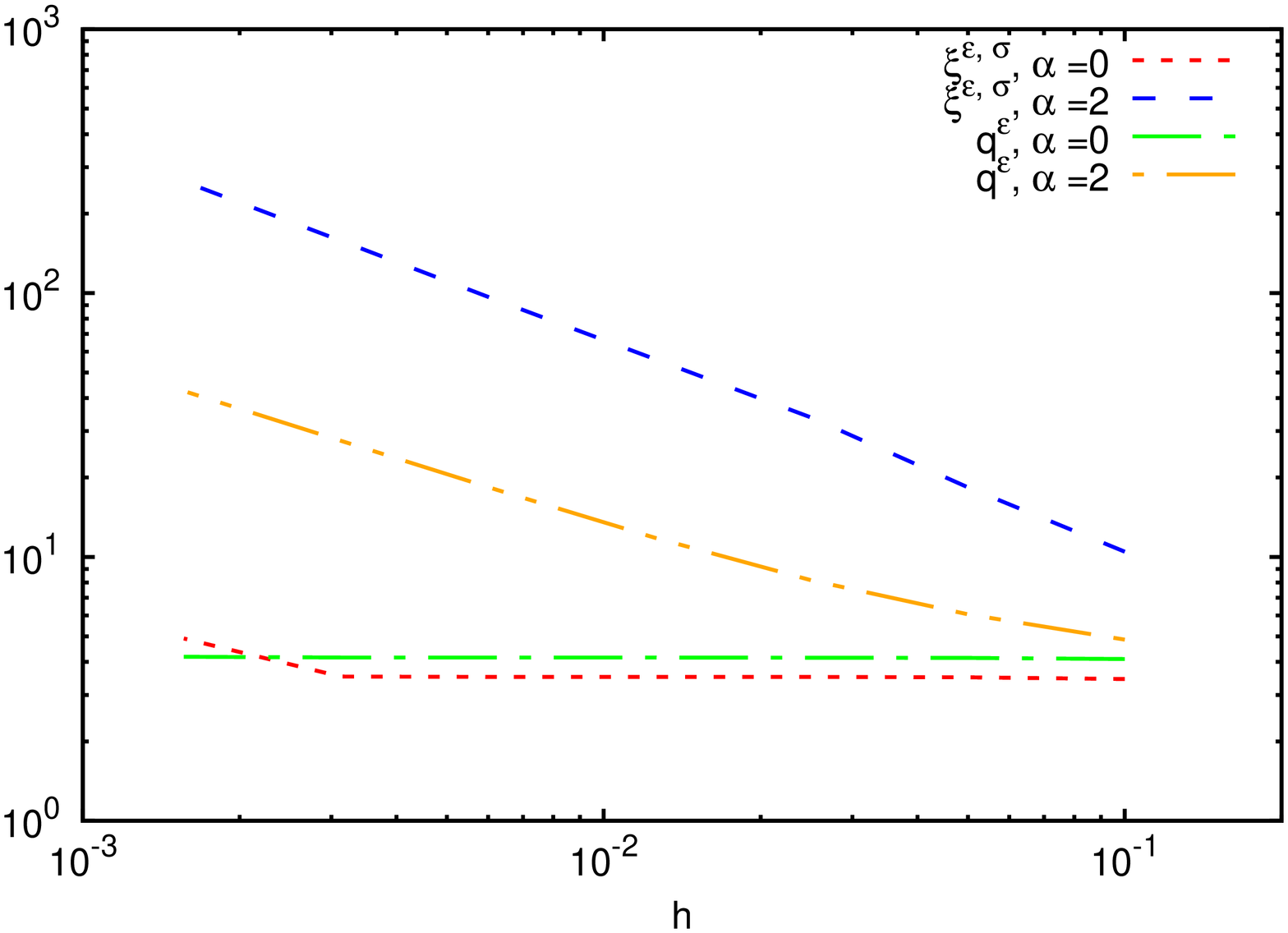}}
  
  \caption{$L^2$ (on the left) and $H^1$ (on the right) norms of
    $\xi^{\varepsilon, \sigma }$ and $q^{\varepsilon}$ as a function
    of the mesh size $h>0$, for $\varepsilon =10^{-10}$, $\sigma=h^2$  and $\alpha =0$ or
    $\alpha =2$. The $H^1$-norms of both auxiliary variables increase
    with decreasing mesh size for variable direction of
    anisotropy. The $L^2$-norms seem to be bounded. }
  \label{fig:notL2_qnorm}
\end{figure}
\def\xxxa{0.49\textwidth}
\begin{figure}
  \centering
  \subfigure[$L^2$ norm]
  {\includegraphics[angle=-90,width=\xxxa]{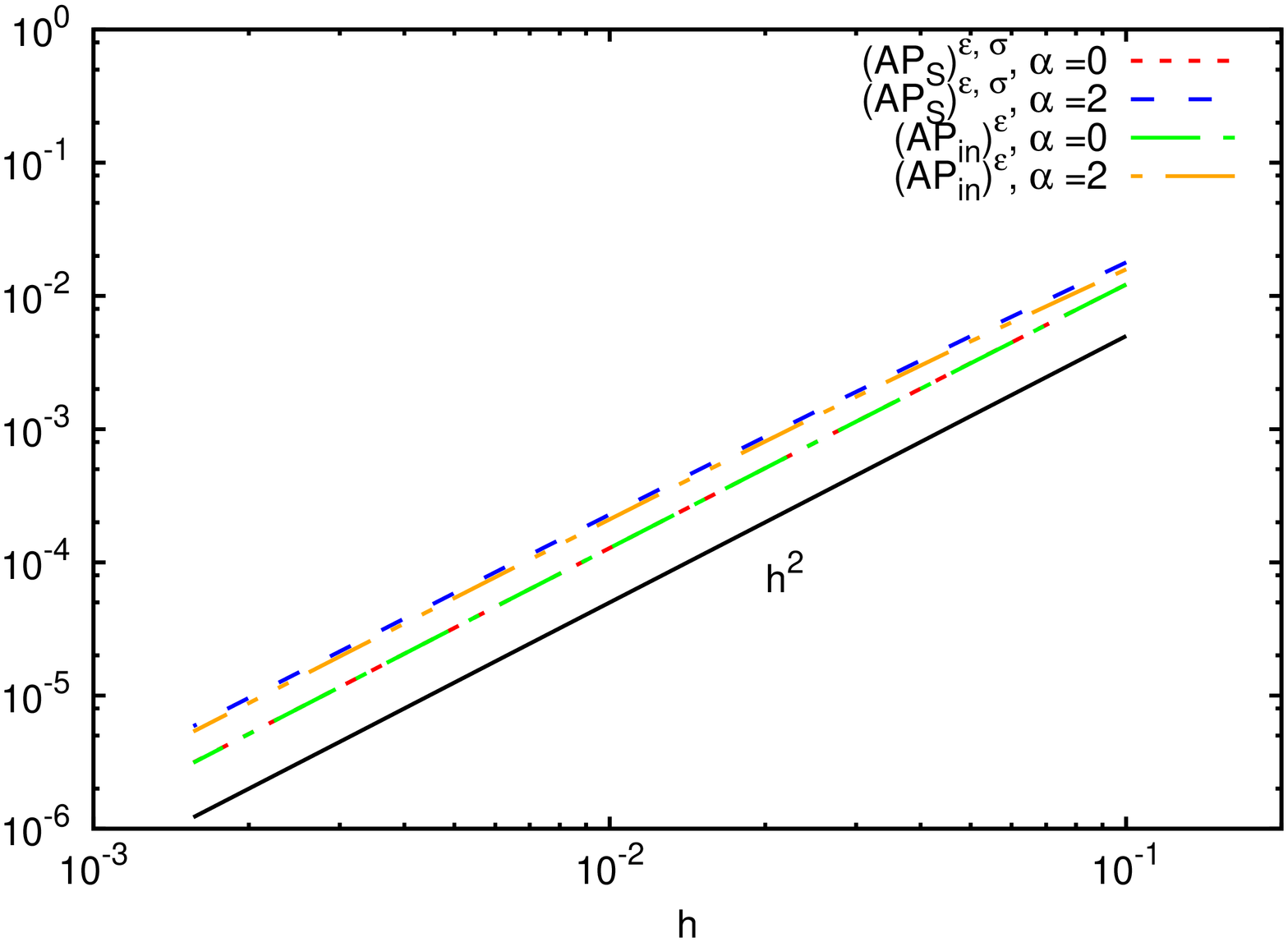}}
  \subfigure[$H^1$ norm]
  {\includegraphics[angle=-90,width=\xxxa]{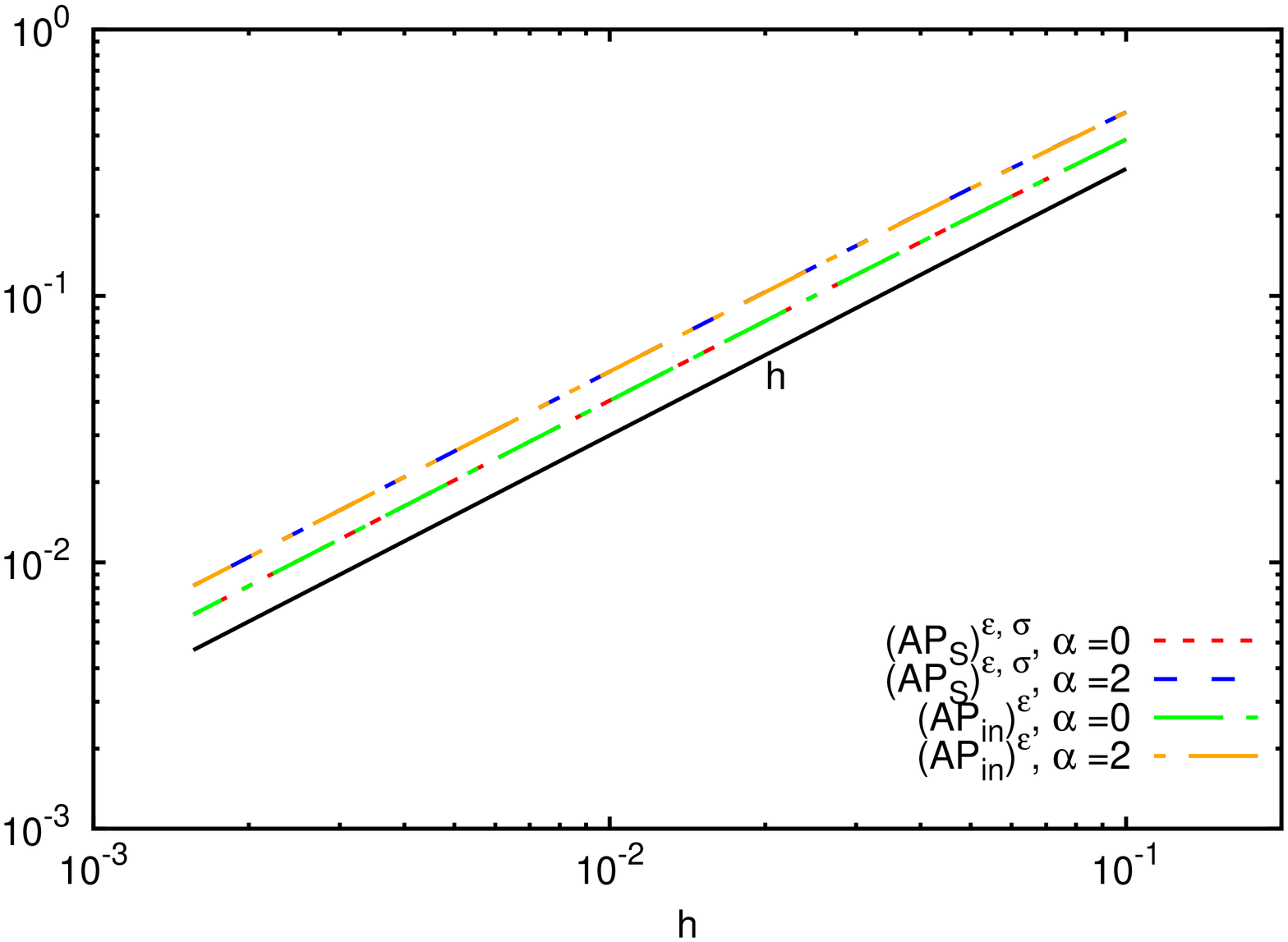}}
  
  \caption{Absolute error $||u^\eps-u_{h}^{\eps,\sigma}||_{L^2}$ (on the left) and $||u^\eps-u_{h}^{\eps,\sigma}||_{H^1}$  (on the right)  as a function
    of the mesh size $h>0$, for $\varepsilon =10^{-10}$, $\sigma=h^2$ and $\alpha =0$ or
    $\alpha =2$. Optimal convergence rate is observed for both
    methods and both anisotropy configurations.}
  \label{fig:notL2_conv}
\end{figure}

\section{Conclusions}

\label{SEC} %%%%%%%%%%%%%%%%%%%%%%%
A detailed numerical analysis of some asymptotic-preserving numerical
schemes, designed to cope with highly anisotropic elliptic problems, was
carried out in the present work. In particular, we have shown rigorously that in
the limit regimes where traditional schemes become inadequate, AP-schemes
are perfectly able to capture the macroscopic behavior of the solution.
Convergence results for the schemes were proven, with an accuracy and
stability which are shown to be $\eps$-independent, $\eps$ being the
perturbation parameter responsible for the stiffness of the problem. The
development of AP-schemes is based on asymptotic arguments and permit hence
to create a link between the various scales in the considered problem, while
the numerical parameters remain independent on the stiffness parameter. 
%%%%%%%%%%%%%%%%%%%%%%%%%%%%%%%%%%%%%%%%%%%%

\FloatBarrier
\appendix

\section{The regularity of the solution in the case of a simple geometry}

Consider the case of $\Omega =(0,\pi )^{2}$ and the field $b$ looking
upwards, {\it i.e.} $b=e_y$. Moreover let $A_{||}=1$ and $A_\perp = Id$. We want to explore in this Appendix the regularity of the solution $(u^{\eps
,\sigma },\xi^{\eps ,\sigma })$ to (\ref{APstab}) when $f$ belongs to $H^s(\Omega)$ and considering an aligned geometry case.\\
To start, let us first remark that the functions $\{ \sqrt{2 / \pi}\, \sin kx \}_{k \ge 1}$ as well as $\{  \sqrt{2 / \pi}\,\cos lx \}_{l \ge 0}$ form an orthogonal basis in $L^2(0,\pi)$ \cite{BRE}, such that each $f \in L^2(\Omega)$ can now be written under the form
\begin{equation*}
f(x,y)=\sum_{k=1}^{\infty }\sum_{l=0}^{\infty }f_{kl}\sin kx\cos ly\,, \quad \{f_{kl} \}_{k,l \in \NN} \subset l^2\,,
\end{equation*}%
which implies immediately that
\begin{equation*}
u^{\eps ,\sigma }(x,y)=\sum_{k=1}^{\infty }\sum_{l=0}^{\infty }\frac{1}{%
k^{2}+l^{2}+\frac{(1-\eps )l^{4}}{\eps l^{2}+\sigma }}f_{kl}\sin kx\cos ly\,,
\end{equation*}%
\begin{equation*}
\xi^{\eps ,\sigma }(x,y)=\sum_{k=1}^{\infty }\sum_{l=1}^{\infty }\frac{l^{2}}{(\eps %
l^{2}+\sigma )(k^{2}+l^{2})+(1-\eps )l^{4}}f_{kl}\sin kx\cos ly.
\end{equation*}%
Now, if $f\in H^{s}(\Omega)$, Parseval's equality permits to show that
\begin{equation*}
|f|_{H^{s}}^{2}\sim \sum_{k=1}^{\infty }\sum_{l=0}^{\infty
}(k^{2}+l^{2})^{s}f_{kl}^{2}\,,
\end{equation*}%
so that%
\begin{equation*}
|u^{\eps ,\sigma }|_{H^{s+2}}^{2}\sim \sum_{k=1}^{\infty }\sum_{l=0}^{\infty
}\frac{(k^{2}+l^{2})^{s+2}}{\left( k^{2}+l^{2}+\frac{(1-\eps )l^{4}}{\eps %
l^{2}+\sigma }\right) ^{2}}f_{kl}^{2}\leq \sum_{k=1}^{\infty
}\sum_{l=0}^{\infty }(k^{2}+l^{2})^{s}f_{kl}^{2}\sim |f|_{H^{s}}^{2}\,,
\end{equation*}%
\begin{equation*}
|\xi^{\eps ,\sigma }|_{H^{s}}^{2}\sim \sum_{k=1}^{\infty }\sum_{l=1}^{\infty }%
\frac{(k^{2}+l^{2})^{s}l^{4}}{\left( (\eps l^{2}+\sigma )(k^{2}+l^{2})+(1-%
\eps )l^{4}\right) ^{2}}f_{kl}^{2}\leq \sum_{k=1}^{\infty
}\sum_{l=1}^{\infty }(k^{2}+l^{2})^{s}f_{kl}^{2}\sim |f|_{H^{s}}^{2}\,.
\end{equation*}%
Moreover, in the case $\sigma =0$ one has
\begin{equation*}
|\partial _{yy}u^{\eps }|_{H^{s}}^{2}\sim \sum_{k=1}^{\infty
}\sum_{l=1}^{\infty }\frac{(k^{2}+l^{2})^{s}l^{4}}{\left( k^{2}+\frac{l^{2}}{%
\eps }\right) ^{2}}f_{kl}^{2}\leq \eps ^{2}\sum_{k=1}^{\infty
}\sum_{l=0}^{\infty }(k^{2}+l^{2})^{s}f_{kl}^{2}\sim \eps ^{2}|f|_{H^{s}}^{2}\,.
\end{equation*}
In conclusion, if $f\in H^{s}(\Omega)$ then $u^{\eps ,\sigma }\in H^{s+2}(\Omega)$, $\xi^{\eps %
,\sigma }\in H^{s}(\Omega)$ and $\partial _{yy}u^{\eps }\in H^{s}(\Omega)$ and there is a
constant $C>0$ independent of $\eps $ and $\sigma $ such that%
\begin{equation*}
|u^{\eps ,\sigma }|_{H^{s+2}}\leq C|f|_{H^{s}},\quad |\xi^{\eps ,\sigma
}|_{H^{s}}\leq C|f|_{H^{s}}\quad \text{and\quad }|\partial
_{yy}u^{\eps }|_{H^{s}}\leq C\, \eps\,|f|_{H^{s}}\,.
\end{equation*}%
The same estimates hold true in the inflow case, problem (\ref{Pa}), {\it i.e.} when $q^{\eps }$ is
associated with zero boundary conditions on the inflow part.
Indeed, the link between $q^{\eps }$ and $\xi^{\eps,0}$  can be explicited as%
\begin{equation*}
q^{\eps }(x,y)=\xi^{\eps ,0}(x,y)-\xi^{\eps ,0}(x,0).
\end{equation*}%
Hence, it suffices to study the regularity of the trace function $\chi ^{\eps
}(x):=\xi^{\eps ,0}(x,0)$. We have%
\begin{equation*}
\chi ^{\eps }(x)=\sum_{k=1}^{\infty }\sum_{l=1}^{\infty }\frac{l^{2}}{\eps %
l^{2}(k^{2}+l^{2})+(1-\eps )l^{4}}f_{kl}\sin kx\,,
\end{equation*}%
implying
\begin{eqnarray*}
|\chi ^{\eps }|_{H^{s}(\Gamma_{in})}^{2} &= &\sum_{k=1}^{\infty }k^{2s}\left(
\sum_{l=1}^{\infty }\frac{l^{2}}{\eps l^{2}(k^{2}+l^{2})+(1-\eps )l^{4}}%
f_{kl}\right) ^{2}\leq \sum_{k=1}^{\infty }k^{2s}\left( \sum_{l=1}^{\infty }%
\frac{f_{kl}}{l^{2}}\right) ^{2} \\
&\leq &\sum_{k=1}^{\infty }k^{2s}\left( \sum_{l=1}^{\infty }\frac{1}{l^{4}}%
\right) \left( \sum_{l=1}^{\infty }f_{kl}^{2}\right) \leq
C\sum_{k=1}^{\infty }\sum_{l=1}^{\infty }(k^{2}+l^{2})^{s}f_{kl}^{2}\sim
|f|_{H^{s}}^{2}\,.
\end{eqnarray*}%
We conclude thus $\chi ^{\eps }\in H^{s}(\Gamma_{in})$ so that $q^{\eps }\in H^{s}(\Omega) $ with
the $\eps -$independent estimate $|q^{\eps
}|_{H^{s}}\leq C|f|_{H^{s}}$.

\section{On the discrete inf-sup condition}

As mentioned earlier, the numerical analysis in this paper would be more
convenient, if the discrete inf-sup condition (\ref{infsupfake}) were true, i.e. 
\begin{equation}
\inf_{q_{h}\in {L}_{h}}\sup_{v_{h}\in {V}_{h}}\frac{a_{\paral}(q_{h},v_{h})}{%
|q_{h}|_*|v_{h}|_{\mathcal{V}}}\geq \alpha\,,  \label{estalp1}
\end{equation}%
with a mesh independent $\alpha >0$. In more explicit form, this means%
\begin{equation}
\forall q_{h}\in {L}_{h}:\sup_{v_{h}\in {V}_{h}}\frac{a_{\paral}(q_{h},v_{h})}{%
|v_{h}|_{\mathcal{V}}}\geq \alpha \, \sup_{v\in\mathcal{V}}\frac{a_{\paral}(q_{h},v)}{|v|_{\mathcal{V}}}\,.
\label{defuh0}
\end{equation}

We show first that (\ref{estalp1}) holds true in a simple aligned geometry. Secondly, we provide a numerical study of a non-aligned case where (\ref{estalp1}) turns out to be false.

\subsection{The case of the aligned geometry}

Assume $\Omega = (0, L_x) \times (0, L_y)$ and $b = e_2$. Choose 
$V_h$ as the finite element space $Q_k$ ($k \geq 1$) on
a rectangular grid $\mathcal{T}_h$ aligned with the coordinate axes. More
precisely, we choose some node points $0 = x_0 < x_1 < \cdots < x_{N_x} =
L_x$ on $[0, L_x]$ and $0 = y_0 < y_1 < \cdots < y_{N_y} = L_y$ on $[0,
L_y]$ with all the steps of order $h$ and introduce $\mathcal{T}_h$ as the collection of rectangles $[ x_i,
x_{i + 1}] \times [ y_j, y_{j + 1}]$ that constitutes a partition of
$\Omega$. Moreover, $\mathcal{E}_h$ shall denote the set of all the edges of
the mesh $\mathcal{T}_h$ with the exception of those lying on $\Gamma_D$.

Our strategy to prove (\ref{estalp1}) is to use Verf{\"u}rth's trick  \cite{verfurth84} by first
establishing the inf-sup conditions with respect to an auxiliary mesh
dependent norm and then going to the original norm with the aid of Cl{\'e}ment
interpolation. Thus, we want to prove first that for all $q_h \in L_h$ there
exists $v_h \in V_h$ such that
\begin{equation}
  N_h ( q_h) : = \left( \frac{1}{h}  \sum_{E \in \mathcal{E}_h} \int_E |
  [\partial_y q_h] |^2 ds + \sum_{K \in \mathcal{T}_h} \int_K | \partial_{yy}
  q_h |^2 dx \right)^{\frac{1}{2}} \leq C \frac{a_{||} (q_h, v_h)}{|| v_h
  ||_{L^2}} \label{qhh1D}
\end{equation}
where $[\partial_y q_h]$ denotes the jump of $\partial_y q_h$ across an edge
$E \in \mathcal{E}_h$ if the edge is internal, and $[\partial_y
q_h]$=$\partial_y q_h$ on an edge $E$ lying on the boundary $\Gamma$.

A convenient reformulation of this is: For all $q_h \in L_h$ there exist $v_h
\in V_h$ such that
\begin{equation}
  || v_h ||_{L^2}^2  \leq  C_1 N_h^2 ( q_h)
  \text{ and }
  a_{||} (q_h, v_h)  \geq  C_2 N_h^2 ( q_h). 
\label{neq1}
\end{equation}
The quantities above can be written in a more explicit manner as
\begin{eqnarray*}
  &  & \\
  N_h ^2 ( q_h) & = & \frac{1}{h}  \sum_{i = 0}^{N_y} \int_0^{L_x} |
  [\partial_y q_h] |^2 (y_i) dx + \sum_{i = 0}^{N_y} \int_0^{L_x} \int_{y_{i -
  1}}^{y_i} | \partial_{yy} q_h |^2 dydx \hspace{0.25em},\\
  || v_h ||_{L^2}^2 & = & \int_0^{L_x} \int_0^{L_y} v_h^2 dydx
  \hspace{0.25em},\\
  a_{||} (q_h, v_h) & = & \sum_{i = 0}^{N_y} \int_0^{L_x} [\partial_y q_h]
  (x, y_i) v_h (x, y_i) dx - \sum_{i = 1}^{N_y} \int_0^{L_x} \int_{y_{i -
  1}}^{y_i} \partial_{yy} q_h v_h dydx.
\end{eqnarray*}

The construction of $v_h$ is particularly easy in the case of bilinear finite
elements ($k = 1$): we can take $v_h \in V_h$ such that for all $i = 0,
\ldots, N_y$
\begin{equation}
  v_h (x, y_i) = \frac{1}{h}  [\partial_y q_h] (x, y_i) \hspace{0.25em} .
  \label{vhconst}
\end{equation}
Then the second inequality in (\ref{neq1}) becomes equality with $C_2 = 1$.
Moreover, one easily gets by a scaling argument
\begin{equation}
  \int_0^{L_y} v_h^2 (x, y) dy \leq C_1 h \sum_{i = 0}^{N_y} v_h^2 (x, y_i)
  \label{neqsca}
\end{equation}
for all $x \in  [0, L_x]$ which gives the first inequality in (\ref{neq1}).

We describe now a more complicated construction in the case $k \geq 2$. For a
given $q_h \in L_h$ we construct $v_h \in V_h$ as $v_h^{} = v_h^{(1)} +
v_h^{(2)}$ where $v_h^{(1)} \in V_h $ is defined for any $x \in [ 0, L_x]$ and
any $y \in$[$y_{i - 1}, y_i]$, $i = 1, \ldots, N_y$ by
\begin{eqnarray*}
  &  & v_{h}^{(1)} (x, y) = - \partial_{yy} q_h (x, y) \frac{(y -
  y_{i - 1})  (y_i - y)}{h^2}  \text{}\\
  &  & \hspace{0.25em} .
\end{eqnarray*}
and $v_h^{(2)} \in V_h$ is such that
\[ v_h^{(2)} (x, y_i) = \frac{1}{h}  [\partial_y q_h] (x, y_i) ,
   \quad i = 0, \ldots, N_x, \quad x \in [ 0, L_x] \]
and $v_h^{(2)} |_{y \in [ y_{i - 1}, y_i]} $ for any $x$ fixed is a
polynomial of degree $k$ that is orthogonal in $L^2 ( y_{i - 1}, y_i)$ to all
the polynomials of degree $\leqslant k - 2$. We get for this $v_h$
\[ \begin{array}{ll}
     a_{||} (q_h, v_h) & =
   \end{array} \frac{1}{h}  \sum_{i = 0}^{N_y} \int_0^{L_x} | [\partial_y q_h]
   |^2 (y_i) dx + \sum_{i = 0}^{N_y} \int_0^{L_x} \int_{y_{i - 1}}^{y_i} |
   \partial_{yy} q_h |^2  \frac{(y - y_{i - 1})  (y_i - y)}{h^2} dydx
   \hspace{0.25em} \]
This yields immediately the second inequality in (\ref{neq1}) \ with some $C_2
> 0$. In order to prove the first inequality in (\ref{neq1}) we employ again a
scaling inequality of type (\ref{neqsca}): for any $x \in [ 0, L_x]$
\[ \int_0^{L_y} v_h^2 dy \leq 2 \int_0^{L_y} \hspace{-0.25em} \hspace{-0.25em}
   |v_h^{(1)} |^2 dy + 2 \int_0^{L_y} \hspace{-0.25em} \hspace{-0.25em}
   |v_h^{(2)} |^2 dy \leq C_1  \left( \sum_{i = 1}^{N_y} \int_{y_{i -
   1}}^{y_i} \hspace{-0.25em} \hspace{-0.25em} | \partial_{yy} q_h |^2 dy +
   \frac{1}{h}  \sum_{i = 0}^{N_y} [\partial_y q_h]^2 (y_i) \right)
   \hspace{0.25em} . \]

Now, (\ref{qhh1D}) being established, \ take any $q_h \in L_h$. By the
definition of the norm $|\cdot|_*$, there exists $v \in V$ such
that $|v|_{\mathcal{V}} = |q_h |_{\ast}$ and $a_{||} (q_h, v) = |q_h
|_{\ast}^2$. Let $\tilde{v}_h \in V_h$ be the Cl{\'e}ment interpolant of $v$
such that \cite{ern}
\[ \|v - \tilde{v}_h \|_{L^2 (\Omega)} \leq Ch |v|_{\mathcal{V}},
   \quad | \tilde{v}_h |_{\mathcal{V}} \leq C |v|_{\mathcal{V}} ,
   \quad \left( \sum_{E \in
   \mathcal{E}_h} \|v - \tilde{v}_h \|^2_{L^2 (E)} \right)^{\frac{1}{2}} \leq
   C \sqrt{h} |v|_{\mathcal{V}} . \]
Observe that
\begin{eqnarray*}
  |q_h |_{\ast}^2 & = & a_{||} (q_h, v) = a_{||}  (q_h, v - \tilde{v}_h) +
  a_{||} (q_h, \tilde{v}_h)\\
  & \leq & a_{||}  (q_h, v - \tilde{v}_h) + | \tilde{v}_h |_{\mathcal{V}}
  \sup_{v_h \in V_h}  \frac{a_{||} (q_h, v_h)}{|v_h |_{\mathcal{V}}} \leq
  a_{||}  (q_h, v - \tilde{v}_h) + C |q_h |_{\ast} \sup_{v_h \in V_h} 
  \frac{a_{||} (q_h, v_h)}{|v_h |_{\mathcal{V}}} \hspace{0.25em} .
\end{eqnarray*}
Integrating by parts element by element in the first term of the last line
yields
\begin{eqnarray*}
  a_{||}  (q_h, v - \tilde{v}_h) & = & \sum_{K \in \mathcal{T}_h} \int_K
  \nabla_{||} q_h \cdot \nabla_{||}  (v - \tilde{v}_h) dx\\
  & = & \sum_{E \in \mathcal{E}_h} \int_E [n_{||} \cdot \nabla_{||} q_h]  (v
  - \tilde{v}_h) ds - \sum_{K \in \mathcal{T}_h} \int_K (\nabla_{||} \cdot
  \nabla_{||} q_h)  (v - \tilde{v}_h) dx\\
  & \leq & N_h ( q_h) \leq  Ch |v|_V \sup_{v_h \in V_h}  \frac{a_{||} (q_h, v_h)}{|| v_h
  ||_{L^2}} \leq C |v|_{\mathcal{V}} \sup_{v_h \in V_h}  \frac{a_{||} (q_h,
  v_h)}{|v_h |_{\mathcal{V}}} .
\end{eqnarray*}
We have used here (\ref{qhh1D}) and the standard inverse inequality. This
enables us to conclude
\[ |q_h |_{\ast}^2 \leq C |v|_{\mathcal{V}} \sup_{v_h \in V_h}  \frac{a_{||}
   (q_h, v_h)}{|v_h |_{\mathcal{V}}} + C |q_h |_{\ast} \sup_{v_h \in V_h} 
   \frac{a_{||} (q_h, v_h)}{|v_h |_{\mathcal{V}}} \leq C |q_h |_{\ast}
   \sup_{v_h \in V_h}  \frac{a_{||} (q_h, v_h)}{|v_h |_{\mathcal{V}}}
   \hspace{0.25em}, \]
since $q_h \in V_h$ and $|v|_{\mathcal{V}} = |q_h |_{\ast}$. The last
inequality gives the desired result (\ref{estalp1}).

%%%%%%%%%%%%%%%%%%%%%%%%%%%%

%%%%%%%%%%%%%%%%%%%%%%%%%%%%

\subsection{A numerical study in the case of a general geometry}

The aim of this section is to investigate the validity of \eqref{estalp1} or equivalently \eqref{defuh0} in a more general context by a series of numerical experiments. As in Appendix B.1, we shall assume that $\Omega =(0,L_{x})\times (0,L_{y})$, $%
b=e_{2}$, however this time the grid is no more aligned with the field lines of $b$. Indeed, we are using here a regular  grid made of triangles such that their hypotenuses are no longer aligned with $b$. Numerical simulations are performed with FreeFem++ \cite{Hecht}.

Let $V_h$ be the ${\mathcal P}_1$ finite element space on a mesh described above of size $h>0$.
Observe that the first supremum in \eqref{defuh0} is attained on $v_{h}^{\ast }\in V_{h}$ that
satisfies%
\begin{equation}
a(v_{h}^{\ast },w_{h})=a_{\paral}(q_{h},w_{h})\,,\quad \forall w_{h}\in V_{h}\,.
\label{estalp2}
\end{equation}%
In order to explore the second supremum in \eqref{defuh0}, we use the finer finite
element space $V_{h/2}^{f}$, constructed via ${\mathcal P}_2$ finite elements on mesh of size $h/2$, i.e. a two-times refinement of the mesh above. The goal in introducing this finer space 
$V_{h/2}^{f}$ is to approximate the infinite-dimensional space in \eqref{defuh0}.

Consider $v_{h/2}^{\ast f}\in V_{h/2}^{f}$ that
satisfies%
\begin{equation}
a(v_{h/2}^{\ast f},w_{h/2}^{f})=a_{\paral}(q_{h/2},w_{h/2}^{f})\,,\quad \forall
w_{h/2}^{f}\in V_{h/2}^{f}\,.
\end{equation}%
If (\ref{defuh0}) holds true, than we have 
\begin{equation*}
\forall q_{h}\in {L}_{h}:\frac{\left|v_{h}^{\ast }\right|_{\mathcal{V}}}{\left|v_{h/2}^{\ast f}\right|_{\mathcal{V}}}%
\geq \alpha\,.
\end{equation*}
Unfortunately, this is false as shown in the following numerical experiment. Let $L_x=L_y=1$ and let us choose on each
mesh of size $h=1/n$ the function $q_{h}\in V_h$ defined by its values at the mess nodes as
\begin{equation}\label{qexample}
q_h(x_i,y_j)=x_i\sin (\pi n\, y_j/2)\,,
\end{equation}%
where $x_i=ih$, $y_j=jh$, $i,j=0,\ldots,n$. 
Note that this function satisfies all the boundary conditions provided $n$
is even. In Fig. \ref{fig:noinfsup} we plot the quantity $\frac{|v_{h}^{\ast }|_{\mathcal{V}}%
}{|v_{h/2}^{\ast f}|_{\mathcal{V}}}$ computed for such a $q_{h}$ on a series of meshes
versus $h=\frac{1}{n}$.
\begin{figure}
\centerline{
\includegraphics[scale=0.3]{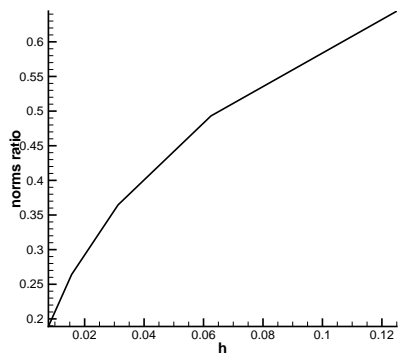} 
}
  \caption{The ratio $\frac{|v_{h}^{\ast }|_{\mathcal{V}}}{|v_{h}^{\ast f}|_{\mathcal{V}}}$ computed for $q_h$ given by (\ref{qexample}).}
  \label{fig:noinfsup}
\end{figure}
It shows clearly that the constant $\alpha $ in (\ref{estalp1}) is mesh
dependent, i.e. it tends to 0 (in general) when the mesh size tends to 0.

\section*{Acknowledgments}

This work has been supported by the ANR project MOONRISE (MOdels, Oscillations and NumeRIcal SchEmes, 2015-2019). This work has been carried out within the framework of the EUROfusion Consortium and has received funding from the Euratom research and training programme 2014-2018 under grant agreement No 633053. The views and opinions expressed herein do not necessarily reflect those of the European Commission.

\bibliographystyle{abbrv}
\bibliography{bib_aniso}

\end{document}